\documentclass[11pt]{amsart}
\usepackage{amsthm}
\usepackage{amsfonts}
\usepackage{amsmath}
\usepackage{amssymb}
\usepackage[all]{xy}
\usepackage{graphics}
\usepackage{tabularx, caption, multirow, hhline}
\usepackage{bbold}

\usepackage{color}
\usepackage{mathrsfs}

\usepackage{graphicx,amssymb,upref}

\let\<\langle
\let\>\rangle

\let\uml\"

\everymath{\displaystyle}
\CompileMatrices

\newtheorem{thm}{Theorem}[section]

\newtheorem{lemma}[thm]{Lemma}
\newtheorem{prop}[thm]{Proposition}
\newtheorem{cor}[thm]{Corollary}

\theoremstyle{definition}
\newtheorem{defn}[thm]{Definition}
\newtheorem{remark}[thm]{Remark}
\newtheorem{example}[thm]{Example}

\newcommand\ve{\varepsilon}

\newcommand\R{{\mathbb R}}
\newcommand\Q{{\mathcal Q}}
\newcommand\C{{\mathbb C}}
\renewcommand\H{{\mathbb H}}
\newcommand\N{{\mathbb N}}
\newcommand\K{{\mathbb K}}
\newcommand\B{{\mathbb B}}
\renewcommand\k{{\mathbb k}}
\newcommand\Z{{\mathbb Z}}
\renewcommand\1{{\mathbb 1}}

\newcommand\pr{^{\scriptstyle \R}}

\newcommand\sr{_{\scriptscriptstyle \R}}
\renewcommand\sc{_{\scriptscriptstyle \C}}
\newcommand\crt{^{\scriptscriptstyle {\it CRT}}}
\newcommand\scrt{_{\scriptscriptstyle {\it CRT}}}
\newcommand\ct{{\it CRT}}

\newcommand\id{\rm{id}}
\newcommand\im{\rm{image} \,}
\renewcommand\hom{\rm{Hom}}
\newcommand{\diag}{\rm{diag}}
\newcommand{\ev}{\rm{ev}}
\newcommand\trace{\rm{trace}}

\newcommand\tr{\mathrm{Tr}}
\newcommand\trt{\widetilde{\rm{Tr}}}
\newcommand\taut{\widetilde{\tau}}
\newcommand\Utd{\widetilde{U}}
\newcommand\Ktd{\widetilde{K}O}

\newcommand\TT{\rule{0pt}{2.6ex}} 
\newcommand\BB{\rule[-1.2ex]{0pt}{0pt}} 

\newcommand\calg{$C\sp *$-algebra}
\newcommand\ctalg{$C\sp {*, \tau}$-algebra}

\newcommand{\sm}[4]
	{ \left( \begin{smallmatrix} {#1} & {#2} \\ {#3} & {#4} \end{smallmatrix} \right)  }
\newcommand{\smv}[2]
	{ \left( \begin{smallmatrix} {#1}  \\ {#2} \end{smallmatrix} \right)  }
\newcommand{\smh}[2]
	{ \left( \begin{smallmatrix} {#1} & {#2}  \end{smallmatrix} \right)  }
\newcommand{\smiiii}[4]
	{\left( \begin{matrix} {#1} & {0} & {0} & {#2} \\  
				{0} & {0} & {0} & {0} 
				\\  {0} & {0} & {0} & {0}  \\  
			      {#3} & {0} & {0} & {#4}  
	\end{matrix} \right) }

\newcommand{\smd}[3]
	{ \left( \begin{smallmatrix} {#1} & 0 & 0 \\ 0 & {#2} & 0 \\ 0 & 0 & {#3} \end{smallmatrix} \right)  }

\renewcommand{\Im}{\operatorname{Im}}

\newcolumntype{C}[1]{>{\Centering}m{#1}}

\everymath{\displaystyle}
\CompileMatrices

\title[Unitary Picture of $K$-theory]{$K$-theory for real $C \sp *$-algebras via unitary elements with symmetries}

	\author{Jeffrey L. Boersema}	
	\address{Seattle University \\ Department of Mathematics \\
	Seattle, Washington 98133, USA}
	\email{boersema@seattle.edu}

	\author{Terry A. Loring}
	\address{University of New Mexico \\ Department of Mathematics and Statistics \\
	 Albuquerque, New Mexico  87131, USA}
	  \email{loring@math.unm.edu}

\thanks{
This work was partially supported by a grant from the Simons Foundation
(208723 to Loring).  
  }



\begin{document}

\begin{abstract}
We prove that all eight $KO$ groups for a real \calg~can be
constructed from homotopy classes of unitary matrices that respect
a variety of symmetries.  In this manifestation of the $KO$ groups,
all eight boudary maps in the 24-term exact sequences associated to an
ideal in a real \calg~can be computed as exponential or
index maps with formulas that are nearly identical to the complex case. \\
\vspace{-1cm}	  
\end{abstract}
	
\maketitle
\setcounter{tocdepth}{1}
\tableofcontents 

\section{Introduction}

In the common picture of $K$-theory for \calg s, the abelian groups $K_0(A)$ and $K_1(A)$ arise from projections and unitaries in 
$M_n(A)$, respectively. Because of Bott periodicity, we do not worry about independent descriptions 
of $K_i(A)$ for other integer values of $i$. In the case of real \calg s, 
the same pictures carry over to give us concrete descriptions of $KO_0(A)$ and 
$KO_1(A)$ in terms of projections and unitaries. The higher $K$-theory groups 
(for $i \neq 0,1$) can be defined using suspensions or using Clifford algebras.
While this reliance on suspensions allows the theoretical development of $K$-theory to proceed nicely, 
it leaves much to be desired in terms of being able to represent specific 
$K$-theory classes for purposes of computation.

\begin{figure}[t]
\centering
\caption{Unitary Picture of $K$-theory}  \label{table:summary1}
\begin{tabular}{|c|c|}
\hline 
K-group & unitary symmetries \TT \BB \\ \hline \hline
$KO_{-1}(A, \tau)$ \TT \BB & $u^\tau = u $ \\ \hline
$KO_0(A, \tau)$ \TT \BB &  $u = u^*$, $u^{\tau} = u^* $  \\ \hline
$KO_1(A, \tau)$ \TT \BB &  $u^{\tau} = u^* $ \\ \hline 
$KO_2(A, \tau)$ \TT \BB &  $u = u^*$, $u^\tau = -u$ \\ \hline 
$KO_{3}(A, \tau)$ \TT \BB & $u^{\tau\otimes \sharp} = u $ \\ \hline
$KO_4(A, \tau)$ \TT \BB &  $u = u^*$, $u^{\tau \otimes \sharp} = u^*$  \\ \hline
$KO_5(A, \tau)$ \TT \BB &  $u^{\tau \otimes \sharp} = u^*$ \\ \hline 
$KO_6(A, \tau)$ \TT \BB &  $u = u^*$, $u^{\tau \otimes \sharp} = -u$ \\ \hline \hline 
\end{tabular} \par
\caption*{
The classes in $KO_j(A,\tau)$, for a unital \calg~ with real
structure are, in our picture, given
by unitary elements of $M_n(\C) \otimes A$ with the symmetries as indicated.
See Theorem~\ref{thm:summary} and Table~\ref{table:summary} for details.}
\end{figure}

We rectify this situation by putting forward a unified description 
of all ten $K$-theory groups (eight $KO$-groups and two $KU$-groups) 
of a real \calg~ $A$ using unitaries in $M_n(\widetilde{A}\sc)$ satisfying appropriate symmetries, 
completing the project that we began in \cite{BLR}. This unified 
description is summarized in condensed form in Table~\ref{table:summary1}. A complete description of our 
picture of $K$-theory can be found in Theorem~\ref{thm:summary} and Table~\ref{table:summary}, 
which summarize the results developed in detail through 
Sections~\ref{sec:even} and~\ref{sec:odd}. A salient feature of our picture is that all of the groups are obtained without using the Grothendieck construction, so any $KO$-element can be represented exactly by a single unitary. 

The boundary maps associated to $I \rightarrow A \rightarrow A/I$ can be 
critical when calculating $K$-theory groups.  In the complex case,
both boundary maps have explicit formulas
in terms of lifting problems associated to projections and unitaries.  Any picture
of real $K$-theory should have computable boundary maps in the 24-term exact sequence
of abelian groups associated to a short exact sequence of real \calg s.

For real \calg s, we have had explicit pictures for $KO_j$ for all
$j$ except $j=3$ and $j=7$ \cite{hastingsloring}.  There were some
details missing for $j=2$ and $j=6$ to adapt to the \calg~ setting,
but essentially these cases were dealt
with already in \cite{wood1966banach}.  The boundary map has been less developed.
Given $I \rightarrow A \rightarrow A/I$ in the real case, it is a folk-theorem
that the usual formulas in the complex case work to determine both
$\partial_1: KO_1(A/I) \rightarrow KO_0(I)$ and 
$\partial_5: KO_5(A/I) \rightarrow KO_4(I)$.  For this form of $\partial_5$
it is essential to work with the isomorphism
$KO_{j+4}(D) \cong KO_j(D \otimes \mathbb \H)$ where $\H$ is the algebra of the
quaternions.

We seek a consistent picture of the $KO$ and $KU$ groups that will allow
us to have essentially only two formulas for the boundary maps, one for the
even-to-odd cases and one for the odd-to-even cases. It will also tie real
$K$-theory more closely to classical mathematics. For example, the isomorphism 
$KO_2(\R)\cong \Z _2$ can be given simply as sign of the Pfaffian of a 
self-adjoint unitary that is purely imaginary.

We work with the complexified form of a real \calg~ with the real structure
determined by a generalized involution.  That is, our objects are typically
pairs $(A,\tau)$ where $A$ is a complex \calg~ and $\tau :A \rightarrow A$
an involution that is antimultiplicative and written $a^\tau$.  In the case where
$A$ has a unit, the unitaries we consider live in $M_n(\R) \otimes A$ and the
symmetries are in terms of the usual involution $*$ and one of two extensions of
$\tau$ to matrix algebras over $A$.  These extensions are $\tau = \tr \otimes \tau$
and $\sharp \otimes \tau$ where $\tr$ is the familiar transpose and $\sharp$
is the dual operation, discussed in detail later, that is based on the derived
involution on the complexification of $\H $.

Recently  there has been much interest in physics regarding
real $K$-theory.  This has been true in string theory, to classifying
$D$-branes, and in condensed matter phonetics, to classify topological insulators.
There are mathematical reasons to study our constructions in real $K$-theory,
but lets us briefly review some of the physics.

In string theory, the utility of real $K$-theory in classifying
$D$-branes was discovered by Witten \cite{WittenKthDbranes}.  A more recent work
more closely related to this paper is \cite{berenstein2012matrix}.  More recent
developments coming from this connection have involved twisted $KR$-theory, as in 
\cite{DoranTwistedKRtheory}.  In condensed matter physics, real $K$-theory 
is used to classify topological insulators \cite{kitaev-2009,ryu2010topological}.
Many of the invariants, for example the computable invariant used to detected 3D topological
insulators \cite{fu2007topological}, do not seem at first to be part of
an $KO$ group.  Recently detailed studies of $KR$-theory of low-dimensional
spaces \cite{deNittisClassificationReal,deNittisClassificationQuaternion}
explain the place in $K$-theory for such invariants, but only in the
case of no disorder.  For methods that handle disorder, see 
\cite{EsinGurari-Bulk-boundary,Loring15,Mondr_Prodan_AIII_1D}.

The ten-fold way in physics  \cite{ryu2010topological} was a key motivation for this
work.  The Altland-Zirnbauer \cite{altland1997nonstandard} classification of 
the essential antiunitary symmetries on a quantum system has ten symmetry
classes, named according to associated Cartan labels. These ten classes correspond
to the two complex and eight real $K$-theory
groups, as in Table~\ref{table:summary}.  It is hoped that the consistent and simple formulas presented here for
all ten boundary maps will be of utility in understanding the indices being developed
in physics.

A typical problem involving topological insulators and $K$-theory involves
a collection of maps
\[
\varphi_t : C(\mathbb{T}^2) \rightarrow \widetilde{\K} 
\]
that are asymptotically multiplicative, while exactly preserving addition,
adjoint and the given real structure.  That is,
we have an element of real $E$-theory.  To identify that element, we need only pair
it with each of the two generators of $KO_{-2}(C(\mathbb{T}^2),\mathrm{id})$.  Other spaces and involutions
arise in a similar fashion, as in \cite{Loring15,LorSorensenOrtho}.  A typical 
real structure on $C(X)$ is $f^\tau =f$ on the domain and a typical
real structure on the compact operators is the dual operation.  
Thus the initial problem is how to calculate an explicit generator of
$KO_{-2}(C(\mathbb{T}^2),\mathrm{id})$.  Let us revisit how the calculation would look 
in the familiar complex case, where we need a generator of the reduced $KU_0$
group.

Consider the short exact sequence
\[
0 \rightarrow C_{0}((0,1)^{2})
 \rightarrow C(\mathbb{T}^{2})
 \rightarrow C(S^{1}\vee S^{1})
 \rightarrow 0
\]
coming from the closed copy of $S^{1}\times S^{1}$ consisting of
points $(z,w)$ that have either $z=1$ or $w=1$. We need to compute
the boundary map
\[
\partial_{1}:KU_{1}(C(S^{1}\vee S^{1}))\rightarrow KU_{0}(C_{0}((0,1)^{2})).
\]
This is easy. One generator of $KU_{1}(C(S^{1}\vee S^{1}))$ is $u_{1}$
defined by $(z,w)\mapsto z$. This lifts as a unitary $v_{1}$ to
$C(\mathbb{T}^{2})$. The same is true of the other generator
so $\partial_{1}=0$. Therefore
\[
\iota_* : KU_{0}(C_{0}((0,1)^{2}))\rightarrow KU_{0}(C(\mathbb{T}^{2}))
\]
is an inclusion, and the element we need comes from the generator
of $KU_{0}(C_{0}((0,1)^{2}))$. To find that, one can look at the
exact sequence 
\[
0\rightarrow C_{0}(U)\rightarrow C(\mathbb{D})\rightarrow C(S^{1})\rightarrow0
\]
and compute $\partial_{1}$ on the unitary $u(z)=z$.  Here  $U$  is the open
disk.

With a few modifications, the standard method to compute $\partial_{1}([u])$
for a unitary in $B$ is as follows, assuming
\[
0\rightarrow I\rightarrow A\rightarrow B\rightarrow0
\]
is exact with $A$ unital.  The first step is to lift $u$ to an element $a$ in $A$
with $\|a\|\leq1$ and then form the projection
\begin{equation}
p=\begin{pmatrix}
aa^{*} & a\sqrt{1-a^{*}a}\\
a^{*}\sqrt{1-aa^{*}} & 1-a^{*}a
\end{pmatrix} .
\label{eq:complex_delta_1}
\end{equation}
To see how this arises from the more usual formulas \cite[\S 9.1]{rordambookblue}, notice
\[
v=\begin{pmatrix}
a & -\sqrt{1-aa^{*}}\\
\sqrt{1-a^{*}a} & a^{*}
\end{pmatrix}
\]
is a unitary in $A$ (c.f. \cite[Lemma 9.2.1]{rordambookblue}) that is a lift of ${{\diag}}(u, u^*)$.
Then $p = v {{\diag}}(1, 0) v^*$ and so, up to identifying $M_{2}(\widetilde{I})$ with a subalgebra
of $M_{2}(B)$, we have $\partial_{1}([u])=[p]-[1]$.

Applying (\ref{eq:complex_delta_1}) in the case $u(z)=z$ on the circle we lift
(extend) to a function $a(z)=z$ on the disk. Then
\[
p(z)= \begin{pmatrix}
|z|^{2} & z\sqrt{1-|z|^{2}}\\
\overline{z}\sqrt{1-|z|^{2}} & 1-|z|^{2}
\end{pmatrix} \; . \]
Taken as a map on the sphere, this is a degree-one mapping of $S^{2}$
onto the set of projections in $M_{2}(\mathbb{C})$ of trace
one. In terms of the real coordinates $(x,y,z)$ restricted to the unit
sphere, we find the desired element of $KU(C_{0}(U))$ is
\[
\left[\left(\begin{array}{cc}
\tfrac{1}{2}z & \tfrac{1}{2}x-\tfrac{i}{2}y\\
\tfrac{1}{2}x+\tfrac{i}{2}y & \tfrac{1}{2}-\tfrac{1}{2}z
\end{array}\right)\right]-\left[\left(\begin{array}{cc}
1 & 0\\
0 & 0
\end{array}\right)\right].
\]
Pushing this forward to $C(\mathbb{T}^{2})$ is a little tricky. One
solution is the element
\begin{equation}
\left[\left(\begin{array}{cc}
f(z) & g(z)+h(z)\overline{w}\\
g(z)+h(z)w & 1-f(z)
\end{array}\right)\right]-\left[\left(\begin{array}{cc}
1 & 0\\
0 & 0
\end{array}\right)\right]
\label{eq:complex_Bott_on_Torus}
\end{equation}
where $f$, $g$ and $h$ are certain real-valued functions on the
circle satisfying $gh=0$ and $f^{2}+g^{2}+h^{2}=1$, as discussed
in \cite{loring14}.

Our immediate goal is to allow the calculation of generators of 
$KO_*$ groups to proceed in essentially the same manner as in the
preceeding calcuation.  In particular, the generator of $KO_{-1}(C(S^1),\mathrm{id})$
will be $[u]$ where $u(z)=z$.  What will be new is having to check
that this matrix is symmetric.

Given
\[
0\rightarrow I\rightarrow A\rightarrow B\rightarrow 0
\]
exact, and unital, but now with real structures, given $u$ a unitary in 
$B$ with $u^\tau =u$, we have a representative of a $KO_{-1}$ class.  To calculate the
boundary, we lift to $a$ with $\|a\|\leq 1$ and $a^\tau = a$ and
form
\[
w=\begin{pmatrix}
2aa^{*}-1 & 2a\sqrt{1-a^{*}a}\\
2a^{*}\sqrt{1-aa^{*}} & 1-2a^{*}a
\end{pmatrix} .
\]
Then $w$ is unitary, self-adjoint, and with the more subtle symmetry
that is component-wise given as 
\[
w_{11}^{\tau}=-w_{22},\quad w_{12}^{\tau}=w_{12},\quad w_{21}^{\tau}=w_{21}.
\]
We will see this is valid to specify an element of $KO_{-2}(I)$.  Thus the
boundary map $\partial_{-1}:KO_{-1}(B) \rightarrow KO_{-2}(I)$ looks very much like
the odd boundary map in the complex case.  We will see that the generator of 
$KO_{-2}(C_{0}(U),\mathrm{id})$ is
\[
\left[\left(\begin{array}{cc}
z & x-iy\\
x+iy & -z
\end{array}\right)\right].
\]
The generator of $KO_{-2}(C(\mathbb T ^2),\mathrm{id})$ will be the same
as in (\ref{eq:complex_Bott_on_Torus}) with just a small modification of
the three functions.

The even boundary maps will also be given as a lifting problem.  A unitary
$u$ with $u^*=u$ and other symmetries gets lifted to $x$ with $-1\leq x\leq 1$
and other symmetries.  The unitary needed is then
\[
- \exp( \pi i x)
\]
which is again very close to the complex case.  Indeed, by reformulating the complex
case in terms of self-adjoint unitaries for $KU_0$ this will be the formula for
the even boundary map.  It should be noted that we are losing track of the
order structure on $KO_0$ and $KU_0$.  In principle we can recover this, but have
no present need.

As preliminary work to developing this picture of real $K$-theory, but 
of independent interest, we also present a collection of 
classifying algebras $A_i$ for $i \in \{0, 1, \dots, 8\}$. These are 
real semiprojective homotopy symmetric
\calg s that classify $K$-theory in the sense that
$$KO_i(D) \cong [A_i, \K\pr \otimes D] \cong \lim_{n \to \infty} [A_i, M_n(\R) \otimes D] \; $$
for all $i$, as we show in Theorem~\ref{thm:classify}. The algebras $A_i$ are thus real analogs of the complex \calg s
$q\C$ and $C_0(\R, \C)$ which are classifying algebras for $K$-theory in the category of complex \calg  s in the same sense. Not unexpectedly, the algebras $A_i$ will all be real forms of matrix algebras over $q\C$ and $C_0(\R,  \C)$.

In Section~\ref{sec:Ai}, we introduce the real \calg s $A_i$ for 
$0 \leq i < 8$ and we calculate their united $K$-theory, finding that 
$KO_*(A_i) \cong \Sigma^{-i} K_*(\R)$. It follows from this (or rather from the stronger statement $K\crt(A_i) \cong \Sigma^{-i} K\crt(\R)$) and the universal 
coefficient theorem that there is a real $KK$-equivalence between $A_i$ and 
$S^{-i} \R$ and that 
$KO_i(B) \cong KKO_0(A_i, B)$
for any real separable \calg~ $B$ in the UCT bootstrap category. 
Also in Section~\ref{sec:Ai}, we will show that each $A_i$ is semiprojective, 
following a short detour to prove a key semiprojectivity closure theorem. 
Then in Section~\ref{sec:Ethy} we will prove 
that each $A_i$ is homotopy symmetric.  We validate the real version of
unsuspended $E$-theory, and it then follows that 
these algebras represent $K$-theory in the strong sense that 
$KO_i(B) \cong \lim_{n \to \infty} [A_i, B \otimes M_n(\R)]$ for 
any separable real \calg~ $B$.

In Sections~\ref{sec:even} and~\ref{sec:odd} we will develop the 
unitary picture of $K$-theory, first in the even degrees and then in the 
odd degrees. We note that we are not attempting to accomplish a complete 
development of $K$-theory from scratch using the 
unitary picture -- although that would be an interesting project. 
Instead, we take it for granted that $K$-theory is an established entity 
with known properties. We will define a sequence of groups $KO_i^u(A)$ 
in terms of unitaries and will then develop its properties mainly to 
get to the point of being able to prove that there is a natural 
isomorphism $KO_i(A) \cong KO_i^u(A)$ in each case. 

Section~\ref{sec:summary} explores some examples where the generators of
the $KO$ groups can be found easily by comparing with the complex case.
Section~\ref{sec:boundary} finds and describes formulas for the eight boundary maps,
and these are applied in Section~\ref{sec:examples2} in finding more
explicit generators of real $K$-theory groups, for several examples.

\section{Preliminaries} \label{sec:prelim}

The category of interest in this paper is the category $\mathcal{R}^*$ of real \calg s (also known as $R\sp *$-algebras), with real *-algebra homomorphisms. A real \calg~ (as in Section~1 of \cite{schroderbook}) is a real Banach *-algebra satisfying the norm condition $\|a^* a \| = \|a\|^2$ and the condition that $\1 + a^* a$ is invertible (in $\widetilde{A}$) for all $a \in A$. 

The category $\mathcal{R}^*$ is equivalent to the category $\mathcal{R}^{*,\tau}$ of \ctalg s with \ctalg ~homomorphisms (see \cite{loringsorensen}). A \ctalg ~is a pair $(A, \tau)$ where $A$ is a (complex) \calg~ and $\tau$ is an involutive antiautomorphism on $A$. Given a \ctalg ~$(A, \tau)$, the corresponding real \calg~ is 
$$A^\tau = \{ a \in A \mid a^\tau = a^* \} \; .$$

Conversely, given a real \calg~ $A$ there is a unique complexification $A\sc = A \otimes\sr \C$, which as an algebra is isomorphic to $A + iA$. The formula $(a + ib) \mapsto (a^* + ib^*)$ is an antimultiplicative involution on $A\sc$. This construction gives a functor from $\mathcal{R}^*$ to $\mathcal{R}^{*, \tau}$, which is inverse (up to isomorphism) to the functor described in the previous paragraph.

We will slide back and forth easily between these two categories, as is appropriate for the situation. In particular, whereas our unitary description of $KO^u_0(-)$ and $KO^u_1(-)$ can be made in terms of a real \calg~ $A$, our description of $KO_i^u(-)$ for other values of $i$ requires the context of a \ctalg. Hence we present our unified picture of $KO_i(-)$ for all $i$ in the setting of a \ctalg ~ (see Section~\ref{sec:summary}). This approach is analogous to the development of $K$-theory for topological spaces with involution in \cite{atiyah66}.

If $(A, \tau)$ is a \ctalg , then so is $(M_n(\C) \otimes A, \tau_n)$ where $\tau_n = \tr_n \otimes \tau$ and $\tr_n$ is the transpose operation on $M_n(\C)$. We will frequently neglect the subscripts on $\tau$ and $\tr$ when we can do so without sacrificing clarity. Similarly, we will let $\tau$ also denote the involution on $\K \otimes A$ induced by $\tau_n$ through a choice of isomorphism $\lim_{n \to \infty} (M_n(\C)) \cong \K$.
These constructions correspond to the real \calg~ constructions of tensoring by $M_n(\R)$ or by $\K\pr$, the real \calg~ of compact operators on a separable real Hilbert space.

There is a related antiautomorphism $\trt$ on $M_2(\C)$ defined by
 $$ \begin{pmatrix} a & b \\ c & d \end{pmatrix}^{\trt} 
    = \begin{pmatrix} d & b \\ c & a \end{pmatrix}
 \; . $$ 
This involution is equivalent to $\tr$ in the sense that there is an isomorphism of \ctalg s,
$(M_2(\C), \tr) \cong (M_2(\C), \trt)$. Indeed, the reader can check that $(W x W^*)^{\tr} = W x^{\trt} W^*$
where
$$ W = \frac{1}{\sqrt{2}} \begin{pmatrix} i & 1 \\ 1 & i \end{pmatrix} \; .$$
More generally, we have $(M_2(\C) \otimes A, \tau) \cong (M_2(\C) \otimes A, \taut)$
where the automorphism $\taut$ is defined by
$$ \begin{pmatrix} a & b \\ c & d \end{pmatrix}^{\taut} 
    = \begin{pmatrix} d^\tau & b^\tau \\ c^\tau & a^\tau \end{pmatrix}
 \; .$$

There is yet another real structure on $M_{2n}(A)$ and on $\K \otimes A$, which is genuinely distinct from $\tr$.
Define $\sharp \colon M_2(\C) \rightarrow M_2(\C)$ by 
  $$ \begin{pmatrix} a & b \\ c & d \end{pmatrix}^\sharp 
    = \begin{pmatrix} d & -b \\ -c & a \end{pmatrix}
 \; .$$
Then $(M_2(\C), \sharp)$ corresponds to the real \calg~ $\H$ of quaternions, and 
$(M_{2n}(\C), \sharp \otimes \tr_n)$ corresponds to the real \calg~ $M_{n}(\H)$. More generally,
if $(A, \tau)$ is a \ctalg , then $(M_{2n}(\C) \otimes A, \sharp \otimes \tr_n \otimes \tau)$
is a \ctalg ~that corresponds to the real \calg~ $M_n(\H) \otimes A^\tau$.  

We will be dealing with these matrix algebras frequently in the subsequent work, and the technicalities require that we clarify the conventions for the action of $\sharp \otimes \tau$ on a matrix in $M_{2n}(A)$, since this action requires a particular choice of isomorphism $M_2(A) \otimes M_{n} (A) \cong M_{2n}(A)$.  The two obvious choices of such an isomorphism lead to two conventions for $\sharp \otimes \tau$ that we will make use of regularly.  
The first is shown by organizing the matrix $a \in M_{2n}(A)$ as an $n \times n$ matrix whose entries are $2 \times 2$ blocks, denoted by $b_{ij} \in M_2(A)$. Then
$$a^{\sharp \otimes \tau} = \begin{pmatrix} 
b_{1 \, 1} & b_{1 \, 2} & \dots &  b_{1 \, n} \\
b_{2 \, 1} & b_{2 \, 2} & \dots &  b_{2 \, n} \\
\vdots  & \vdots & \ddots  & \vdots \\
b_{n \, 1} & b_{n \, 2} & \dots & \ b_{n \, n} 
 \end{pmatrix}^{\sharp \otimes \tau} 
= \begin{pmatrix}b_{1 \, 1}^{\sharp \otimes \tau} & b_{2 \, 1}^{\sharp \otimes \tau} & \dots &  b_{n \, 1}^{\sharp \otimes \tau} \\
b_{1 \, 2}^{\sharp \otimes \tau} & b_{2 \, 2}^{\sharp \otimes \tau} & \dots &  b_{n \, 2}^{\sharp \otimes \tau} \\
\vdots  & \vdots & \ddots  & \vdots \\
b_{1 \, n}^{\sharp \otimes \tau} & b_{2 \, n}^{\sharp \otimes \tau} & \dots & \ b_{n \, n}^{\sharp \otimes \tau} \end{pmatrix} \; .$$ 
The second convention for an involution on $M_{2n}(A)$ will be denoted by $\widetilde{\sharp} \otimes \tau$ and is shown by organizing the matrix $a \in M_{2n}(A)$ as a $2 \times 2$ matrix whose entries are $n \times n$ blocks, denoted by $c_{ij} \in M_n(A)$. Then
$$a^{\widetilde{\sharp} \otimes \tau} = \begin{pmatrix} 
c_{1 \, 1} & c_{1 \, 2} \\ c_{2\, 1} & c_{2 \, 2} 
 \end{pmatrix}^{\widetilde{\sharp} \otimes \tau}
= \begin{pmatrix} c_{2 \, 2}^{\tau_n} & -c_{1 \, 2}^{\tau_n}  \\
-c_{2 \, 1}^{\tau_n} & c_{1 \, 1}^{\tau_n} \end{pmatrix} \; .$$ 
The first convention for $\sharp \otimes \tau$ will be our preferred convention.

As mentioned, we will take for granted the full development and known properties of both $K$-theory and $KK$-theory for real \calg s. The development of $KK$-theory for real \calg s goes back to \cite{kasparov80} while much what is known about both $K$-theory and $KK$-theory can be found in \cite{schroderbook}.

For a real \calg~ $A$, we will also occasionally make reference to the united $K$-theory $K\crt(A)$, as developed in \cite{boersema02}. Briefly, $K\crt(A)$ consists of the eight real $K$-theory groups $KO_i(A)$, the two complex $K$-theory groups $KU_i(A)$ (coinciding with the $K$-theory of the complexification of $A$), and the four self-conjugate $K$-theory groups $KT_i(A)$; as well as the several natural transformations among them. The main result about united $K$-theory that we will make use of is the Universal Coefficient Theorem proven in \cite{boersema04}, which implies that united $K$-theory classifies $KK$-equivalence for real \calg s that are nuclear, separable, and in the bootstrap class for the UCT.

For a final note regarding conventions, we will use $\1$ to denote the adjoined unit in $\widetilde{A}$ for any \calg~ $A$ (unital or not). Similarly $\1_n$ will denote the diagonal identity matrix in $M_n(\widetilde{A})$.

\section{Semiprojective suspension $C^*$-algebras} \label{sec:Ai}

Let 
$q\C = \{ f \in C_0((0, 1], M_2(\C) ) \mid f(1) \in \C^2 \} \; $
where we are identifying $\C^2$ with the subalgebra of diagonal elements of $M_2(\C)$.
The algebras $A_i$ for $i$ even are defined as follows. Three are real structures of $q\C$ and one is a real structure of $M_2(q\C)$.
\begin{align*}
A_0 & = \{ f \in C_0((0, 1], M_2(\R)) \mid f(1) \in \R^2 \} \\
A_2 &= \{ f \in C_0((0, 1], \H )\mid f(1) \in \C \} \\
A_4 &= \{ f \in C_0((0, 1], M_2(\H))\mid f(1) \in \H^2 \} \\
A_6 &= \{ f \in C_0((0, 1], M_2(\R)) \mid f(1) \in \C \} \\
\end{align*}

For $i$ odd, the algebras $A_i$ are defined as follows. Each is a real structure of either 
$C_0(S^1 \setminus \{1\}, \C)$
or $C_0(S^1 \setminus \{1\}, M_2(\C))$.
\begin{align*}
A_{-1} &= 
S \R = \{ f \in C(S^1, \R) \mid f(1) =0 \} \\
A_1 & = S^{-1} \R = \{ f \in C(S^1, \C) \mid 
      f(1) =0 \text{~and~} f(\overline{z}) = \overline{f(z)}  \} \\
A_3 &= S \H = \{ f \in C(S^1, \H) \mid f(1) =0 \} \\
A_5 &= S^{-1} \H  = \{ f \in C(S^1, M_2(\C)) \mid 
      f(1) =0 \text{~and~} f(\overline{z})^\sharp = \overline{f(z)}  \}\\
\end{align*}

These real \calg s have corresponding objects in the cateogry of \ctalg s as shown in Table~\ref{table:Ai}. In this table, the involution $\zeta$ denotes the involution on $C_0(S^1 \setminus \{1\}, \C)$ induced by the involution $z \mapsto \overline{z}$ on $S^1$. 

\begin{figure}[t] 
\centering
\caption{The Classifying Algebras}  \label{table:Ai}
$$\begin{tabular}{r|c|c|}
\cline{2-3}
& $~R\sp *$-algebra~ \TT \BB & \ctalg \\ \hhline{===}
\multicolumn{1}{|r|}{\multirow{4}{*}{even cases}} 
			& $A_0$ \TT \BB & $(q\C, \tr)$  \\ \cline{2-3} 
\multicolumn{1}{ |c|  }{} & $A_2$ \TT \BB & $(q\C , \sharp )$ \\ \cline{2-3}  
\multicolumn{1}{ |c|  }{} & $A_4$ \TT \BB & $(M_2(\C) \otimes q\C , \sharp \otimes \tr)$ \\ \cline{2-3}
\multicolumn{1}{ |c|  }{} & $A_6$ \TT \BB & $(q\C, \trt)$ \\ \hhline{===}
\multicolumn{1}{|r|}{\multirow{4}{*}{odd cases}} 
			& $A_{-1}$ \TT \BB & $(C_0(S^1 \setminus \{1\}, \C), \id )$ \\ \cline{2-3}
\multicolumn{1}{ |c|  }{} & $A_1$ \TT \BB & $(C_0(S^1 \setminus \{1\}, \C)), \zeta)$ \\ \cline{2-3}
\multicolumn{1}{ |c|  }{} & $A_3$ \TT \BB & $(M_2(\C) \otimes C_0(S^1 \setminus \{1\}, \C) , \sharp \otimes \id )$ \\ \cline{2-3}  
\multicolumn{1}{ |c|  }{} & $A_5$ \TT \BB & $(M_2(\C) \otimes C_0(S^1 \setminus \{1\}, \C) , \sharp \otimes \zeta)$ \\ \hhline{===}
\end{tabular}
$$
\caption*{
This table shows the real \calg s $A_i$ and the corresponding objects in the category of \ctalg s.
They classify real $K$-theory in the sense of Theorem~\ref{thm:classify}.}
\end{figure}

\begin{prop} \label{thm:KAieven}
$K\crt(A_i) \cong \Sigma^{-i} K\crt(\R)$ for all $i \in \{ 0, 2, 4, 6\}$.
\end{prop}

\begin{proof}
In each case, $A_i \otimes \C \cong q\C$ or $A_i \otimes \C \cong M_2(q\C)$. So $K_*(A_i \otimes \C) \cong K_*(q\C) \cong K_*(\C)$. Thus by Theorem~3.2 of \cite{bousfield90}, $K\crt(A_i)$ is a free \ct-module. Furthermore, from Section~2.4 of \cite{bousfield90}, the only free \ct-module that has the complex part isomorphic to $K_*(\C)$ is $K\crt(\R)$ up to an even suspension. Therefore there are only four possibilities for $K\crt(A_i)$ up to isomorphism. A full description of the \ct-module $K\crt(\R)$ is in Table~1 of \cite{boersema02}, but in particular recall that the real part of it is given by $KO_*(\R)$ as shown below. In each case, this will be enough to determine which of the four possible suspensions is isomorphic to $K\crt(A_i)$.
$$\begin{array}{|c||c|c|c|c|c|c|c|c|}
\hline
i & 0 & 1 & 2 & 3 & 4 & 5 & 6 & 7  \\ \hline
KO_i(\R) & \Z & \Z_2 & \Z_2 & 0 
	    & \Z & 0 & 0 & 0 \\
\hline
\end{array}$$

We first consider $A_0$. Use the extension of real \calg s
\begin{equation} \label{qrext1}
0 \rightarrow SM_2(\R) \xrightarrow{\iota} A_0 \xrightarrow{{\ev}_1} \R^2 \rightarrow 0
\end{equation}
where ${\ev}_1$ is the evaluation map at $t= 1$.
Then we have the long exact sequence 
$$\dots \rightarrow K\crt(\R^2) \xrightarrow{\partial} K\crt(\R) \xrightarrow{\iota_*} K\crt(A_0) \xrightarrow{({\ev}_1)_*} K\crt(\R^2) \xrightarrow{\partial}  \cdots \; .$$

The map $\partial$ as written has degree 0 and can be determined
by its action on the generators of the two $K\crt(\R)$ summands, which are elements in $KO_0(\R) \cong \Z$. 
The complex part of this long exact sequence arises from the complexification of Sequence~(\ref{qrext1}), which is
$$0 \rightarrow SM_2(\C) \xrightarrow{\iota} q\C \xrightarrow{{\ev}_1} \C^2 \rightarrow 0 $$
and for which the boundary map $\partial \colon K_0(\C^2) \rightarrow K_0(M_2(\C))$ is known to be $\Z^2 \xrightarrow{\smh{1}{1}} \Z$ up to isomorphism.
In the commutative diagram below, we know that the complexification maps $c$ are both isomorphisms, so it follows that the boundary map $\partial \colon K_0(\R^2) \rightarrow K_0(M_2(\R))$ is also isomorphic to $\Z^2 \xrightarrow{\smh{1}{1}} \Z$.
$$\xymatrix{
KO_0(\R^2) \ar[r]^\partial \ar[d]_c &
KO_0(M_2(\R)) \ar[d]^c \\
K_0(\C^2) \ar[r]^\partial &
K_0(M_2(\C))
} $$
It follows that $\partial \colon K\crt(\R^2) \rightarrow K\crt(\R)$ is surjective and has kernel isomorphic to $K\crt(\R)$. Thus $K\crt(A_0) \cong K\crt(\R)$.

For $A_2$, we have the short exact sequence
\begin{equation} \label{qrext2}
0 \rightarrow S \H \rightarrow A_2 \xrightarrow{{\ev}_1} \C \rightarrow 0
\end{equation}
and the corresponding long exact sequence
$$\dots \rightarrow K\crt(\C) \xrightarrow{\partial} K\crt(\H) \rightarrow K\crt(A_2) \rightarrow K\crt(\C) \xrightarrow{\partial} K\crt(\H) \rightarrow \cdots \; .$$
The complexification of Sequence~(\ref{qrext2}) is the same as that of Sequence~(\ref{qrext1})
so again we can use the complexification map to calculate the boundary map.
The commutative diagram we obtain is as follows which shows that $\partial \colon KO_0(\C) \rightarrow KO_0(\H)$ is an isomorphism from $\Z$ to $\Z$.
$$ \xymatrix {
KO_0(\C)  \ar[d]_c \ar[r]^\partial  &
KO_0(\H) \ar[d]^c &
{\makebox[1cm]{isomorphic to}} &
\Z \ar[d]_{\smv{1}{1}} \ar[rr] & &
\Z  \ar[d]^2 \\
K_0(\C \oplus \C) \ar[r]^\partial  &
K_0(M_2(\C)) & &
\Z \oplus \Z \ar[rr]^{\smh{1}{1}} & &
\Z 
}  $$
Then the long exact sequence shows that $KO_0(A_2) \cong 0$. Of the four possibilities for
$K\crt(A_2)$, there is only one that is consistent with this fact. Thus we conclude that
$K\crt(A_2) \cong \Sigma^{-2} K\crt(\R)$.

From the K\"unneth Formula we know that $K\crt(B) \cong \Sigma^4 K\crt(\H \otimes B)$ for any real \calg .
Hence, $K\crt(A_{6})$ and $K\crt(A_{4})$ are determined by the isomorphisms 
$M_2(\R) \otimes A_2 \cong \H \otimes A_6$ and $A_4 \cong \H \otimes A_0$. 
\end{proof}

\begin{prop} \label{thm:KAiodd}
$K\crt(A_i) \cong \Sigma^{-i} K\crt(\R)$ for all $i \in \{-1, 1, 3 ,5\}$.
\end{prop}

\begin{proof}
For $i = \pm 1$ this follows from Proposition~1.20 of \cite{boersema02}.  For $i = 3,5$, this follows from the K\"unneth Formula and the isomorphisms $A_i \cong \H \otimes A_{i-4}$.
\end{proof}


As in Section~1 of \cite{loring08}, consider the following relations for elements $h,x,k$ in a \calg~ $A$:
\begin{equation} \label{qc-relations} \begin{aligned} 
h^* h + x^* x &= h, \\
k^* k + x x ^* &= k, \\
kx &= xh, \\
hk &= 0 \; . 
\end{aligned} \end{equation}
The same relations can be formally encoded by 
\begin{equation} \label{qc-relations2}
\begin{aligned}
&hk = 0, \\
T(h,x,k)^2 = &T(h,x,k)^* = T(h,x,k).
\end{aligned} \end{equation}
where
$$T(h,x,k) = \begin{pmatrix} \1-h & x^* \\ x & k \end{pmatrix} \in M_2(\widetilde{A}) \; .$$

Of particular interest, we have the elements
$$h_0 = t \otimes e_{11} \; , \qquad k_0 = t \otimes e_{22} \; , \qquad
  x_0 = \sqrt{t-t^2} \otimes e_{21}$$
that satisfy \eqref{qc-relations2} in $q\C$.
Recall from Lemma~2 of \cite{loring08} that $q\C$ is the universal \calg~ generated by $h,x,k$ subject to the relations 
\eqref{qc-relations2}. The following theorem gives a version of this result for $A_0$, $A_2$, and $A_6$; characterizing \ctalg -homomorphisms from $(q\C, \tr)$, $(q\C, \sharp)$, and $(q\C, \trt)$.

\begin{prop} \label{qCuniversal1}
Let $(A, \tau)$ be a \ctalg .  
\begin{enumerate}
\item Given elements $h,k,x$ in $A$ satisfying $h^\tau = h$, $k^\tau = k$, $x^\tau = x^*$ 
and Equations~\eqref{qc-relations}, then there exists a unique homomorphism 
$$\alpha \colon (q\C, \tr) \rightarrow (A, \tau)$$ 
such that 
$\alpha(h_0) = h$, $\alpha(k_0) = k$, and $\alpha(x_0) = x$.
\item Given elements $h,k,x$ in $A$ satisfying $h^\tau = k$, $k^\tau = h$, $x^\tau = -x$ and Equations~\eqref{qc-relations}, then there exists a unique homomorphism 
$$\alpha \colon (q\C, \sharp) \rightarrow (A, \tau)$$ 
such that 
$\alpha(h_0) = h$, $\alpha(k_0) = k$, and $\alpha(x_0) = x$.
\item Given elements $h,k,x$ in $A$ satisfying $h^\tau = k$, $k^\tau = h$, $x^\tau = x$ and Equations~\eqref{qc-relations}, then there exists a unique homomorphism 
$$\alpha \colon (q\C, \trt) \rightarrow (A, \tau)$$ 
such that 
$\alpha(h_0) = h$, $\alpha(k_0) = k$, and $\alpha(x_0) = x$.
\end{enumerate}
\end{prop}

\begin{proof}
Under the hypotheses of Part (1), Lemma~1 of \cite{loring08} gives a unique \calg~ homomorphism $\alpha \colon q\C \rightarrow A$ satisfying $\alpha(h_0) = k_0$, $\alpha(k_0) = k$, and $\alpha(x_0) = x$. It is only required here to verify that $\alpha$ respects the real structures; that is to verify that
\begin{equation} \label{alpharealstru} \alpha(a^\tr) = \alpha(a)^\tau \end{equation}
holds for all $a \in q\C$. 
In $q\C$ we have 
$$h_0^\tr = h_0, \quad k_0^\tr = k_0, \quad \text{and} \quad x_0^\tr = x_0^* \; $$ 
from which it follows that \eqref{alpharealstru} holds for $a = h_0, k_0, x_0$. But since these elements generate $q\C$ and since the set of elements that satisfy 
\eqref{alpharealstru} is a subalgebra of $q\C$, it follows that \eqref{alpharealstru} holds for on $q\C$.

The proofs in the second and third cases are the same, noting that in $q\C$ we have
$$h_0^\sharp = k_0, \quad k_0^\sharp = h_0, \quad \text{and} \quad  x_0^\sharp = -x_0$$
and
$$h_0^{\trt} = k_0, \quad k_0^{\trt} = h_0, \quad \text{and} \quad x_0^{\trt} = x_0 \; . $$
\end{proof}

The concepts of projectivity and semiprojectivity were first introduced and developed in the context of real \calg s (and \ctalg s) in Section~3 of \cite{loringsorensen}. In what follows, we extend that work by proving a significant closure theroem, namely that $A \otimes \H$ and $A \otimes M_n(\R)$ are semiprojective if $A$ is semiprojective (Theorem~\ref{Hsemiproj}). This result will subsequently be applied to show that each of the real \calg s $A_i$ is semiprojective.

The cone  $CM_2(\C) = C_0((0, 1], M_2(\C) )$ has two real structures, corresponding to the antiautomorphisms $\tr$ and $\sharp$ defined pointwise on $CM_2(\C)$. The corresponding real \calg s are 
$CM_2(\R) = C_0((0, 1], M_2(\R) ) $ and $C \H = C_0((0, 1], \H )$.
More generally $CM_n(\C)$ has one real structure for $n$ odd (corresponding to $\tr$) and two real structures for $n$ even (corresponding to $\tr$ and to $\sharp \otimes \tr$).

\begin{lemma} \label{orthogonallift}
Let $(A,\tau)$ and $(B, \tau)$ be \ctalg s and let $\pi \colon (A, \tau) \rightarrow (B, \tau)$ be a surjective \ctalg ~homomorphism. Let $h$ and $k$ be positive orthogonal elements in $B$. 
\begin{enumerate}
\item If $h^\tau = h$ and $k^\tau = k$, then there are positive orthogonal elements 
$h'$ and $k'$ in $A$ that satisfy $(h')^\tau = h'$ and $(k')^\tau = k'$.
\item If $h^\tau = k$ and $k^\tau = h$, then there are positive orthogonal elements 
$h'$ and $k'$ in $A$ that satisfy $(h')^\tau = k'$ and $(k')^\tau = h'$.
\end{enumerate}
Furthermore, $h'$ and $k'$ can be taken to satisfy $\|h' \| \leq \|h \|$ and $\|k' \| \leq \|k \|$.
\end{lemma}

\begin{proof}
Let $a = h-k$. Let $a' \in A$ be a self-adjoint lift of $a$. Furthermore, in case (1) we have $a^\tau = a$ and we can take $a'$ to satisfy the same 
by replacing $a'$ with $\tfrac{1}{2} \left( a' + (a')^\tau \right)$. 
Let $f_+, f_- \colon \R \rightarrow \R$ be defined by $f_+(t)= \max\{0, t\}$ and $f_-(t) = -\min\{t, 0\}$ so that $(f_+ - f_-)(t) = t$. Let $h' = f_+(a')$ and $k' = f_-(a')$. Then $h'$ and $k'$ are positive and orthogonal.  Also,
$$ \pi(h') = \pi(f_+(a')) = f_+(\pi(a')) = f_+(a) = h  $$
and similarly, $\pi(k') = k$. Finally, 
$$ (h')^\tau = (f_+(a'))^\tau = f_+( (a')^\tau ) = f_+(a') = h'$$
and similarly, $(k')^\tau = k'$.

In case (2) we have $a^\tau = -a$ and we can take $a'$ to satisfy the same by replacing $a'$ with $\tfrac{1}{2}\left(a' - (a')^\tau \right)$. Using $h'$ and $k'$ as above, we obtain 
$$ (h')^\tau = (f_+(a'))^\tau = f_+( (a')^\tau ) = f_+(-a') = k'  $$
and similarly $(k')^\tau = h'$.

In either case, the norm condition can be obtained by truncating the elements $h'$ and $k'$ using the functions
$g_K(t) = \min\{t, K\}$ where $K = \| h \|$ and $K = \| k \|$ respectively. 
\end{proof}

\begin{prop} \label{CM_2}
The \ctalg s $(CM_n(\C), \tr)$ and $(CM_2(\C) , \sharp)$ are projective. 
\end{prop}

\begin{proof}
A proof that $(CM_n(\C), \tr)$ is projective can be obtained by examining a proof that $CM_n(\C)$ is projective in the context of 
complex \calg s. For example, see Theorem~10.2.1 in \cite{loringbook} or Theorem~3.5 of \cite{loringpedersen98}.

To show that $(CM_2(\C), \sharp)$ is projective let $\phi \colon (CM_2(\C), \sharp) \rightarrow (B, \tau)$ be a \ctalg ~homomorphism and let $\pi \colon (A, \tau) \rightarrow (B, \tau)$ be a surjective \ctalg ~homomorphism. Let $x = \phi(t \otimes e_{12})$. Then $x$ satisfies $\|x\| \leq 1$, $x^2 = 0$, and $x^\tau = -x$. In fact $(CM_2(\C), \sharp)$ is universal for these relations so it suffices to show that $x$ can be lifted to an element in $A$ satisfying the same.

Using Lemma~\ref{orthogonallift}, lift $\phi(t^{1/3} \otimes e_{11})$ and $\phi(t^{1/3} \otimes e_{22})$ to elements $h,k \in A$ satisfying $0 \leq h,k \leq 1$, $hk = 0$, and $h^{\tau} = k$. Lift $\phi(t^{1/3} \otimes e_{12}) = x^{1/3}$ to an element $y \in A$ satisfying $y^\tau = -y$. Let $z = kyh$ so that $\pi(z) = x$, $z^2 = 0$, and $z^\tau = -z$.

To finish, let
$$f(t) = \begin{cases} 1 & 0 \leq t \leq 1 \\ t^{-1/2} & 1 \leq t \end{cases} \qquad
\text{and} \qquad
w = z f(z^* z)  \; .$$
Then $\|w \| \leq 1$ since $w^* w = f(z^* z) z^*z f(z^* z) = g(z^* z)$
where 
$$g(t) = \begin{cases} t & 0 \leq t \leq 1 \\ 1 & 1 \leq t. \end{cases}$$
Since $\pi(w) = x f(x^* x) = x$, we have that $w$ is still a lift of $x$.
We also have $w^2 =   f(z z^*) z z f(z^* z) = 0$.
Finally, we show that $w^{\tau} = -w$.  Check that $(z^* z)^\tau = z^\tau z^{* \tau} = -z (-z^*) = z z^*$ so that
$$w^\tau = (z f(z^* z))^\tau = f(z^* z)^\tau z^\tau =  -f(z z^*) z = -w \; .$$
\end{proof}

We remark the Proposition~\ref{CM_2} can be strengthened to state that the cone $(CM_{2n}(\C), \sharp \otimes \tr_n)$ is projective using a similar proof to the above. This however will be a direct consequence of Proposition~\ref{H1semiproj} below and the stronger result is not required for us before that point.

\begin{lemma} \label{thm:A}
Suppose $\varphi:C{M}_{n}(\mathbb{C})\rightarrow B$
is a $*$-homomorphism of \calg s. Denote by $B_{0}$ and
$B_{n}$ the hereditary subalgebras of $B$ generated by
$\varphi(C_{0}(0,1]\otimes e_{11})$
and $\varphi(C{M}_{n}(\mathbb{C}))$, respectively.
Then there is a natural isomorphism 
\[
\Phi:B_{0}\otimes {M}_{n}(\C) \rightarrow B_{n}.
\]
defined by
\[
\Phi\left(\varphi(f\otimes e_{1r})b\varphi(g\otimes e_{s1})\otimes e_{jk}\right)
=
\varphi(f\otimes e_{jr})b\varphi(g\otimes e_{sk})
\]
for $f,g \in C_0(0,1]$ and $b \in B$.

Furthermore, suppose 
$\varphi \colon (C M_n(\C), \tr) \rightarrow (B, \tau)$
is a \ctalg ~homomorphism of \ctalg s. Then there is a natural isomorphism of \ctalg s
$$\Phi \colon (B_{0} \otimes M_n(\C), \tau \otimes \tr) \rightarrow (B_{n}, \tau)$$
given by the same formula as above.
\end{lemma}

\begin{proof}
We start with some nice descriptions of $B_{0}$ and $B_{n}$. Since $C(0,1] \otimes e_{11}$ is generated by 
$t \otimes e_{11}$, we have
\[
B_{0}=\overline{\varphi(t\otimes e_{11})B\varphi(t\otimes e_{11})}
\qquad \text{and} \qquad 
B_{n}=\overline{\varphi(t\otimes 1_n)B\varphi(t\otimes 1_n)}.
\]
On the other hand the nice factorization result, Corollary~4.6 of \cite{pedersen98}, implies that
\begin{align*}
B_{0} &=
\varphi(C_{0}(0,1]\otimes e_{11})B\varphi(C_{0}(0,1]\otimes e_{11})   \\
\text{and~}  B_{n} &=
\varphi(C_{0}(0,1]\otimes 1_n)B\varphi(C_{0}(0,1]\otimes 1_n) \; ,
\end{align*}
which shows that it is enough to define $\Phi$ on the elements of the form
$$x = \varphi(f \otimes e_{11}) b \varphi (g \otimes e_{11}) \; .$$

We first establish that $\Phi$ is well-defined as a map restricted
to $B_{0}\otimes e_{jk}$. Suppose
\[
\varphi(f\otimes e_{1r})b\varphi(g\otimes e_{s1})
=
\varphi(h\otimes e_{1p})b^{\prime}\varphi(k\otimes e_{q1}).
\]
Select any $\mu_{n}$ that is an approximate identity in $C_{0}(0,1]$
and calculate:
\begin{align*}
& \varphi(f\otimes e_{jr})b\varphi(g\otimes e_{sk}) \\
& \quad=\lim_{m}\lim_{n}
  \varphi(\mu_{n}\otimes e_{j1})
  \varphi(f\otimes e_{1r})
  b
  \varphi(g\otimes e_{s1})
  \varphi(\mu_{m}\otimes e_{1k})\\
&\quad =\lim_{m}\lim_{n}
  \varphi(\mu_{n}\otimes e_{j1})
  \varphi(h\otimes e_{1p})
  b^{\prime}
  \varphi(k\otimes e_{q1})
  \varphi(\mu_{m}\otimes e_{1k})\\
& \quad=\varphi(h\otimes e_{jp})
  b^{\prime}
  \varphi(k\otimes e_{qk}).
\end{align*}
To see this is additive, consider two elements in $B_{0}$, 
\[
x=\varphi(f\otimes e_{11})b\varphi(g\otimes e_{11}),
 \quad 
y=\varphi(h\otimes e_{11})b' \varphi(k\otimes e_{11}).
\]
We claim that we can rewrite these elements so that $f = h$ and $g = k$.  
Indeed, we can factor the functions as $f= \mu f_{1}$ and $h= \mu h_{1}$ where $\mu(x) = \sqrt{|f(x)| + |h(x)|}$
to get $f \otimes e_{11} = (\mu \otimes e_{11})(f_1 \otimes e_{11})$ 
and $h \otimes e_{11} = (\mu \otimes e_{11})(f_1 \otimes e_{11})$
\[
x=\varphi(\eta\otimes e_{11})b^{\prime\prime\prime}\varphi(g\otimes e_{11}),
\quad 
y=\varphi(\eta\otimes e_{11})b^{\prime\prime}\varphi(k\otimes e_{11}).
\]
Perform a similar procedure using the functions $g$ and $k$. Therefore, we can assume that we have
\[
x=\varphi(f\otimes e_{11})b\varphi(g\otimes e_{11}),
\quad 
y=\varphi(f\otimes e_{11})b^{\prime}\varphi(g\otimes e_{11}).
\]
Then we prove additivity as follows,
\begin{align*}
& \Phi\left(
 \varphi(f\otimes e_{11})b\varphi(g\otimes e_{11})\otimes e_{jk}
 +
 \varphi(f\otimes e_{11})b^{\prime}\varphi(g\otimes e_{11})\otimes e_{jk}
 \right)\\
& \quad=\Phi\left(
 \varphi(f\otimes e_{11})(b+b^{\prime})\varphi(g\otimes e_{11})\otimes e_{jk}
 \right)\\
& \quad=
  \varphi(f\otimes e_{j1})(b+b^{\prime})\varphi(g\otimes e_{1k})\\
& \quad=
  \varphi(f\otimes e_{j1})b\varphi(g\otimes e_{1k})
  +
  \varphi(f\otimes e_{j1})b^{\prime}\varphi(g\otimes e_{1k})\\
& \quad=
  \Phi\left(
  \varphi(f\otimes e_{11})b\varphi(g\otimes e_{11})\otimes e_{jk}
  \right)
  +
  \Phi\left(
  \varphi(f\otimes e_{11})b^{\prime}\varphi(g\otimes e_{11})\otimes e_{jk}
  \right).
\end{align*}
Now we easily conclude that $\Phi$ is a well-defined linear map on all
of $B\otimes M_{n}(\mathbb{C})$.

As to the product, we observe 
\begin{align*}
& \Phi\left(
 \varphi(f\otimes e_{11})b\varphi(g\otimes e_{11})\otimes e_{jk}
 \right)
 \Phi\left(
 \varphi(g\otimes e_{11})b^{\prime}\varphi(h\otimes e_{11})\otimes e_{kl}
 \right)\\
& \quad=
  \varphi(f\otimes e_{j1})b\varphi(g\otimes e_{1k})
  \varphi(h\otimes e_{k1})b\varphi(k\otimes e_{1l})\\
& \quad=
  \varphi(f\otimes e_{j1})b\varphi(gh\otimes e_{11})b\varphi(k\otimes e_{1l})\\
& \quad=
 \Phi\left(
 \varphi(f\otimes e_{11})b\varphi(gh\otimes e_{11})b\varphi(k\otimes e_{11})\otimes e_{jl}
 \right).
\end{align*}
Proving that $\Phi$ is a $*$-homomorphism is easier:
\begin{align*}
  \Phi\left(
  \left(\varphi(f\otimes e_{11})b\varphi(g\otimes e_{11})\otimes e_{jk}\right)^{*}
  \right) 
& =\Phi\left(
  \varphi(\bar{g}\otimes e_{11})b^{*}\varphi(\bar{f}\otimes e_{11})\otimes e_{kj}
  \right)\\
& =\varphi(\bar{g}\otimes e_{k1})b^{*}\varphi(\bar{f}\otimes e_{1j})\\
& =\left(\varphi(f\otimes e_{j1})b\varphi(g\otimes e_{1k})\right)^{*}\\
& =\left(\Phi\left(
  \varphi(f\otimes e_{11})b\varphi(g\otimes e_{11})\otimes e_{jk}
  \right)\right)^{*}.
\end{align*}

To prove $\Phi$ is onto, we start with an element 
\[
\varphi(f\otimes 1_n)b\varphi(g\otimes 1_n)
\]
 which we expand as
\[
\sum\varphi(f\otimes e_{jj})b\varphi(g\otimes e_{kk})
=
\sum\Phi(\varphi(f\otimes e_{1j})b\varphi(g\otimes e_{k1})\otimes e_{jk}).
\]
Injectivity is easy since $\Phi$ will be injective if and only if
its restriction to $B_{0}\otimes e_{11}$ is injective, and
\[
\Phi\left(\varphi(f\otimes e_{11})b\varphi(g\otimes e_{11})\otimes e_{11}\right)
=
\varphi(f\otimes e_{11})b\varphi(g\otimes e_{11}).
\]

For naturality, suppose that $\gamma:B\rightarrow C$ is a homomorphism of \calg s or \ctalg s. 
Then define $\psi=\gamma\circ\varphi$ and subsequently define 
$\Psi \colon C_0 \otimes M_n(\C) \rightarrow C_n$ as above. 
Check that $\gamma(B_{0})\subseteq C_{0}$
and $\gamma(B_{n})\subseteq C_{n}$. Then we show
$ \Psi(\gamma(x)\otimes e_{jk})=\gamma(\Psi(x\otimes e_{jk}))$
as follows:
\begin{align*}
\Psi(\gamma(\varphi(f\otimes e_{11})b\varphi(g\otimes e_{11}))\otimes e_{jk}) 
& =\Psi(\psi(f\otimes e_{11})\gamma(b)\psi(g\otimes e_{11}))\otimes e_{jk})\\
& =\psi(f\otimes e_{j1})\gamma(b)\psi(g\otimes e_{1k}))\\
& = \gamma(\varphi(f \otimes e_{j1})) \gamma(b) \gamma( \varphi(g \otimes e_{1k})) \\
& = \gamma( \varphi(f \otimes e_{j1}) b \varphi( g \otimes e_{1k} )) \\
&= \gamma \Phi ( \varphi (f \otimes e_{11}) b \varphi(g \otimes e_{11}) \otimes e_{jk}) \; .
\end{align*}

In the case that there is an involution $\tau$ on $B$ with $\varphi(x^{\tr})=\varphi(x)^{\tau}$ for $x \in CM_2(\C)$,
we show that $\Phi((x \otimes e_{jk})^{\tau \otimes \tr}) = \Phi(x \otimes e_{jk})^\tau$ for $x \otimes e_{jk} $ in $B_0 \otimes M_n(\C)$:
\begin{align*}
\Phi\left(\left(\varphi(f\otimes e_{11})b\varphi(g\otimes e_{11})\right)^{\tau}\otimes e_{jk}^{\tr}\right)
  & =\Phi \left( \varphi(g\otimes e_{11})b^{\tau}\varphi(f\otimes e_{11}) \otimes e_{kj} \right)\\
  & =\varphi(g\otimes e_{k1})b^{\tau}\varphi(f\otimes e_{1j})\\
 & =\varphi((g\otimes e_{1k})^{\tr})b^{\tau}\varphi((f\otimes e_{j1})^{\tr})\\
 & =\left(\varphi(f\otimes e_{j1})b\varphi(g\otimes e_{1k})\right)^{\tau}\\
 &=  \Phi\left(  \varphi(f\otimes e_{11})b\varphi(g\otimes e_{11})\otimes e_{jk}  \right)^{\tau}.
\end{align*}
\end{proof}

\begin{cor}
Let $h$ be a strictly positive element in a \calg~ $B$. There is an embedding
$C \C \hookrightarrow B$ sending the
canonical generator to $h$. Similarly, there is an embedding
$C \C \rightarrow C M_n(\C)$ by 
$f\mapsto f\otimes e_{11}$. 
Then there is an isomorphism
\[
B *_{C \C }C M_n(\C)
\cong
B \otimes M_n(\C)
\]
given by 
\[
b\mapsto b\otimes e_{11}
\qquad \text{and} \qquad
f\otimes e_{jk}\mapsto f(h)\otimes e_{jk} \;.
\]
If there is a real structure $\tau$ on $B$ and if $h$ satisfies $h^\tau = h$,
then the isomorphism is $\tau$-preserving.
\end{cor}

\begin{lemma} \label{thm:B}
Suppose
$\varphi \colon (C M_2(\C), \sharp) \rightarrow (B, \tau)$
is a \ctalg ~homomorphism of \ctalg s.  Then there is a natural isomorphism of \ctalg s
$$\Phi \colon (B_0 \otimes M_2(\C), \sigma \otimes \sharp) \rightarrow (B_2, \tau)$$
where $B_0$, $B_2$, and $\Phi$ are as in Lemma~\ref{thm:A}
and where $\sigma$ is an antimultiplicative involution on $B_0$ defined by
\[
\left(\varphi(f\otimes e_{11})b\varphi(g\otimes e_{11})\right)^{\sigma}
=
\varphi(g\otimes e_{12})b^{\tau}\varphi(f\otimes e_{21}).
\]

Furthermore, the construction $\varphi \mapsto (B_0, \sigma)$ is natural.
\end{lemma}

\begin{proof}
We already know from Lemma~\ref{thm:A} that $\Phi$ is a well defined isomorphism.
Suppose now that $\varphi \colon CM_2(\C) \rightarrow B$ satisfies $\varphi(x^\sharp) = \varphi(x)^\tau$.
We first show that $\sigma$ is well-defined and is a real structure on $B_0$. 
Since 
\begin{align*}
\left(\varphi(f\otimes e_{11})b\varphi(g\otimes e_{11})\right)^{\sigma}\otimes e_{11} 
& =\varphi(g\otimes e_{12})b^{\tau}\varphi(f\otimes e_{21})\otimes e_{11}\\
& =\Phi^{-1}\left(
  \varphi(g\otimes e_{12})b^{\tau}\varphi(f\otimes e_{21})
  \right)\\
& =\Phi^{-1}\left(
  \varphi(g\otimes e_{12}^{\sharp})b^{\tau}\varphi(f\otimes e_{21}^{\sharp})
  \right)\\
& =\Phi^{-1}\left(
  \left(\varphi(f\otimes e_{21})b\varphi(g\otimes e_{12})\right)^{\tau}
  \right)\\
& =\Phi^{-1}\left(
  \left(\Phi\left(\varphi(f\otimes e_{11})b\varphi(g\otimes e_{11})\otimes e_{22}\right)\right)^{\tau}
  \right)
\end{align*}
we see that 
\[
x^{\sigma}\otimes e_{11}
=
\Phi^{-1}\left(\left(\Phi(x\otimes e_{22})\right)^{\tau}\right)
\]
and so $\sigma$ is an anti-$*$-homomorphism, being a composition of four homomorphisms and
one anti-homomorphism. That $\sigma$ is an involution on $B_0$ is shown by:
\begin{align*}
\left(\varphi(fg\otimes e_{11})b\varphi(hk\otimes e_{11})\right)^{\sigma^{2}} 
& =\left(\varphi(kh\otimes e_{12})b^{\tau}\varphi(gf\otimes e_{21})\right)^{\sigma}\\
& =\left( \varphi(k \otimes e_{11}) \varphi(h \otimes e_{12}) b^\tau
		\varphi(g \otimes e_{21}) \varphi(f \otimes e_{11}) \right)^\sigma \\
& = \varphi(f \otimes e_{12}) \left( \varphi(h \otimes e_{12}) b^\tau
		\varphi(g \otimes e_{21}) \right)^\tau   \varphi(k \otimes e_{21})   \\
& = \varphi(f \otimes e_{12}) \varphi(g \otimes e_{21}^\sharp) b 
		\varphi(h \otimes e_{12}^\sharp)   \varphi(k \otimes e_{21})   \\
& =\varphi(f\otimes e_{12})\varphi(g\otimes e_{21})
  b
  \varphi(h\otimes e_{12})\varphi(k\otimes e_{21})\\
& =\varphi(fg\otimes e_{11})b\varphi(hk\otimes e_{11}).
\end{align*}

Now we show that $\Phi$ commutes with the appropriate real structures; that is we prove that
$\Phi((x \otimes e_{jk})^{\sigma \otimes \sharp}) = \Phi(x \otimes e_{jk})^{\tau}$ for all $x \otimes e_{jk} \in B_0 \otimes M_2(\C)$. 
Of the four cases to consider, we will show the calculations for the cases $x \otimes e_{11}$ and $x \otimes e_{12}$ since the cases for $x \otimes e_{22}$ and $x \otimes e_{21}$ are similar.
\begin{align*}
\Phi\left(\left( \varphi( f \otimes e_{11} )b\varphi(g\otimes e_{11}) \right)^\sigma \otimes 			e_{11}^{\sharp} \right) 
& =\Phi\left(\varphi(g\otimes e_{12}) b^{\tau}\varphi(f\otimes e_{21}) \otimes e_{22} \right)\\
& =\varphi(g\otimes e_{22})b^{\tau}\varphi(f\otimes e_{22})\\
& =\left(\varphi(f\otimes e_{11})b\varphi(g\otimes e_{11})\right)^{\tau}\\
& = \Phi\left(\varphi(f\otimes e_{11})b\varphi(g\otimes e_{11})\otimes e_{11}\right)^{\tau}
\end{align*}
and 
\begin{align*}
\Phi\left(\left( \varphi( f \otimes e_{11} )b\varphi(g\otimes e_{11}) \right)^\sigma \otimes 			e_{12}^{\sharp} \right) 
& =\Phi\left(\varphi(g\otimes e_{12}) b^{\tau}\varphi(f\otimes e_{21}) \otimes - e_{12} \right)\\
& = - \varphi(g\otimes e_{12})b^{\tau}\varphi(f\otimes e_{22})\\
& =\left(\varphi(f\otimes e_{11})b\varphi(g\otimes e_{12})\right)^{\tau}\\
& = \Phi\left(\varphi(f\otimes e_{11})b\varphi(g\otimes e_{11})\otimes e_{12}\right)^{\tau}.
\end{align*}

Finally, we consider the question of naturality. For a \ctalg~homomorphism
$\gamma:(B, \tau) \rightarrow (C, \tau)$ we define $\psi=\gamma\circ\varphi$.
We obtain a real stucture $\sigma$ on $C_{0}$, 
\[
\left(\psi(f\otimes e_{11})c\psi(g\otimes e_{11})\right)^{\sigma}
=
\psi(g\otimes e_{12})c^{\tau}\psi(f\otimes e_{21}).
\]
The claim of naturality is the claim that the restriction of $\gamma$ to a map
$B_{0}\rightarrow C_{0}$ is a \ctalg~homomorphism
$
\gamma_{*}:(B_{0},\sigma)\rightarrow(C_{0},\sigma) .
$
Indeed,
\begin{align*}
  \gamma(\varphi(f\otimes e_{11})b\varphi(g\otimes e_{11}))^{\sigma}
& = \left( \psi(f\otimes e_{11})\gamma(b)\psi(g\otimes e_{11}) \right) ^{\sigma}\\
& =\psi(g\otimes e_{12})\gamma(b^\tau) \psi(f\otimes e_{21})\\
& =\gamma(\varphi(g\otimes e_{12})b^{\tau}\varphi(f \otimes e_{21}))\\
& =\gamma((\varphi(f\otimes e_{11})b\varphi(g\otimes e_{11}))^{\sigma}).
\end{align*}
\end{proof}

An important special case of Lemma~\ref{thm:B} occurs when $B= C \otimes M_2(\C) $
with involution $ \tau \otimes \sharp$
and when the map
$\varphi: (CM_2(\C), \sharp)
\rightarrow
(C \otimes M_2(\C) , \tau \otimes \sharp)$
sends $f\otimes e_{jk}$ to $f(h)\otimes e_{jk}$ for some strictly
positive self-$\tau$ element $h$ in $C$. 
In that case,
$B_{2}= C \otimes M_2(\C)$
and $B_{0}=C\otimes e_{11}$. Then 
the induced real structure $\sigma$ on $B_0$, defined
by 
\[
\left(\varphi(f\otimes e_{11})b\varphi(g\otimes e_{11})\right)^{\sigma}
=
\varphi(g\otimes e_{12})b^{\tau\otimes\sharp}\varphi(f\otimes e_{21})
\]
satisfies
\begin{align*}
\left(hbh\otimes e_{11}\right)^{\sigma}
& =\left((h\otimes e_{11})(b\otimes e_{11})(h\otimes e_{11})\right)^{\sigma}\\
& =(h\otimes e_{12})(b^{\tau}\otimes e_{22})(h\otimes e_{21})\\
& =hb^{\tau}h\otimes e_{11}.
\end{align*}
Thus we find that $\sigma$ is just $\tau\otimes\mathrm{id}$, restricted to $B_0 = C \otimes e_{11}$.

\begin{prop} \label{H1semiproj}
Let $A$ be a real \calg~ . 
If $A$ is projective then $A \otimes \H$ is projective. If $A$ is semiprojective then $A \otimes \H$ is semiprojective.
\end{prop}

\begin{proof}
We work in the category of \ctalg s. Suppose that $(A, \tau)$ is projective, that 
$(B, \tau), (C, \tau)$ are \ctalg s, and that we have \ctalg ~homomorphisms $\varphi$ and $\pi$ as in the diagram
\[
\xymatrix{
& B \ar[d] ^(0.4){\pi}\\
A \otimes M_2(\C)  \ar[r] ^(0.6){\varphi}  & C
}
\]
where $\pi$ is surjective and the involution on $A \otimes M_2(\C)$ is $\tau \otimes \sharp$. 
We select a strictly positive
element $h \in A$ satisfying $h^\tau = h$ and define
\[
\gamma: ( CM_2(\C) ,\sharp )
\rightarrow
(A \otimes M_2(\C), \tau \otimes \sharp)
\]
by $
\gamma(f\otimes e_{jk}) = f(h) \otimes e_{jk} .
$
By Proposition~\ref{CM_2} there is a homomorphism
\[
\varphi_{1}:(C M_2(\C),\sharp )
\rightarrow
(B,\tau)
\]
with $\pi\circ\varphi_{1}=\varphi \circ \gamma$. We apply
Lemma~\ref{thm:B}
to get two commutative diagrams of real \calg s. The first diagram is
\[
\xymatrix{
& B_0 \otimes M_2(\C) \ar[d] \ar[r]^(0.6){\Phi_3} 
& B_2 \ar[dd] \ar@{^{(}->}[r]
& B \ar[ddd]^\pi \\
(A\otimes e_{11}) \otimes M_2(\C) \ar[r]\ar[d]^(0.4){\Phi_1}  
& C_0 \otimes M_2(\C) \ar[rd] ^{\Phi_2} \\
A \otimes M_2(\C) \ar[rr] \ar@{=}[d]&&C_2 \ar@{^{(}->}[rd]\\
A \otimes M_2(\C) \ar[rrr]^{\varphi}&&& C
}
\]
where each $\Phi_{j}$ is an isomorphism,
and the real structures on the algebras closest to the upper
left of the diagram are all $\sigma_{j}\otimes\sharp$ where the $\sigma_{j}$
are in the second diagram:
\[
\xymatrix{
& (B_0,\sigma_3) \ar[d] \\
(A\otimes e_{11},\sigma_1) \ar[r]  & (C_0,\sigma_2)
}
\]
By the remark following Lemma~\ref{thm:B}
we know that $(A \otimes e_{11},\sigma_{1})$ is isomorphic to $(A,\tau)$
and so we get a lift in the second diagram by the hypothesis on $A$. 
Tensoring by the identity
on $M_2(\C)$ now gives a lift in the upper-left
portion of the first diagram, which then provides the desired lift of
$\varphi$.

Adjusting the given proof to the semiprojectivity case proceeds exactly as
in Section 14.2 of \cite{loringbook}.
\end{proof}

\begin{thm} \label{Hsemiproj}
If a real \calg~ $A$ is projective then $A \otimes M_n(\R)$ and $A \otimes M_n(\R) \otimes \H$ are projective for all $n$. If $A$ is semiprojective then $A \otimes M_n(\R)$ and $A \otimes M_n(\R) \otimes \H$ are semiprojective for all $n$.
\end{thm}

\begin{proof}
Suppose that $A$ is projective. The statement that $A \otimes M_n(\R)$ is projective is proven exactly as in the complex case, Theorem~3.3 of \cite{loring93}. Similarly, if $A$ is semiprojective, the proof of Theorem~4.3 of \cite{loring93} applies to the case of real \calg s to show that $A \otimes M_n(\R)$ is semiprojective. Proposition~\ref{H1semiproj} completes the proof.
\end{proof}

\begin{prop} \label{A0semiproj} \label{Aevensemiproj}
$A_i$ is semiprojective for $i$ even.
\end{prop}

\begin{proof}
First we consider $A_0$.
Suppose that $J_1 \subseteq J_2 \subseteq \dots$ be an increasing sequence of $\tau$-invariant ideals in a \ctalg ~$(B, \tau)$ and let $J = \overline{\cup_{n} J_n}$. We will use the same notation $\tau$ for the involution $\tau$ passing to each quotient algebra $B/J_n$ and $B/J$. Establish the following notation for the natural quotient maps, which all commute with $\tau$.
\begin{align*}
\pi_n &\colon B \rightarrow B/J_n \\
\pi_\infty &\colon B \rightarrow B/J \\
\pi_{n,m} &\colon B/J_n \rightarrow B/J_m \\
\pi_{n, \infty} & \colon B/J_n \rightarrow B/J
\end{align*}

Let $\phi \colon (q\C, \tr) \rightarrow (B/J, \tau)$ be a \ctalg ~homomorphism.
We will produce a \ctalg ~homomorphism $\psi \colon (q\C, \tr) \rightarrow (B/J_n, \tau)$ for some $n$ such that $\pi_{n, \infty} \circ \psi = \phi$. 

Let $h_\infty = \phi(h_0)$, $k_\infty = \phi(k_0)$, $x_\infty = \phi(x_0)$ in $B/J$. The elements $h_\infty$ and $k_\infty$ are positive, contractions, orthogonal, and fixed by $\tau$.
Thus by Lemma~\ref{orthogonallift}, there are elements $h, k \in B$ with the same properties such that 
$\pi_\infty(h) = h_\infty$ and $\pi_\infty(k) = k_\infty$.

We will take $x \in B$ to be a lift of $x_\infty$.  
As in the proof of Theorem~6 of \cite{loring08}, this can be arranged so that $x \in k^{1/8} B h^{1/8}$.
Furthermore, replacing $x$ by $\tfrac{1}{2}(x + x^{\tau *})$ we can assume that $x^\tau = x^*$ holds.
Then 
$T = T(h,x,k)$ is an element in the subalgebra
$$\widehat{B} = \begin{pmatrix} 
    \C \cdot \1 \oplus \overline{hBh} & 
    \overline{hBk} \\ 
    \overline{kBh} & 
    \C \cdot \1 \oplus \overline{kBk} \end{pmatrix} 
		\subseteq M_2(\widetilde{B})\; .$$
Furthermore, $T$ satisfies $T^{\tau \otimes \tr} = T^* = T$
and is a lift of 
$$T_\infty = T(h_\infty, x_\infty, k_\infty) \in M_2(\widetilde{B}/J) \; .$$
Since $\pi_{\infty}(T)$ is a projection, there is an $n$ large enough so that the spectrum of $T_n:= \pi_n(T) \in M_2(\widetilde{B}/J_n)$ does not contain $1/2$. Then $T'_n = f(T_n)$ is a projection in $M_2(\widetilde{B}/J_n)$ where 
$$f_{1/2}(t) = \begin{cases} 0 & \text{if~ $t < 1/2$} \\ 1 & \text{if~ $t > 1/2$.} 
	    \end{cases}$$
Furthermore, $T'_n$ is a lift of $T_\infty$ and the relation $T'_{n} = (T'_n)^{\tau \otimes \tr}$ holds.
Write
$$T'_n = \begin{pmatrix} \1- h'_n & (x'_n)^* \\ x'_n & k'_n \end{pmatrix} \; $$
where $h'_n$, $k'_n$, $x'_n$ are elements of $\widetilde{B}/J_n$ and are necessary lifts of $h_\infty$, $k_\infty$, and $x_\infty$ respectively.  Since we have $T'_n = (T'_n)^\tau$, it follows that $h_n' = (h_n')^\tau$, $k_n' = (k_n')^\tau$, and $(x_n')^* = (x_n')^\tau$.
We claim that
$h_n'$ and $k_n'$ are orthogonal. Indeed, we know that $h_n = \pi_n(h)$ and $k_n = \pi_n(k)$ are orthogonal and that $T_n$ (and hence $T'_n = f(T_n)$) lies in the subalgebra
$$\widehat{B_n} = \begin{pmatrix} 
    \C \cdot \1 \oplus \overline{h_nBh_n} & 
    \overline{h_nBk_n} \\ 
    \overline{k_nBh_n} & 
     \overline{k_nBk_n} \end{pmatrix} 
		\subseteq M_2(\widetilde{B}/J_n)\; ,$$
proving our claim.

Therefore, the elements $h_n'$, $x_n'$, and $k_n'$ are elements in $B/J_n$ which satisfy the universal relations for $q\C$ as in Proposition~\ref{qCuniversal1}, so there exists a homomorphism $\psi \colon (q\C, \tr) \rightarrow (B/J_n, \tau)$ that maps
$h_0$, $k_0$, and $x_0$ to $h'_n$, $k'_n$, and $x'_n$ respectively. Since 
$\pi_\infty(h'_n) = h_\infty$, $\pi_\infty(k'_n) = k_\infty$, and $\pi_\infty(x'_n) = x_\infty$; 
it follows that $\psi$ is a lift of $\phi$.  

For $A_2$ the proof is quite similar to that for $A_0$.  The initial difference is that we are using the involution $\sharp$ on $q\C$. So $\phi$ is assumed to satisfy $\phi(a^\sharp) = \phi(a)^\tau$ and we must find a lift $\psi$ which satisfies the same.  

If we let $h_\infty$, $k_\infty$, and $x_\infty$ be as in the proof above, then we have 
$h_\infty^\tau = k_\infty$, $k_\infty^\tau = h_\infty$, and $x_\infty^\tau = -x_\infty$. We use Lemma~\ref{orthogonallift} to find elements $h$ and $k$ in $B$ that satisfy 
$h^\tau = k$ and $k^\tau = h$. Lift $x_\infty$ to an element $x \in k^{1/8}Bh^{1/8}$ that satisfies $x^\tau = -x$ 
(using the adjustment $\tfrac{1}{2}\left( x - x^{\tau} \right)$).
Then $T = T(h,x,k)$ is in $\widehat{B}$ as before and satisfies $T = T^*$. Now we have
$$T =  \begin{pmatrix} \1- h & x^* \\ x & k \end{pmatrix} \quad 
\text{and} \quad T^{\tau \otimes \trt} =  \begin{pmatrix} h & -x^*  \\ -x  & \1-k \end{pmatrix}  $$
so we have $T^{\tau \otimes \trt} = \1_2 - T$.

Then as in the proof for $A_0$, find $n$ large enough so that $1/2$ is in the spectral gap for $T_n$ and let $T'_n = f_{1/2}(T_n)$. Then $T'_n$ is a projection and satisfies $(T_n')^{\tau \otimes \trt} = \1_2 - T'_n$ (since $f'(\1_2) = \1_2$). So we can write
$$T'_n =  \begin{pmatrix} \1- h_n' & (x'_n)^* \\ x_n' & k_n' \end{pmatrix}$$
where $(h_n')^\tau = k_n', (k_n')^\tau = h_n',$ and $(x_n')^\tau = - x_n'$. Then by Proposition~\ref{qCuniversal1}, there exists a homomorphism $\psi$ which is the desired lift of $\phi$.

Now we consider the case of $A_{6}$. In this case the we have elements in $B/J$ that satisfy 
$h_\infty^\tau = k_\infty$, $k_\infty^\tau = h_\infty$, and $x_\infty^\tau = x_\infty$ in $B/J$ which are lifted to elements in $B$ that satisfy
$h^\tau = k$, $k^\tau = h$, and $x^\tau = x$. 
So $T$ satisfies $T^{\tau \otimes \sharp} = \1_2 - T$. Then for $n$ large enough we obtain
$$T'_n =  \begin{pmatrix} \1- h_n' & (x'_n)^* \\ x_n' & k_n' \end{pmatrix}$$
where $(h_n')^\tau = k_n', (k_n')^\tau = h_n',$ and $(x_n')^\tau = x_n'$ and we apply Proposition~\ref{qCuniversal1} as before.

Finally, to show that $A_4$ is semiprojective we make use of the isomorphism $A_4 \cong A_0 \otimes \H$. Since $A_0$ is semiprojective, Proposition~\ref{H1semiproj} implies that $A_4$ is semiprojective.
\end{proof}

\begin{prop} \label{Aoddsemiproj}
$A_i$ is semiprojective for $i$ odd.
\end{prop}

\begin{proof}
For $n = 1$ and $n = -1$, this is Example~3.10 and Corollary~3.12 of \cite{loringsorensen}.
Then the cases $n = 3$ and $n = 5$ follow by Proposition~\ref{H1semiproj}.
\end{proof}

\section{Unsuspended $E$-theory for real $C^*$-algebras } \label{sec:Ethy}

In this section, we develop the theory of homotopy symmetric real \calg s, along the lines of \cite{dadarlatloring94} in the complex case. Our main theoretical result states (as in the complex case) that a real \calg~ $A$ is homotopy symmetric if and only if the usual natural homomorphism 
$[[A \otimes \K\pr, B \otimes \K\pr]]_\mathcal{R} \rightarrow E(A, B)$ is an isomorphism for all real \calg s $B$. Furthermore, we prove that homotopy symmetry has permanence with respect to complexification: a real \calg~ $A$ is homotopy symmetric if and only if $A\sc$ is homotopy symmetric (in the category of \calg s). It will follow that all of the algebras $A_i$ introduced in the previous section are homotopy symmetric. We introduce a standing assumption in this section that all real \calg s are separable. This will apply to all of our discussion of $E$-theory and of homotopy symmetry. However our main result Theorem~\ref{thm:classify} will be proven in full generality for all real \calg s.

We refer the reader to Section~4 of \cite{BLR} and Section~8 of \cite{BRS} for the development of asymptotic morphisms for real \calg s. In what follows we will use the notation
$[[A, B]]_\mathcal{R}$ to denote the homotopy classes of asymptotic morphisms in the category of real \calg s and $[[A, B]]_\mathcal{C}$ to denote the same in the category of complex \calg s, unless the meaning is clear from context. In both cases, this set has the structure of a semigroup if $B$ is stable. And in both cases, as we shall see, the property of homotopy symmetry is connected to the question of whether or not this semigroup has inverses.

Let $e$ be a rank 1 projection in $\K\pr \subset \K$. Then ${\id}_A(a) = a \otimes e$ defines a homomorphism, either $A \rightarrow A \otimes\sr \K\pr$ in the category of real \calg s or $A \rightarrow A \otimes \K$ in the category of complex \calg s. If $A$ and $B$ are real \calg s, then complexification induces a natural semigroup homomorphism
$$\theta_{A, B} \colon [[A, B \otimes\sr \K\pr]]_\mathcal{R} 
      \rightarrow[[A\sc, B\sc \otimes\sc \K]]_\mathcal{C} \; .$$
In particular, we have
$\theta_{A, A}({\id}_A) = {\id}_{A\sc}$.

\begin{defn} 
[See Section~5 of \cite{dadarlatloring94}] A \calg~ $A$ is {\it homotopy symmetric} if the class $[[{\id}_{A}]]$ is invertible in $[[A, A \otimes\sc \K]]_\mathcal{C}$. A real \calg~ $A$ is {\it homotopy symmetric} if the class $[[{\id}_A]]$ is invertible
in $[[A, A \otimes\sr \K\pr]]_\mathcal{R}$.
\end{defn}

\begin{lemma} \label{lem:group}
Suppose that $A$ and $B$ are real stable \calg~  with $A$ homotopy symmetric.  Then $[[A, B]]_\mathcal{R}$ is a group. In particular the asymptotic morphism $\eta_A$ that is inverse to ${\id}_A$ is unique up to homotopy.
\end{lemma}

\begin{proof}
Suppose that $\eta_A$ is an asymptotic morphism such that $[\eta_A]$ is inverse to $[{\id}_A]$. Then ${\id}_A \oplus \eta_A$ is null-homotopic in $[[A, A]]_\mathcal{R}$, and it follows that 
$[\psi \circ \eta_A]$ is an inverse to $[\psi]$ in $[[A, B]]_\mathcal{R}$.
\end{proof}

\begin{lemma} \label{lem:1}
Let $A,B$, and $D$ be real \calg s and let $\alpha \colon A \rightarrow B$ be a homomorphism.  Then  
$$[[D, S A]]_\mathcal{R} \xrightarrow{(S \alpha)_*} [[D, S B]]_\mathcal{R} \xrightarrow{\partial} [[D, C_\alpha]]_\mathcal{R}
    \xrightarrow{\kappa_*} [[D, A]]_\mathcal{R} \xrightarrow{\alpha*} [[D, B]]_\mathcal{R}$$
is an exact sequence where $C_\alpha$ is the mapping cone of $\alpha$ and $\kappa \colon C_\alpha \rightarrow A$ is defined by $\kappa(a, f) = a$.
\end{lemma}

\begin{proof}
As in Proposition 6 of \cite{dadarlat94}. 
\end{proof}

\begin{lemma} \label{lem:splitexact}
For any split exact sequence
$$0 \rightarrow J \rightarrow A \xrightarrow{\pi}  B \rightarrow 0 \; $$
of real \calg s
and any real \calg~ $D$,
the exact sequence
$$0 \rightarrow [[D, J]] \rightarrow [[D,A]] \xrightarrow{\pi} [[D,B]] \rightarrow 0 \; $$
is split.
\end{lemma}

\begin{proof}
As in Proposition~3.2 of \cite{dadarlatloring94}.
\end{proof}

Recall that if $A$ and $B$ are stable \calg s, then we have natural homomorphisms
$\Sigma \colon [[A, B]] \rightarrow [[SA, SB]]$
and
$\Sigma^{-1} \colon [[A, B]] \rightarrow [[S^{-1}A, S^{-1}B]] \; .$

\begin{lemma} \label{lemma:suspend-iso}
If $A$ and $B$ are real \calg s and $B$ is stable, then
$$
\Sigma \colon [[SA, SB]] \rightarrow [[S^2 A, S^2 B]] \quad \text{and} \quad
\Sigma^{-1} \colon [[SA, SB]] \rightarrow [[S^{-1} SA, S^{-1} SB]]$$
are isomorphisms.
\end{lemma}

\begin{proof}
The first statement is Lemma~4.5 of \cite{BLR} and the second statement can be proven in a similar way. Instead of using the elements in $E(\R, S^{8} \R)$ and $E(S^8 \R, \R)$ associated with the Bott isomorphism, we use elements in $E(\R, S^{-1} S\R)$ and $E(S^{-1} S \R, \R)$ that are inverses to each other arising from the $KK$-equivalence between $\R$ and $S^{-1} S \R$.
\end{proof}

The following definition is from Section~4 of \cite{BLR}. 

\begin{defn}
Let $A$ and $B$ be real separable \calg s. Then we define $$E(A,B) = [[SA, SB \otimes \K\pr]] \; .$$
\end{defn}

\begin{lemma} \label{lemma:alpha0} 
There exists an asymptotic morphism $\alpha_t \colon SS^{-1} \R \rightarrow \K\pr$ such that
$\alpha_* \colon KO_0(SS^{-1} \R) \rightarrow KO_0(\K\pr)$ is an isomorphism. Thus $\alpha_*$ is an isomorphism on $K\crt(-)$.
\end{lemma}

\begin{proof}
Note that $KO_0(SS^{-1} \R) \cong KO_0(\K\pr) \cong \Z$. In fact, 
$K\crt(SS^{-1} \R) \cong K\crt(\K\pr)$ 
is isomorphic to the free \ct-module with a generator in the real part in degree 0. So the Universal Coefficient Theorem for real \calg s implies that
\begin{align*}
KKO(SS^{-1} \R, \K\pr) \cong KKO(\R, \R) &\cong {\hom}\scrt(K\crt(\R), K\crt(\R))  \\
  & \cong {\hom}_{\Z}(KO_0(\R), KO_0(\R)) \cong \Z \; .
\end{align*}

As in the remarks preceding Theorem~5.2 of \cite{BLR}, the isomorphism
$$KKO(SS^{-1} \R, \K\pr) \rightarrow {\hom}_{\Z}(KO_0(\R), KO_0(\R))$$
factors through $E(S^{-1} \R, \K\pr) \cong [[S S^{-1} \R, \K\pr]]$,
giving the existence of $\alpha$ as desired.
\end{proof}

We now pause to establish some notation and to define several more homomorphisms that we will make use of for the rest of this section. We will use $\zeta$ to denote an involution on either $\R^n$ or on a sphere $S^{n-1}$, given by multiplication by $-1$ in exactly one coordinate (let us take it to be the $y$-coordinate).
Then for example, 
$$C_0(\R^2; \zeta) = \{ f \in C_0(\R^2, \C) \mid f(x,y) = \overline{f(x,-y)} \}
  \cong S S^{-1} \R \; . $$
More generally, $C_0(\R^n; \zeta) \cong S^{n-1} S^{-1} \R$.
There is a split exact sequence
$$0 \rightarrow C_0(\R^n; \zeta) \xrightarrow{i} C_0(S^n; \zeta)
  \xrightarrow{\ve} \R \rightarrow 0 $$
where $i$ is the standard inclusion via stereographic projection and $\ve$ is evaluation at any point fixed by $\zeta$.

Now consider the projection
$$p_0(x,y,z) = \tfrac{1}{2} \sm{1+z}{x-iy}{x+iy}{1-z} \;, $$
in $C(S^2, \C)$. 
We know that $p_0$ satisfies 
$[p_0] = (1,1) \in KO_0(C(S^2,\C)) \cong \Z \oplus \Z$
(see Example~6.2.3 of \cite{rosenbergbook}).
Since $p_0^{\zeta \otimes \tr} = p_0$, it follows that $[p_0]$ is an element in
$KO_0(C(S^2; \zeta)) \cong \Z \oplus \Z$. Since the complexification functor
$$c \colon KO_0(A) \rightarrow K_0(A\sc)$$
is known to be an isomorphism in this case where $A = C(S^2; \zeta)$, we conclude that
$$[p_0]= (1,1) \in KO_0(C(S^2; \zeta))$$ 
in the usual identification of $KO_0(C(S^2; \zeta))\cong \Z \oplus \Z$.
More precisely, this means that
 $\ve_*([p_0])$ is a generator of $KO_0(\R) \cong \Z$; and 
that $[p_0] - [\sm{0}{0}{0}{1}]$ is a generator of 
$KO_0(C(\R^2; \zeta)) \cong \ker(\ve_*) \cong \Z$. 
For future reference, we can take $\ve$ to be evaluation at the point $(0,0, -1)$ and we obtain the exact formula
$$\ve(p_0) = \begin{pmatrix} 0 & 0 \\ 0 & 1 \end{pmatrix} \; .$$

We define a *-homomorphism $\gamma_1$ and an asymptotic morphism $\gamma_2$ by
\begin{align*}
\gamma_1 &\colon A \rightarrow A \otimes C(S^2; \zeta) \otimes M_2(\R) & {\text{~by~}}
   \gamma_1(a) &= a \otimes p_0 \\
\gamma_2 &\colon A \rightarrow A \otimes C(S^2; \zeta)  &{\text{~by~}}
   \gamma_2(a) &= \eta_t(a) \otimes 1 \\
\end{align*}
For later reference we note that we have
$$(({\id}_A \otimes \ve \otimes {\id}_{M_2(\R)}) \circ \gamma_1)(a) 
  = a \otimes \sm{0}{0}{0}{1} 
  = \sm{0}{0}{0}{a}  
  \in A \otimes M_2(\R) \; .$$

\begin{prop}  \label{prop:betaA}
Suppose that $A$ is stable and homotopy symmetric. There exists an asymptotic morphism
$$\beta^A_t \colon A \rightarrow A \otimes C(S^2; \zeta) \otimes M_3(\R)$$
unique up to homotopy, so that the diagram 
\begin{equation}  \label{diagram:beta1}
\xymatrix{
A \ar[rr]^-{\beta^A} \ar[drr]_-{\gamma_1 \oplus \gamma_2}
&& A \otimes C_0(\R^2; \zeta) \otimes M_3(\R) \ar[d]^{i} \\
&& A \otimes C(S^2; \zeta) \otimes M_3(\R) 
} 
\end{equation}
commutes up to homotopy. 
Furthermore,  
$\Sigma \beta^A$ is homotopic to ${\id}_A \otimes \beta^{S\R}$ as asymptotic morphisms
from $SA$ to $SA \otimes C_0(\R^2; \zeta) \otimes M_3(\R)$,
and $\beta^{S\R}$ is an isomorphism on $K$-theory.
\end{prop}

\begin{proof}
Composing $\ve$ and $\gamma_1 \oplus \gamma_2$ we have 
$$({\id}_A \otimes \ve \otimes {\id}_{M_2(\R)})(\gamma_1 \oplus \gamma_2)(a) = 
\begin{pmatrix} 0 & 0 & 0 \\ 0 & a & 0 \\ 0 & 0 & \eta_t(a) \end{pmatrix} $$
where $\eta_t$ is the asymptotic inverse to ${\id}_A$. Thus this composition is
null-homotopic. So from the split exact sequence
$$0 \rightarrow C_0(\R^2; \zeta) \xrightarrow{i} C(S^2; \zeta) \xrightarrow\ve \R \rightarrow 0 \; $$
(or rather from the split exact sequence obtained by tensoring the above with $A \otimes M_3(\R)$),
Lemma~\ref{lem:splitexact} implies that there is a unique asymptotic morphism $\beta^A$ making 
Diagram~(\ref{diagram:beta1}) commute.

Taking the special case $A = S\R$, we obtain the diagram
\begin{equation} \label{diagram:beta0} \xymatrix{
S\R \ar[rr]^-{\beta^{S\R}} \ar[drr]_-{\gamma_1 \oplus \gamma_2}
&& S\R \otimes C_0(\R^2; \zeta) \otimes M_3(\R) \ar[d]^{i} \\
&& S\R \otimes C(S^2; \zeta) \otimes M_3(\R) 
} \end{equation}
Now, we construct two diagrams that both look like
$$
\xymatrix{
S A \ar[rr]^-{\beta^{S A}} \ar[drr]_-{\gamma_1 \oplus \gamma_2}
&& SA \otimes C_0(\R^2; \zeta) \otimes M_3(\R) \ar[d]^{i} \\
&& SA \otimes C(S^2; \zeta) \otimes M_3(\R) }
$$
by either suspending Diagram~(\ref{diagram:beta1}) or by tensoring Diagram~(\ref{diagram:beta0}) by $A$. 
In these two diagrams, the homomorphisms $i$ and $\gamma_1$ 
are exactly the same and the homomorphism $\gamma_2$ is the same up to homotopy in since 
$[[\eta_{SA}]] = [[\eta_{S\R} \otimes {\id}_A]] = [[{\id}_{S\R} \otimes \eta_A]]$ 
(using Lemma~\ref{lem:group}). 
Therefore, by uniqueness of $\beta^{SA}$ we have, 
$[[\beta^{SA}]] = [[{\id}_A \otimes \beta^{S\R}]] = [[{\id}_{S\R} \otimes \beta^A]]$.

To prove the statement about $K$-theory, we can calculate the action of $\gamma_1$ and $\gamma_2$ on $KO_{-1}(S\R)$ as in Diagram~(\ref{diagram:beta0}).
We can write 
$$\gamma_1^{S\R} = {\id}_{S\R} \otimes \gamma_1^\R \colon S\R \otimes \R \rightarrow
	  S\R \otimes C(S^2; \zeta) \otimes M_2(\R)$$
and see that $(\gamma_1^{S\R})_*$ maps the generator of $KO_{-1}(S\R) \cong \Z$ to the class corresponding to $[p_0]$ in 
$KO_{-1}(S\R \otimes C(S^2; \zeta)) \cong KO_{0}(C(S^2; \zeta)) \cong \Z \oplus \Z$. At the same time,
$(\gamma_2^{S\R})_*$ maps the generator of $KO_{-1}(S\R)$ to the additive inverse of the class 
representing the unit in the same group. Thus, we see that $(\gamma_1 \oplus \gamma_2)_*$ maps the generator of $KO_{-1}(S\R)$ to the kernel of $\ve_*$ (which we already knew) and to the generator 
of ${\im}(i_*) \cong KO_{-1}(S\R \otimes C_0(\R^2; \zeta))
    \cong KO_0(C_0(\R^2; \zeta)) \cong \Z$.
This proves that $\beta^{S\R}$ is an isomorphism on $KO_{-1}(-)$ and hence on $K\crt(-)$.
\end{proof}

\begin{thm} \label{thm:unsuspendE}
Let $A$ be a stable homotopy symmetric real \calg . Then 
$$\Sigma \colon [[A, B]] \rightarrow E(A,B)$$ is an isomorphism
for all real stable \calg s $B$.
\end{thm}

\begin{proof}
From Lemma~\ref{lemma:suspend-iso} 
use the isomorphism $E(A, B) \cong [[SS^{-1} A, SS^{-1} B]]$ to
show that 
$$\Sigma \, \Sigma^{-1} \colon [[A, B]] \rightarrow [[SS^{-1} A, SS^{-1} B]] \; $$
is an isomorphism with inverse
$$\Theta \colon [[SS^{-1} A, SS^{-1} B]] \rightarrow [[A, B]]$$
defined by 
$$\Theta([[\varphi]]) = [[{\id}_B \otimes \alpha \otimes {\id}_{M_2(\R)}]]
    \circ [[\varphi \otimes M_3(\R)]] \circ [[\beta^A]] \; .$$

By the Yoneda Lemma, it suffices to consider the case $A = B$ and to then show that 
${\id}_A$ maps to ${\id}_A$ under the homomorphism
$$\Theta \circ \Sigma \, \Sigma^{-1} \colon [[A, A]] \rightarrow [[A, A]] \; .$$ 
We have $\Sigma \, \Sigma^{-1} ({\id}_A) = {\id}_{S S^{-1} \R}$ and we have
$\Theta({\id}_{S S^{-1} \R}) 
    = ({\id}_A \otimes \alpha \otimes {\id}_{M_3(\R)}) \circ \beta^A$.
So we need to show that 
$$({\id}_A \otimes \alpha \otimes {\id}_{M_3(\R)}) \circ \beta^A \colon A \rightarrow A \otimes \K\pr \otimes M_3(\R)$$
is homotopic to ${\id}_A$ as an asymptotic morphism.
For this we use the commutative diagram
$$ \xymatrix{
A \ar[rr]^-{\beta^{A}} \ar[drr]_-{\gamma_1 \oplus \gamma_2}
&& A \otimes C_0(\R^2; \zeta) \otimes M_3(\R) \ar[rrr]^{{\id} \otimes \alpha \otimes {\id}}   \ar[d]^{i}
&&& A \otimes \K\pr \otimes M_3(\R) \ar[d]^{i} \\
&& A \otimes C(S^2; \zeta) \otimes M_3(\R) \ar[rrr]^{{\id} \otimes \widetilde{\alpha} \otimes {\id}} 
&&& A \otimes \widetilde{\K}\pr \otimes M_3(\R)
}$$
By Lemma~\ref{lem:splitexact}, the homomorphism
$$[[A, A \otimes \K\pr \otimes M_3(\R)]] \xrightarrow{i_*}
      [[A, A \otimes \widetilde{\K}\pr \otimes M_3(\R)]] $$
is injective, so (in yet another reduction) it suffices to show that
$$( {\id}_A \otimes \widetilde{\alpha} \otimes {\id}_{M_3(\R)} ) \circ (\gamma_1 \oplus \gamma_2)
 \colon A \rightarrow A \otimes \widetilde{\K}\pr \otimes M_3(\R)$$
is homotopic to $i \circ {\id}_A$.

For any projection $p$ in a real \calg~ $B$, let $j_{p} \colon \R \rightarrow B$ be the homomorphism given by
$j_{p}(t) = t p$.
Let $q_t = \widetilde{\alpha_t}(p_0) \in \widetilde{\K}\pr \otimes M_2(\R)$. Since $q_t$ is asymptotically a projection, there exists an actual projection $q_0\in \widetilde{\K}\pr \otimes M_2(\R)$ such that 
$(\alpha_t) \circ j_{p_0}$ is homotopic to $j_{q_0}$. Furthermore,
by calculating the class $\alpha_*([p_0]) = [q_0] \in KO_0(\K\pr)$, we know that
$q_0$ is homotopic to the projection 
$q_0' = \sm{e}{0}{0}{1}$
where $e$ is a rank one projection in $\K\pr$.
So we can and do assume that $q_0 = q_0'$.
Then up to a homotopy of asymptotic morphisms we have
$$({\id}_{A} \otimes \widetilde{\alpha} \otimes {\id}_{M_3(\R)}) \circ (\gamma_1 \oplus \gamma_2)(a) 
  = \begin{pmatrix} a \otimes e & 0 & 0 \\ 0 & a & 0 \\ 0 & 0 & \eta_A(a) \end{pmatrix}$$
and thus $({\id}_{A} \otimes \widetilde{\alpha} \otimes {\id}_{M_3(\R)}) \circ (\gamma_1 \oplus \gamma_2)$
is homotopic to ${\id}_A$.

For the other direction, again by the Yoneda Lemma it suffices to compute 
$\Sigma\, \Sigma^{-1} \circ \Theta$ applied to ${\id}_A \otimes {\id}_{S S^{-1} \R}$. But 
we have just seen that
$\Theta({\id}_{S S^{-1} \R}) = {\id}_A$ and that $\Sigma\, \Sigma^{-1} ({\id}_A) = {\id}_{S S^{-1} \R}$,
which completes the proof.
\end{proof}

\begin{prop} \label{thm:homsym}
A real \calg~ $A$ is homotopy symmetric if and only if the complexification $A\sc$ is homotopy symmetric.
\end{prop}

\begin{proof}
We assume that $A\sc$ is homotopy symmetric (in the category of \calg s), 
so there is an asymptotic morphism
$\eta_A \in [[A\sc, A\sc \otimes\sc \K]]_\mathcal{C}$ such that ${\id}_{A\sc} \oplus \eta_A$ is null-homotopic through asymptotic morphisms of complex \calg s.
Let $c \colon A \rightarrow A\sc$ be the standard inclusion and let $r$ be the homomorphism 
$$r \colon A\sc \otimes\sc \K \rightarrow A \otimes\sr M_2(\R) \otimes\sr \K\pr \; .$$ 
Notice that $r \circ {\id}_{A\sc} \circ c = {\id}_A \oplus {\id}_A$, since for any $a \in A$ we have 
$$(r \circ {\id}_{A\sc}  \circ c)(a) 
  = \left( \begin{matrix} a \otimes e & 0 \\ 0 & a \otimes e \end{matrix} \right) \; .$$

By hypothesis, then, the composition $r \circ ({\id}_{A\sc} \oplus \eta_A) \circ c$ is null-homotopic.  On the other
hand, we have
\begin{align*}
r \circ ({\id}_{A\sc} \oplus \eta_A) \circ c
  &= (r \circ {\id}_{A\sc} \circ c) \oplus (r \circ \eta_A \circ c) \\
  &= {\id}_{A} \oplus {\id}_{A} \oplus (r \circ \eta_A \circ c)
\end{align*}
which shows that $[[{\id}_{A} \oplus (r \circ \eta_A \circ c)]]$ is an inverse for $[[{\id}_A]]$
in $[[A, A \otimes \K\pr]]_{\mathcal{R}}$.

For the other direction, we have a semigroup homomorphism 
$$\theta_{A,A} \colon [[A, A \otimes\sr \K\pr]]_\mathcal{R} 
      \rightarrow [[A\sc, A\sc \otimes\sc \K]]_\mathcal{C} \; . $$ 
So if $[[{\id}_A]]$ is invertible in the former, then it immediately follows that $\theta_{A,A}([[{\id}_A]]) = [[{\id}_{A\sc}]]$ is invertible in the latter.
\end{proof}

\begin{cor}
The real \calg s $A_i$ are homotopy symmetric for all $0 \leq i < 8$.
\end{cor}

\begin{proof}
For all $i$, we have that $(A_i)\sc$ is isomorphic to one of the following \calg s: 
$q\C, q\C \otimes M_2(\C), S\C$, and $S\C \otimes M_2(\C)$. 
From the comments at the beginning of Section~5 of \cite{dadarlatloring94} we know that $q\C$ and $S\C$ are homotopy symmetric; and from Lemma~5.1 of \cite{dadarlatloring94} we know that $q\C \otimes M_n(\C)$ and $S \C \otimes M_n(\C)$ are homotopy symmetric. Therefore Proposition~\ref{thm:homsym} implies that $A_i$ is homotopy symmetric for all $i$.
\end{proof}

\begin{lemma} \label{lemma:semipro}
Let $D$ and $B$ be real \calg s, with $D$ semiprojective.  Then 
\begin{enumerate}
\item[(1)] $[D, B] \cong [[D, B]]$
\item[(2)] If $B = \lim_{n \to \infty} B_n$, then $[D, B] \cong \lim_{n \to \infty} [D, B_n]$.
\end{enumerate}
\end{lemma}

\begin{proof}
Both of these results have proofs that carry over directly to the real case from the complex case. The proofs in the complex case are found at the beginning of Section~6 of \cite{dadarlatloring94} and as the proof to Corollary~15.1.3 of \cite{loringbook}, respectively.
\end{proof}

\begin{thm} \label{thm:classify}
For each integer $i$ in the range $0 \leq i < 8$ and for any real \calg~ $B$ (not necessarily separable), there is a natural isomorphism
$$KO_i(B) \cong [A_i, \K\pr \otimes B] \cong \lim_{n \to \infty} [A_i, M_n(B)] \; .$$
If $B$ is stable, then
$$KO_i(B) \cong [A_i, B] \; .$$
\end{thm}

\begin{proof}
First consider the case that $B$ is separable. From Propositions~\ref{thm:KAieven} and \ref{thm:KAiodd}
we have $K\crt(A_i) \cong \Sigma^{-i} K\crt(\R)$ for all $i$, so the Universal Coefficient Theorem
(Corollary~4.11 of \cite{boersema04}) implies that $A_i$ is $KK$-equivalent to $S^{-i} \R$. We note that the condition for the Universal Coefficient Theorem to apply is that the complexification of $A_i$ is in the bootstrap category of separable nuclear \calg s. This is easy to check since the complexifications of these algebras are all stably isomorphic to a commutative \calg~ or to $q\C$.
Furthermore, each $A_i$ is semiprojective by Propositions~\ref{Aevensemiproj} and \ref{Aoddsemiproj}.
Therefore,
\begin{align*}
KO_i(B) &\cong KKO(S^{-i}\R, B)  \\
  &\cong KKO(A_i, B)   \\
  &\cong E(A_i, B) && \text{by Theorem~4.6 of \cite{BLR}} \\
  &\cong [[A_i, \K\pr \otimes B]]  && \text{by Theorem~\ref{thm:unsuspendE}} \\
  &\cong [A_i, \K\pr \otimes B]  && \text{by Lemma~\ref{lemma:semipro}(1)} \\
  &\cong \lim_{n \to \infty} [A_i, M_n(B)]  && \text{by Lemma~\ref{lemma:semipro}(2).}
\end{align*}

To address the general case, let 
$F_i(B) = \lim_{n \to \infty} [A_i, M_n(B)]$
and consider the natural homomorphism
$$\alpha_B \colon F_i(B) \rightarrow KO_i(B) \; $$ 
defined by $\alpha_B([\phi]) = \phi_*(\xi_i)$ where $\xi_i$ is a generator of $KO_i(A_i) \cong \Z$.
This homomorphism exists for all real \calg s and is an isomorphism when $B$ is separable.
In general, write $B$ as the inductive limit $B = \lim_{\lambda} B_\lambda$ where
$\{B_\lambda \}$ is the net of all separable subalgebras of $B$. We leave it to the reader to verify that $F_i$ is continuous with respect to inductive limits, using the fact that $A_i$ is semiprojective.
Then since both
functors $F_i$ and $KO_i$ are continuous with respect to inductive limits and since $\alpha_{B_\lambda}$ is an isomorphism
for all $\lambda$, it follows that $\alpha_B$ is an isomorphism.
\end{proof}

\section{$K$-theory via unitaries -- the even cases} \label{sec:even}

In the next two sections, we develop pictures of all eight fundamental $KO$-groups in terms of unitaries.
We use the notation $KO_i^u$ for these functors defined on the category of real \calg s or (equivalently) the category of \ctalg s.
For $i = 0,1$, we will write down the definition both in terms of a real \calg~ $A$ and in terms of a \ctalg ~$(A, \tau)$. However for $i \neq 0,1$, we will only consider $KO^u_i(A, \tau)$ in the context of a \ctalg $(A, \tau)$, since that picture gives the most direct and consistent definitions for varying values of $i$. For a real \calg~ $A$, one should consider the associated \ctalg ~$(A\sc, \tau)$. Thus $KO_i^u(A) = KO_i^u(A\sc, \tau)$.

In each case, we have a picture in terms of unitaries in matrix algebras over $A$ satisfying certain relations. In each case, we will prove that our picture is a well-defined group and then prove that it is naturally isomorphic to the standard version of $K$-theory. A reader who wishes to skip our detailed development can see the final pictures summarized in Section~\ref{sec:summary}, where we  also include a description of complex $K$-theory, $KU_i(A, \tau)$.

\subsection{$KO_0$ via unitaries}

\begin{defn} \label{defn:KO_0}
Let $A$ be a unital real \calg . Let $U_{\infty}^{(0)}(A)$ be the set of all unitaries $u$ in $\cup_{n \in \N} M_{2n}(A)$ satisfying $u^2 = 1$ (equivalently, unitaries $u$ that satisfy $u = u^*$). Let $\sim_0$ be the equivalence relation on $U^{(0)}_{\infty}(A)$, generated by 
\begin{enumerate}
\item[(1)] $u_0 \sim_0 u_1$ if $u_t \in M_{2n}(A)$ is a continuous path of self-adjoint unitaries on $[0,1]$; 
and 
\item[(2)] $u \sim_0 \iota^{(0)}_n(u)$ for $u \in M_{2n}(A)$ where 
$\iota_n^{(0)} \colon M_{2n}(A) \rightarrow M_{2n+2}(A)$ is given by
$$\iota_n^{(0)}(a) 
  = {{\diag}} \left(a, I^{(0)} \right) \; \quad \text{where $I^{(0)} = \sm{1}{0}{0}{-1}$} \;. $$
\end{enumerate}
Then we define $KO^u_0(A) = U_{\infty}^{(0)}(A)/\sim_0$, with a binary operation given by $[u] + [v] = \left[ \sm{u}{0}{0}{v} \right]$
for $u,v \in U_{\infty}^{(0)}(A)$.
\end{defn}

In addition to the notation for $I^{(0)}$ given above, we will use
$$I_n^{(0)} = {{\diag}}(I^{(0)}, I^{(0)}, \dots, I^{(0)}) \in M_{2n}(\C) \; .$$

\begin{prop} \label{K0group}
$KO_0^u(A)$ is a homotopy invariant functor from the category of all unital \calg s to the category of abelian groups. The inverse of an element $[u] \in KO_0^u(A)$ is $[-u]$.
\end{prop}

\begin{proof}
The binary operation on $KO^u_0(A)$ is clearly associative and by definition the element $I^{(0)}$ represents the identity.

Consider self-adjoint unitaries $u \in M_{2n}(A)$ and $v \in M_{2m}(A)$. Let $w$ be a unitary in $M_{2n+2m}(\R)$ such that 
$w \sm{u}{0}{0}{v} w^* = \sm{v}{0}{0}{u}$ and $\det w = 1$. This can be done by taking $w$ to be a change of basis matrix corresponding to a particular even permutation of the basis elements. There is then a path of unitaries $w_t \in M_{2n+2m}(\R)$ such that $w_0 = 1_{2n+2m}$ and $w_1 = w$. Then $w_t \cdot \sm{u}{0}{0}{v} \cdot w_t^*$ is a self-adjoint unitary for all $t$ showing that $[u] + [v] = [v] + [u]$.

Let $u \in M_{2n}(A)$ be a self-adjoint unitary. 
Let 
$$r_t = 
\begin{pmatrix}{\cos((\pi/2) t) \cdot 1_{2n}}&{-\sin((\pi/2) t) \cdot 1_{2n}}\\{\sin((\pi/2) t) \cdot 1_{2n}}&{\cos((\pi/2) t)\cdot 1_{2n}} 
	\end{pmatrix}\; $$
and let $u_t = r_t \cdot
\sm{u}{0}{0}{1} \cdot r_t^*$ be the path
from
$\sm{u}{0}{0}{1_{2n}}$ to $\sm{1_{2n}}{0}{0}{u}$. Then using the relation $u^2 = 1_{2n}$, one can show that $u_t$ commutes with $\sm{1_{2n}}{0}{0}{-u}$. Hence
$u_t \cdot \sm{1_{2n}}{0}{0}{-u}$ gives a path in $U_{\infty}^{(0)}(A)$ from $\sm{u}{0}{0}{-u}$ to
$\sm{1_{2n}}{0}{0}{-u u} = \sm{1_{2n}}{0}{0}{-1_{2n}}$. This last matrix in $M_{4n}(A)$ is equivalent to $I^{(0)}_{2n}$ representing the identity element in $KO_0^u(A)$, as shown in the proof of Proposition~\ref{K0ofR} below.
\end{proof}

\begin{prop} \label{K0ofR}
$KO_0^u(\R) \cong \Z$.
\end{prop}

\begin{proof}
Consider the map $\phi \colon U_{\infty}^{(0)}(\k) \rightarrow \Z$ given by 
$\phi(u) = \tfrac{1}{2} {\trace}(u)$. If $u$ is a self-adjoint unitary, then $u$ is unitarily equivalent to a diagonal matrix with eigenvalues in $\{1, -1\}$. It follows that the range of $\phi$ is exactly $\Z$. Furthermore, since $\phi$ is continuous, is invariant under unitary equivalence, and satisfies $\phi(u) = \phi(\iota^{(0)}_n(u))$; it follows that $\phi$ is well-defined on $KO_0^u(\R)$.

Suppose now that $u$ and $v$ are self-adjoint unitaries with entries in $\R$ having the same trace. We may assume that $u$ and $v$ have the same dimension by perhaps replacing $u$ with $\iota^{(0)}_n(u)$ or replacing $v$ with $\iota^{(0)}_n(v)$. Let $u = xvx^*$ where $x$ is a unitary in $M_{2n}(\R)$. We can assume that $x$ is in the same component as the identity among unitaries in $M_{2n}(\R)$. For if $\det x = -1$, we can replace $u, v, x$ by $\iota_n^{(0)}(u), \iota_n^{(0)}(v), \iota_n^{(0)}(x)$ in $M_{2n+2}(\R)$ and note that 
$\det \iota_n^{(0)}(x) = - \det x$.

Now let $x_t$ be a path of unitaries from $1_{2n}$ to $x$. Then $u_t = x_t v x_t^*$ is  a path of self-adjoint unitaries from $v$ to $u$ showing that $[u] = [v]$. The result follows.
\end{proof}

\begin{defn}
Let $A$ be any unital \calg . Then we define
$KO^u_0(A) = \ker(\lambda_*)$ where $\lambda \colon \widetilde{A} \rightarrow \R$ is the natural projection from the unitization of $A$ with kernel isomorphic to $A$.
\end{defn}

We note that the formula in this definition is valid also in case $A$ is unital.
Therefore we have a picture in which any element of $KO_0^u(A)$ is represented by a self-adjoint unitary $u$ in $M_{2n}(\widetilde{A})$ such that ${\trace} (\lambda_{2n}(u)) = 0$. The following proposition makes clear the picture of $KO_0^u(A)$ that we are presenting.

\begin{prop} \label{K0unitpicture}
Let $A$ be a real \calg . Any element of $KO_0^u(A)$ can be represented as $[u]$ where $u \in M_{2n}(\widetilde{A})$ is a self-adjoint unitary satisfying 
$\lambda(u) = I^{(0)}_n$.
\end{prop}

\begin{proof}
Suppose $u$ is a self-adjoint unitary in $M_{2n}(\widetilde{A})$ and $\lambda_*([u]) = 0$ in $KO_0^u(\R)$. Then ${\trace}(\lambda(u)) = 0$. So there is a unitary $v \in M_{2n}(\R)$ such that 
$v \lambda(u) v^* = I^{(0)}_n$. Let $u' = v u v^*$, so that $\lambda(u') = I^{(0)}_n$. Furthermore, as in the proof of Proposition~\ref{K0ofR} we can choose $v$ so that $\det v = 1$ (possibly by increasing $n$).
Then there is a path of unitaries $v_t$ in $M_{2n}(\R)$ from $v$ to $\1_{2n}$; so $u_t = v_t u v_t^*$ is a path of self-adjoint unitaries from $u'$ to $u$ showing that $[u'] = [u]$. 
\end{proof}

\begin{thm} \label{K0iso}
Let $A$ be a real \calg . Then there is a natural isomorphism 
$\theta \colon KO_0^u(A) \rightarrow KO_0(A)$. The isomorphism $\theta$ is given by 
$$\theta([u]) = \left[\tfrac{1}{2}(u + \1_{2n}) \right] - [\1_n] \; $$
for any self-adjoint unitary $u \in M_{2n}(\widetilde{A})$.
\end{thm}

\begin{proof}
It suffices to consider the case where $A$ is unital and $u \in M_{2n}(A)$. 
The reader can check that if $u$ is a self-adjoint unitary, then $\tfrac{1}{2}(u + 1_{2n})$ is a projection in $M_{2n}(A)$, and that if $u_t$ is a path of self-adjoint unitaries in $M_{2n}(A)$, then $\tfrac{1}{2}(u_t + 1_{2n})$ is a path of projections in $M_{2n}(A)$. For $\theta$ to be well-defined, we also have  
\begin{align*}
\theta\left( \left[ \smd{u}{1}{-1} \right] \right) 
  &= \left[\tfrac{1}{2}  \smd{u+1_{2n}}{2 \cdot 1}{0}  \right] -[1_{n+1}] \\
  &= \left[ \smd{\tfrac{1}{2}(u + 1_{2n})}{1}{0} \right] - [1_{n+1}] \\
 &= \left[\tfrac{1}{2}(u + 1_{2n}) \right] - [1_n]  \\
  &= \theta([u]) \; .
\end{align*}

To show that $\theta$ is a group homomorphism, we check that for $u \in M_{2m}(A)$ and $v \in M_{2n}(A)$ we have
\begin{align*} \theta \left( \left[ \sm{u}{0}{0}{v} \right] \right)  
    &= \left[ \tfrac{1}{2}\left(\sm{u}{0}{0}{v} + 1_{2m+2n} \right) \right] - [1_{m+n}]  \\
    &= \left[\tfrac{1}{2}(u + 1_{2m}) \right] - [1_m] + \left[\tfrac{1}{2}(v + 1_{2n}) \right] - [1_n] \\
    &= \theta([u]) + \theta([v]).
\end{align*}

It remains to show that $\theta$ is a bijection.
To show that $\theta$ is onto it suffices to show that for any projection $p \in M_n(A)$, the element $[p] \in KO_0(A)$ is in the range of $\theta$. In fact, taking
$u = \sm{2p - 1_n}{0}{0}{1_n} \in M_{2n}(A)$, we have
$$\theta([u]) = \left[ \sm{p}{0}{0}{1_n} \right] - \left[ 1_n \right] = [p] \; .$$

To show that $\theta$ is one-to-one, suppose that $u$ and $v$ are unitaries such that $\theta([u]) = \theta([v])$. We can assume that $u,v \in M_{2n}(A)$ for some $n$. Then we have
$[\tfrac{1}{2}(u + 1_{2n})] = [\tfrac{1}{2}(v + 1_{2n})]$ in $KO_0(A)$. It follows that there is an integer $m$ such that the projections 
$p = \tfrac{1}{2}(u + 1_{2n}) \oplus 1_{m}$ and 
$q = \tfrac{1}{2}(v + 1_{2n}) \oplus 1_{m}$ are homotopic in $M_{2n+2m}(A)$. 
Up to a homotopy of projections in $M_{2n+2m}(A)$, we can now write $p$ and $q$ in the form
\begin{align*}
p &= \tfrac{1}{2} \left( {{\diag}}(u,1_m, -1_m) + 1_{2(n+m)} \right)    \\
 & \qquad \sim \tfrac{1}{2} \left( {{\diag}} \left(u,\sm{1}{0}{0}{-1}, \dots,\sm{1}{0}{0}{-1} \right) + 1_{2(n+m)} \right) \\ 
q &= \tfrac{1}{2} \left( {{\diag}}(v,1_m, -1_m) + 1_{2(n+m)} \right)    \\
 & \qquad \sim \tfrac{1}{2} \left( {{\diag}} \left(v,\sm{1}{0}{0}{-\1}, \dots,\sm{1}{0}{0}{-1} \right) + 1_{2(n+m)} \right).
\end{align*} 
Then $2p - 1_{2(n+m)}$ and $2q - 1_{2(n+m)}$ are homotopic through unitaries that satisfy $u^2 = u$, so it follows that 
$${{\diag}} \left(u,\sm{1}{0}{0}{-1}, \dots,\sm{1}{0}{0}{-1} \right) \sim
{{\diag}} \left(v,\sm{1}{0}{0}{-1}, \dots,\sm{1}{0}{0}{-1} \right)$$ 
where there are $m$ copies of the block
$\sm{1}{0}{0}{-1}$ along the diagonal. Therefore $[u] = [v]$ in $KO_0^u(A)$. 
\end{proof}

For elements $h,k,x$ in a \calg~ $A$, 
recall from Section~\ref{sec:Ai} that 
\begin{align*}
T(h,x,k) &= \begin{pmatrix} \1-h & x^*\\ x&k \end{pmatrix} \\ \text{and~}
U(h,x,k) &= 2T - \1_2 = \begin{pmatrix} \1-2h & 2x^* \\ 2x & 2k-\1 \end{pmatrix} \; ,
\end{align*}
which are both elements in $M_2(\widetilde{A})$.
In particular, let
$$u_0 = U(h_0, x_0, k_0) 
  = \left(  \begin{matrix}
      \1-2t & 0 & 0 & 2 \sqrt{t-t^2} \\
      0 & \1 & 0 & 0 \\
      0 & 0 & -\1 & 0 \\
      2 \sqrt{t-t^2} & 0 & 0 & 2t-\1 \end{matrix} \right)
    \in M_2(\widetilde{A_0}) \; $$
where $A_0$ is defined as in Section~\ref{sec:Ai}.

\begin{prop} \label{prop:u0}
The class $[u_0]$ is a generator of $KO^u_0(A_i) \cong \Z$.
\end{prop}

\begin{proof}
Evidently, $u_0$ is a self-adjoint unitary.
We write
$$\widetilde{A_0} = \{ f \colon [0, 1] \rightarrow M_2(\R) \mid
  f(0) = \sm{t}{0}{0}{t}, f(1) = \sm{r}{0}{0}{s}, r,s,t \in \R\} \; $$
and the map $\lambda \colon \widetilde{A_0} \rightarrow \R$ coincides with evaluation at $0$. Since we have 
$\lambda_2(u_0) = I^{(0)}$ (where $\lambda_2 \colon M_2(\widetilde{A_0} ) \rightarrow M_2(\R)$), we know $[u_0]$ is a class in $KO_0^u(A_0)$.

Now, we have the map ${\ev}_1 \colon \widetilde{A_0} \rightarrow \R^2$ which is evaluation at $t = 1$ and we have $\pi_1 \colon \R^2 \rightarrow \R$ which is projection onto the first coordinate. It follows from the calculation of $KO_*(A_0)$ in Section~\ref{sec:Ai} that the map 
$\phi = \pi_1 \circ {\ev}_1 \colon \widetilde{A_0} \rightarrow \R$ 
is an isomorphism on standard $K$-theory and hence on $KO_0^u$. We calculate that $\phi(u_0) = -1_2$, which is a generator of $KO^u_0(\R) = \Z$.
\end{proof}

\begin{prop} \label{K0class}
For any real \calg~ $A$, the map $[\phi] \mapsto \phi_*[u_0]$ defines a natural isomorphism 
$$\Theta \colon [A_0, A \otimes\sr \K\pr] \rightarrow KO_0^u(A) \; .$$
\end{prop}

To be precise, the formula for $\Theta$ is
$$\Theta([\phi]) = (\phi_2)_*([u_0]) = [\phi_2(u_0)]$$
where $\phi_2(u_0) \in M_2(\widetilde{A})$.

We note that it already follows from Theorems~\ref{thm:classify} and~\ref{K0iso}
that the groups in question are isomorphic. However, we give a direct proof here since it establishes the concrete formula for the isomorphism and since it will serve as a model for the proof that $KO^u_2(A) \cong [A_2, A \otimes\sr \K\pr]$ in the next section.

\begin{proof}
Since $[A_0, \K\pr \otimes\sc A] \cong \lim_{n \to \infty} [A_0, M_n(A)]$ it suffices to define $\Theta$ for 
$\phi \colon A_0 \rightarrow M_n(A)$.
To show that the formula above gives a well-defined function $\Theta$, we first mention that $\phi_*([u_0])$ does not depend on the homotopy class of $\phi$. Here let us show more carefully that the homomorphisms $\phi \colon A_0 \rightarrow M_n(A)$ and  
$\phi' = \sm{\phi}{0}{0}{0} \colon A_0 \rightarrow M_{n+1}(A)$ will give the same element of $KO_0^u(A)$. We have
$$\phi_2(u_0) = U( \phi(h_0), \phi(x_0), \phi(k_0))
  = \begin{pmatrix} {\1_n - 2 \phi(h_0)}&{2 \phi(x_0)^*}\\{2 \phi(x_0)}&{2 \phi(k_0) -\1_n}
    \end{pmatrix}
$$
 and
\begin{align*}
\phi'_2(u_0) &= U(\phi'(h_0), \phi'(x_0), \phi'(k_0))   \\
  &= \left( \begin{matrix}
      ~\1_n - 2 \phi(h_0)~ & 0 & 2 \phi(x_0)^* & 0 \\
      0 & \1 & 0 & 0 \\
      2 \phi(x_0) & 0 & ~2 \phi(k_0) - \1_n~ & 0 \\
      0 & 0 & 0 & -\1 
      \end{matrix} \right)  
\end{align*}
showing that $[\phi_2(u_0)] = [\phi'_2(u_0)]$ in $KO^u_0(A)$, hence $\Theta([\phi]) = \Theta([\phi'])$. Therefore $\Theta$ is well-defined.

Now suppose that we have an element $[u] \in KO^u_0(A)$ where 
$u \in M_{2n}(\widetilde{A})$ is a unitary that satisfies $u^* = u$ and $\lambda_*([u]) = 0$ in $KO_0^u(\R)$. 
After conjugating by a unitary in $M_{2n}(\R)$ we can assume that $\lambda(u) = \sm{\1_n}{0}{0}{-\1_n}$ and there exists 
$h, k, x \in M_n(A)$ such that 
$$u = U(h, x, k) 
  = \begin{pmatrix}{\1_n - 2h}&{2x^*}\\{2x}&{2k - \1_n} \end{pmatrix} \; .$$
Since it is not guaranteed that $h$ and $k$ are orthogonal, we define the elements
\begin{equation} \label{h''} 
h' = \begin{pmatrix}{h}&{0}\\{0}&{0}\end{pmatrix}, \qquad 
k' = \begin{pmatrix}{0}&{0}\\{0}&{k}\end{pmatrix}, \qquad \text{and} \qquad 
x' = \begin{pmatrix}{0}&{0}\\{x}&{0}\end{pmatrix} \; 
\end{equation}
 in $M_{2n}(A)$.
Then $u$ represents the same element of $KO^u_0(A)$ as
the self-adjoint unitary
$$u' := U(h', x', k') = \left( \begin{matrix}
\1_n - 2h & 0 & 0 & 2x^* \\
0 & \1_n & 0 & 0 \\
0 & 0 & -\1_n & 0 \\
2 x & 0 & 0 & 2k - \1_n
\end{matrix} \right) \in M_{4n}(\widetilde{A}).$$
Since $h'$ and $k'$ are orthogonal and $U(h', x', k')$ is a unitary, by Proposition~\ref{qCuniversal1} there is a homomorphism
$\phi_u \colon A_0 \rightarrow M_{2n}(A)$ such that
$h_0$, $k_0,$ and $x_0$ map to $h'$, $k'$, and $x'$ respectively.
Then $\Theta([\phi_u]) = [\phi_u(U(h_0, x_0, k_0))] = [u'] = [u]$ showing that $\Theta$ is surjective.

In fact, we show that the construction of $\phi_u$ in the previous paragraph defines a homomorphism 
$\Phi$ from $KO^u_0(A)$ to $[A_0, \lim_{n \to \infty} M_n(A)]$; and that $\Phi$ is  inverse to $\Theta$.
First of all, if $u$ and $v$ are self-adjoint unitaries in $M_{2n}(\widetilde{A})$ that are homotopic through self-adjoint unitaries satisfying $\lambda(u_t) = \sm{\1_n}{0}{0}{-\1_n}$, then the construction in the previous paragraph results in a homotopy between $\phi_u$ and $\phi_v$.  Now let $u \in M_{2n}(\widetilde{A})$ and let 
$v = \left( \begin{smallmatrix} u & 0 & 0 \\ 0 & \1 & 0 \\ 0 & 0 & -\1 \end{smallmatrix} \right) \in M_{2n+2}(\widetilde{A})$. We show that $\phi_u$ and $\phi_v$ are equivalent elements of $[A_0, \K\pr \otimes\sr A]$. If we write 
$u= \sm{\1_n - 2h}{2x^*}{2x}{2k - \1_n}$
then we can write
$$v = \left( \begin{matrix} 
    \1_n - 2h & 2x^* & 0 & 0 \\
    2x & 2k-\1_n & 0 & 0 \\
    0 & 0 & \1 & 0 \\
    0 & 0 & 0 & -\1 
	  \end{matrix} \right) \; .$$
In order to have $\lambda(v)$ of the right form, we conjugate $v$ by a unitary in $M_{2n+2}(\R)$ and we write instead
$$v = \left( \begin{matrix} 
    \1_n - 2h & 0 & 0 & 2x^* \\
   0 & \1 & 0 & 0 \\
    0 & 0 & -\1 & 0 \\
    2x & 0 & 0 & 2k -\1_n
	  \end{matrix} \right)
= \left( \begin{matrix} 
    \1_{n+1} - 2 \begin{pmatrix} h & 0 \\ 0 & 0 \end{pmatrix} 
    & 2 \begin{pmatrix} {0}& {x^*}\\ {0}& {0} \end{pmatrix} \\
    2 \begin{pmatrix} {0}&{0}\\{x}&{0} \end{pmatrix} & 
    2 \begin{pmatrix} {0}&{0}\\{0}&{k} \end{pmatrix}  - \1_{n+1}
	  \end{matrix} \right) \; .$$
Hence, $\phi_v$ will map the elements $h_0, k_0$, and $x_0$, respectively, to the elements
$$h_v' = \smiiii{h}{0}{0}{0} ,
\quad k_v' = \smiiii{0}{0}{0}{k}, \quad \text{and}
\quad x_v' = \smiiii{0}{0}{x}{0} \; 
$$
in $M_{2n+2}(\widetilde{A})$ and thus we see that $\phi_v$ is unitary equivalent to $\sm{\phi_u}{0}{0}{0}$.
It follows that $\Phi$ is well-defined. 

We have already seen above that $\Theta \circ \Phi$ is the identity on $KO_0^u(A)$.  To see that $\Phi \circ \Theta$ is the identity
on $[A_0, \K\pr \otimes\sr A]$, let $\phi \colon A_0 \rightarrow M_n(A)$ be a given homomorphism. Let $h = \phi(h_0)$, $x = \phi(x_0)$, and $k = \phi(k_0)$.
Then $\Theta(\phi) = [\phi_2(u_0)] = [U(h, x, k)]$.
Then the reader can verify that $(\Phi \circ \Theta)(\phi)$ carries $h_0$, $x_0$, and $k_0$ to
$h'$, $x'$, and $k'$ in $M_{2n}(A)$ as given by Equations~(\ref{h''}). We will show that $\sm{\phi}{0}{0}{0}$ and $(\Phi \circ \Theta)(\phi)$ are homotopic by producing a homotopy of triples $\{h_t, x_t, k_t\}$ from $\{h, x, k\}$ to $\{h', x', k' \}$ in $M_{2n}(A)$ that satisfy (for each $t$) the conditions that $h_t k_t = 0$ and that $U(h_t, x_t, k_t)$ is a unitary in $M_{4n}(\widetilde{A})$.
For $t \in [0,1]$, let 
$$r_t = \begin{pmatrix}{\cos(\pi t/2 )  }&{-\sin(\pi t/2 )} \\
{\sin(\pi t/2 )}&{\cos(\pi t/2 )} \end{pmatrix} \; .$$  
Then let
$h_t  = h' ,
k_t = r_t k' r_t^*$, and
$x_t = r_t x'$.
The reader can verify directly that $h_t k_t = 0$ and that $U(h_t, x_t, k_t)$ is a unitary for all $t$ since
$$U(h_t, x_t, k_t) = 
  \begin{pmatrix}{1_{2n}}&{0}\\{0}&{r_t} \end{pmatrix}
  \cdot U(h', x', k') \cdot 
  \begin{pmatrix}{1_{2n}}&{0}\\{0}&{r_t^*} \end{pmatrix} \; .$$
\end{proof}

We end this section by giving a rephrasing of the definition of $KO_0^u$ in the context of \ctalg s. 
This gives a description of $KO_0^u(A, \tau)$ that is parallel to the forthcoming descriptions of $KO_j^u(A, \tau)$ for all values of $j$.

\begin{defn}
Let $(A, \tau)$ be a unital \ctalg. Let $U_{\infty}^{(0)}(A, \tau)$ be the set of all unitaries $u$ in $\cup_{n \in \N} M_{2n}(A)$ satisfying $u^2 = 1$ and $u^\tau = u$. Let $\sim_0$ be the equivalence relation on $U^{(0)}_{\infty}(A, \tau)$, generated by 
\begin{enumerate}
\item[(1)] $u_0 \sim_0 u_1$ if $u_t \in M_{2n}(A)$ is a continuous path of self-adjoint unitaries satisfying $u_t^\tau = u_t$; 
and 
\item[(2)] $u \sim_0 \iota^{(0)}_n(u)$ for $u \in M_{2n}(A)$ where 
$\iota^{(0)}_n \colon M_{2n}(A) \rightarrow M_{2n+2}(A)$ is given by
$$\iota^{(0)}_n(a) = {{\diag}} \left(a, I^{(0)} \right) \; \quad \text{where $I^{(0)} = \sm{1}{0}{0}{-1}$} \;. $$ 
\end{enumerate}
Then we define $KO^u_0(A, \tau) = U_{\infty}^{(0)}(A, \tau)/\sim_0$, with a binary operation given by $[u] + [v] = \left[\sm{u}{0}{0}{v}\right]$
for $u,v \in U_{\infty}^{(0)}(A, \tau)$.
\end{defn}

Since the sets $U_{\infty}^{(0)}(A, \tau)$ and $U_{\infty}^{(0)}(A^\tau)$ are identical and have the same equivalence relation, the following is immediate.

\begin{prop} \label{K0iii}
Let $(A, \tau)$ be a \ctalg~ and let $A^\tau = \{a \in A \mid a^\tau = a^* \}$ be the associated real \calg .  Then there is an isomorphism
$KO_0^u(A, \tau) \cong KO_0^u(A^\tau)$.
\end{prop}

\subsection{$KO_2$ via unitaries}

In this section, we will produce the definition of 
$KO^u_2(A, \tau)$ and we will prove that 
$KO^u_2(A, \tau) \cong [(q\C, \sharp), (\K \otimes A, \tau)  ] \cong [A_2, \K\pr \otimes A^\tau]$.  From Theorem~\ref{thm:classify}, we know that $[A_2, \K\pr \otimes A^\tau] \cong KO_2(A)$.  This will then imply that
$KO^u_2(A) \cong KO_2(A)$ (where the latter is defined in terms of projections in the double suspension).

\begin{defn} \label{K2def}
Let $(A, \tau)$ be a unital \ctalg . Let $U_{\infty}^{(2)}(A, \tau)$ be the set of all unitaries $u$ in $\cup_{n \in \N} M_{2n}(A)$ satisfying $u^2 = 1$ and $u^\tau = -u$. Let $\sim_2$ be the equivalence relation on $U^{(2)}_{\infty}(A, \tau)$, generated by 
\begin{enumerate}
\item[(1)] $u_0 \sim_2 u_1$ if $u_t \in M_{2n}(A)$ is a continuous path of self-adjoint unitaries satisfying $u_t^\tau = -u_t$; 
and 
\item[(2)] $u \sim_2 \iota^{(2)}_n(u)$ for $u \in M_{2n}(A)$ where 
$\iota^{(2)}_n \colon M_{2n}(A, \tau) \rightarrow M_{2n+2}(A)$ is given by
$$\iota^{(2)}_n(a)  
      = {{\diag}}(a, I^{(2)}) \quad \text{where $I^{(2)} = \sm{0}{i}{-i}{0} $}\; . $$
\end{enumerate}
Then we define $KO^u_2(A, \tau) = U_{\infty}^{(2)}(A, \tau)/\sim_2$, with a binary operation given by $[u] + [v] = \left[ \sm{u}{0}{0}{v} \right]$
for $u,v \in U_{\infty}^{(2)} (A, \tau)$.
\end{defn}

\begin{prop} \label{K2thm}
$KO_2^u(A, \tau)$ is a homotopy invariant functor from the category of unital 
\ctalg s to the category of abelian groups. 
\end{prop}

\begin{proof}
We note that the set $U_{\infty}^{(2)}(A, \tau)$ is closed under conjugation by elements in $O(2n)$. Furthermore, since $SO(2n)$ is connected, conjugation by any element in $SO(2n)$ induces the identity automorphism on 
$KO^u_2(A, \tau)$. 
For any $u,v \in U_{\infty}^{(2)}(A, \tau)$, with $u \in M_{2m}(A)$ and $v \in M_{2n}(A)$ we have
$$\begin{pmatrix}{0}&{1_{2n}}\\{1_{2m}}&{0} \end{pmatrix}
\begin{pmatrix}{u}&{0}\\{0}&{v} \end{pmatrix}
\begin{pmatrix}{0}&{1_{2m}}\\{1_{2n}}&{0}\end{pmatrix} 
= \begin{pmatrix}{v}&{0}\\{0}&{u}\end{pmatrix}\; .$$
Since $\sm{0}{1_{2n}}{1_{2m}}{0} \in SO(2n+2m)$, it follows that $[u] + [v] = [v] + [u]$ .

By definition the element $[I^{(2)} ]$ is an identity for the semigroup. We defer the proof that there are inverses until Proposition~\ref{K2inverse} below.
\end{proof}

\begin{prop}
$KO_2^u(\R) = KO_2^u(\C, {\id}) \cong \Z_2$.
\end{prop}

\begin{proof}
Let $u \in M_{2n}(\C)$ be a self-adjoint, skew-symmetric unitary.
As in Section~4 of \cite{loring14}, there is a factorization 
$u = x \cdot {{\diag}}(I^{(2)}, \dots, I^{(2)}) \cdot x^\tr$
where $x \in O(2n)$. Then depending on whether $x \in SO(2n)$ or not, we can find a path from $x$ either to $1_{2n}$ or to ${{\diag}}(1_{2n-1}, -1)$.
Therefore, we either have
$u \sim_2 {{\diag}}(I^{(2)}, \dots, I^{(2)}, I^{(2)})$ or
$u \sim_2 {{\diag}}(I^{(2)}, \dots, I^{(2)}, -I^{(2)})$. 

Thus there are at most two equivalence classes of $U_\infty^{(2)}(\C, {\id})$. Furthermore, $I^{(2)} \nsim_{(2)} -I^{(2)}$ as the two elements are distinguished by the value of the Pfaffian (see Theorem~4.1 of \cite{loring14}). Since 
$${{\diag}}(I^{(2)}, I^{(2)}) \sim_2 {{\diag}}(-I^{(2)}, -I^{(2)})$$ 
it follows that $KO^u_2(\C, {\id}) \cong \Z_2$.
\end{proof}

The proof above indicates that two self-adjoint skew-symmetric unitaries in $M_{2n}(\C)$ 
represent the same element of $KO_2^u(\C, {\id})$ if and only if they have the same Pfaffian (although there are different conventions regarding how the Pfaffian is actually defined, especially when $n$ is odd).
We establish the notation $I^{(2)}_n = {{\diag}}(I^{(2)}, \dots, I^{(2)}) \in M_{2n}(\C)$ for the standard representative of the identity in $KO_2^u(\C, {\id})$.

\begin{prop} \label{K2inverse}
Every element of $KO_2^u(A, \tau)$ has an inverse. If $u \in M_{2n}(A)$ and $n$ is even, then the inverse of $[u]$ is $[-u]$.
\end{prop}

\begin{proof}
Let $A$ be unital and let $u$ be a self-adjoint, skew-symmetric unitary in $M_{2n}(A)$. Then the matrix
$$u_t = \begin{pmatrix} {\cos (\pi t/2) \cdot u}&{i \sin (\pi t /2) \cdot 1_{2n}}\\
{i \sin (\pi t /2) \cdot 1_{2n}} & {-\cos (\pi t /2) \cdot u} \end{pmatrix} $$ 
for $t \in [0,1]$, gives a continuous path of self-adjoint skew-symmetric unitaries from $\sm{u}{0}{0}{-u}$ to $\sm{0}{i \cdot 1_{2n}}{-i \cdot 1_{2n}}{0}$. Depending on the parity of $n$, the latter matrix is similar via conjugation by a special orthogonal matrix either to $I_{2n}^{(2)}$ or to ${{\diag}}(I_{2n-1}^{(2)}, -I^{(2)})$. So in $KO_2^u(A,\tau)$, we either have $[u] + [-u] = 0$ or 
$[u] + [-u] + [-I^{(2)}] = 0$.
\end{proof}

\begin{defn}
For a \ctalg~$(A, \tau)$, we define $KO_2^u(A, \tau) = \ker(\lambda_*)$ where
$\lambda \colon \widetilde{A} \rightarrow \R$ is the natural projection on the unitization of $A$.
\end{defn}

\begin{prop} \label{K2unitpicture}
Let $A$ be a real \calg . Any element of $KO_2^u(A)$ can be represented as $[u]$ where $u \in M_{2n}(\widetilde{A})$ is a unitary satisfying $u^\tau = u$ and $\lambda(u) = I^{(2)}_n$.
\end{prop}

\begin{proof}
Let $u$ be a skew-symmetric self-adjoint unitary in $M_{2n}(\widetilde{A})$ such that $[\lambda(u)] = 0$ in $KO_2^u(\C, {\id})$. Then $\lambda(u) = x I_n^{(2)} x^\tr$ for some $x \in SO(2n)$. 
Let $v = x^\tr u x \in M_{2n}(\widetilde{A})$.  Then $u \sim_2 v$ and $\lambda(v) = I_n^{(2)}$ as desired.
\end{proof}

Recall from Section~\ref{sec:prelim} that there is an isomophism $(M_2(A), \tau) \cong (M_2(A), \taut)$, where $\taut$ is an alternate form of the transpose operator on matrices.
Thus one can make the following alternative definition of $KO_2^u(A)$ using $\taut$ in place of $\tau$.

\begin{defn} ~\label{DefKtd2}
Let $(A, \tau)$ be a unital \ctalg . 
Let $\Utd_{\infty}^{(2)}(A, \tau)$ be the set of all unitaries 
$u$ in $\cup_{n \in \N} M_{2n}(A)$ satisfying $u^2 = 1$ and $u^{\taut} = -u$. 
Let $\sim_2$ be the equivalence relation on 
$\Utd^{(2)}_{\infty}(A, \tau)$, generated by 
\begin{enumerate}
\item[(1)] $u_0 \sim_2 u_1$ if $u_t \in M_{2n}(A)$ is a continuous path of self-adjoint unitaries satisfying $u_t^{\taut} = -u_t$; 
and 
\item[(2)] $u \sim_2 \iota^{(2)}_n(u)$ for $u \in M_{2n}(A)$ where 
$\iota^{(2)}_n \colon M_{2n}(A) \rightarrow M_{2n+2}(A)$ is given by
$$\iota^{(2)}_n(a)  
      = {{\diag}}(a, 1, -1) . $$
\end{enumerate}
Then we define $\Ktd^u_2(A, \tau) = \Utd_{\infty}^{(2)}(A, \tau)/\sim_2$, with a binary operation given by $[u] + [v] = \left[ \sm{u}{0}{0}{v} \right]$
for $u,v \in \Utd_{\infty}^{(2)} (A, \tau)$.
\end{defn}

There is a natural isomorphism
$KO_2^u(A, \tau) \cong \Ktd_2^u(A, \tau)$
given by conjugation by the matrix $W$ introduced in Section~\ref{sec:prelim}.
This model of $KO_2^u(A, \tau)$ will be used only in this section to obtain our results.

We set a different convention for the action of the involution 
$\taut_{2n} = \trt \otimes \tau_n$ on $M_{2n}(A)$ 
(similar to our discussion of the action of $\sharp$ and $\widetilde{\sharp}$ in Section~\ref{sec:prelim}).
We use the formula
$$\begin{pmatrix} x & y \\ z & w \end{pmatrix} ^{\taut} 
      = \begin{pmatrix} w^{\tau_n} & y^{\tau_n} \\ z^{\tau_n} & x^{\tau_n}
    \end{pmatrix} \; $$
where $x,y,z,w \in M_{n}(A)$.
This convention changes the formula for 
$$\iota_n^{(2)} \colon M_{2n}({A}) \rightarrow M_{2n+2}({A})$$ 
given in Definition~\ref{DefKtd2}. 
Instead of the formula there we have
$$ \iota_n^{(2)} 
  \begin{pmatrix} x & y \\ z & w \end{pmatrix} 
  =  \begin{pmatrix} x & 0 & y & 0 \\ 0 & 1 & 0 & 0 \\ 
	      z & 0 & w & 0 \\ 0 & 0 & 0 & -1
	    \end{pmatrix}  \; $$
for $x,y,z,w \in M_n({A})$.
With this notation, if $x \in M_{2n}(A)$ satisfies
$x^{\taut} = -x$, then the element $y = \iota^{(2)}_n(x)$ also satisfies 
$y^{\taut} = -y$. With this notation, the trivial element of $\Ktd_2^u(\R) \cong \Z_2$ is represented by
$\sm{1_n}{0}{0}{-1_n}$ for any $n$.

\begin{thm} \label{K2class}
For any \ctalg~$(A, \tau)$,
there is a natural isomorphism 
$$[(q\C, \sharp), (\K \otimes A, \tau)] 
      \cong KO_2^u(A) \; .$$
\end{thm}

\begin{proof}
Assume $A$ is unital.
We will prove that there is an isomorphism between 
$[(q\C, \sharp), (\K\pr \otimes A, \tau)]$ and 
$\Ktd_2^u(A, \tau)$ using a similar proof to that of Proposition~\ref{K0class}.
Let 
$$u_0 = U(h_0, x_0, k_0) 
    = \left( \begin{matrix} 1 - 2t & 0 & 0 & 2\sqrt{t - t^2} \\
		  0 & 1 & 0 & 0 \\ 0 & 0 & -1 & 0 \\
		  2 \sqrt{t - t^2} & 0 & 0 & 2t-1 \end{matrix} \right)
  \in M_2(\widetilde{q\C})$$ 
and note that we have $u_0^{\sharp \otimes \trt} = -u_0$; so $[u_0] \in KO_2^u(q\C, \sharp)$. If 
$\phi \colon (q \C, \sharp) \rightarrow (M_{n}(A), \tau)$ is a \ctalg~homomorphism,
then $\phi(u_0) \in  \Utd_{\infty}^{(2)}(A, \tau)$. We define $\Theta(\phi) = [\phi(u_0)] \in \Ktd_2^u(A, \tau)$. To show this is well-defined, we need to consider $\phi' = \sm{\phi}{0}{0}{0}$.

We can now check that as in the proof of Proposition~\ref{K0class}, 
\begin{align*} \phi_2(u_0) 
&= \left( \begin{matrix}{1_n - 2 \phi(h_0)}&{2 \phi(x_0)^*}\\{2 \phi(x_0)}&{2 \phi(k_0) -1_n} 
          \end{matrix} \right)  \\
\text{and} \qquad 
\phi'_2(u_0)   
  &= \left( \begin{matrix}
      ~1_n - 2 \phi(h_0)~ & 0 & 2 \phi(x_0)^* & 0 \\
      0 & 1 & 0 & 0 \\
      2 \phi(x_0) & 0 & ~2 \phi(k_0) - 1_n~ & 0 \\
      0 & 0 & 0 & -1 
      \end{matrix} \right) \; .  
\end{align*}
Thus we have $\Theta(\phi'(u_0)) = \iota^{(2)}_n(\phi(u_0))$, showing that $\Theta$ is well defined.

To construct an inverse to $\Theta$ (as in the proof of Proposition~\ref{K0class}), suppose $u \in M_{2n}(\widetilde{A})$ is a unitary satisfying $u^2 = 1$, $u^{\taut} = -u$, and 
$\lambda(u) = \sm{1_n}{0}{0}{-1_n}$. Then there exist $h,x,k \in M_n(A)$ such that
$$u = U(h, x, k) 
  = \begin{pmatrix} {1_n - 2h}&{2x^*}\\{2x}&{2k - 1_n} \end{pmatrix} \; .$$
The conditions that $u = u^*$ and $u^{\taut} = -u$
are equivalent to the conditions that $h$ and $k$ are self-adjoint, they are interchanged by $\tau$, and $x^\tau = -x$.

We now construct modified elements $h', k', x'$ that satisfy the same conditions as well as the condition $h' k' = 0$. Let 
$$W_{2n} = \tfrac{1}{\sqrt{2}} \begin{pmatrix} {i \cdot 1_n} & {1_n} \\ {1_n} & {i \cdot 1_n} \end{pmatrix} 
			\in M_{2n}(\C) $$ 
(generalizing the definition of $W$ from Section~\ref{sec:prelim})
and define
\begin{align*}
h' &= W_{2n} \sm{h}{0}{0}{0} W_{2n}^*  = \tfrac{1}{2} \sm{h}{ih}{-ih}{h} , \\ 
   k' &= W_{2n}^* \sm{k}{0}{0}{0} W_{2n} = \tfrac{1}{2} \sm{k}{-ik}{ik}{k} ,    \\
 x' &= W_{2n}^* \sm{x}{0}{0}{0} W_{2n}^* = \tfrac{1}{2} \sm{-x}{-ix}{-ix}{x}  \; 
\end{align*}
in $M_{2n}(A)$. We leave it to the reader to check that $h'$ and $k'$ are self-adjoint, that $h' k' = 0$, that $(h')^{\tau} = k'$, and that
$(x')^{\tau} = -x'$. The following calculation shows that $u' = U(h', x', k')$ is a unitary:
\begin{align*}
u' &= \sm{1_{2n} - 2h'}{2(x')^*}{2x'}{2k'-1_{2n}} \\
&= \sm{1_{2n} - 2W_{2n} \sm{h}{0}{0}{0} W_{2n}^* }  { 2W_{2n} \sm{x^*}{0}{0}{0}W_{2n} }
      {2 W_{2n}^* \sm{x}{0}{0}{0} W_{2n}^* }{ 2W_{2n}^* \sm{k}{0}{0}{0} W_{2n} - 1_{2n}  } \\
&= \sm{W_{2n}}{0}{0}{W_{2n}^*} 
  \sm{1_{2n} - 2\sm{h}{0}{0}{0}} {2 \sm{x^*}{0}{0}{0}}
  {2 \sm{x}{0}{0}{0}}  {2  \sm{k}{0}{0}{0} - 1_{2n} } 
  \sm{W_{2n}^*}{0}{0}{W_{2n}}  \\
&= \sm{W_{2n}}{0}{0}{W_{2n}^*} \cdot \iota^{(2)}(u) \cdot \sm{W_{2n}^*}{0}{0}{W_{2n}} \; .
\end{align*}
where $\iota^{(2)} \colon M_{2n}({A}) \rightarrow M_{4n}({A})$ is the composition
$\iota_{2n-1}^{(2)} \dots \iota_{n+1}^{(2)} \iota_n^{(2)}$ and
$$\iota^{(2)}(u) = \begin{pmatrix}
1_n - 2 h & 0 & 2x^* & 0 \\
0 & 1_n & 0 & 0 \\
2x & 0 & 2k - 1_n & 0 \\
0 & 0 & 0 & -1_n
\end{pmatrix} \; .
$$

Thus $u' = U(h', x', k')$ is a unitary that satisfies $u'= (u')^*$ and $(u')^{\taut} = -u'$.
By Proposition~\ref{qCuniversal1}, there is a homomorphism
$\phi_u \colon (q\C, \sharp) \rightarrow (M_{2n}(A), \tau)$ such that
$\phi_u(h_0) = h'$, $\phi_u(k_0) = k'$, and $\phi_u(x_0)= x'$.
Thus $(\phi_u)_2(u_0) = u'$.

We claim that $u \sim_2 u'$, which will imply that
$\Theta(\phi_u) = [u]$ as desired. Specifically, we will show that $u'$ is homotopic to 
$\iota^{(2)}(u)$ in $M_{4n}({A})$. 
The previous calculation shows that $u' = v \iota^{(2)}(u) v^*$ where $v = \sm{w}{0}{0}{w^*}$.
So it suffices to show that 
$vx v^* \sim_2 x$ for any self-adjoint unital $x \in M_{4n}({A})$ satisfying $x^{\taut} = -x$. Let 
$$W_{4n} 
  = \tfrac{1}{\sqrt{2}} 
  \begin{pmatrix}{i \cdot \1_{2n}}&{\1_{2n}}\\{\1_{2n}}&{i \cdot \1_{2n}}
  \end{pmatrix}
 \; .$$ 
For any $x \in M_{4n}({A})$ we have $x^{\taut} = -x$ if and only if $(W_{4n} x W_{4n}^*)^\tau = -(W_{4n} x W_{4n}^*)$. Then as $U^{(2)}({A}, \tau)$ is closed under conjugation by $SO(4n)$, so $\Utd^{(2)}({A}, \tau)$ is closed under conjugation by elements in 
$$\widetilde{SO}(4n) = W_{4n}^* SO(4n) W_{4n}\; .$$ 
Now  
$v \in \widetilde{SO}(4n)$ since 
$$W_{4n} v W_{4n}^* = \frac{1}{\sqrt{2}} \begin{pmatrix} 0 & 1 & -1 & 0 \\1 & 0 & 0 & -1 \\
  1 & 0 & 0 & 1 \\ 0 & 1 & 1 & 0 \end{pmatrix} \in SO(4n) \; .$$
Since $SO(4n)$ is connected, so is $\widetilde{SO}(4n)$. Therefore there is path in $\widetilde{SO}(4n)$ from $v$ to the identity which proves that $v x v^* \sim_2 x$.

This completes the proof that $\Theta(\phi_u) = [u]$. The rest of the proof consists in showing that the construction described is actually an inverse to $\Theta$. This proceeds as in the proof of Proposition~\ref{K0class}.
\end{proof}

\subsection{$KO_4$ via unitaries}

\begin{defn} \label{K4def}
Let $(A, \tau)$ be a unital \ctalg. Let $U_{\infty}^{(4)}(A, \tau)$ be the set of all unitaries $u$ in $\cup_{n \in \N} M_{4n}(A)$ satisfying $u^2 = 1$ and $u^{\sharp \otimes \tau} = u$. Let $\sim_4$ be the equivalence relation on $U^{(4)}_{\infty}(A, \tau)$, generated by 
\begin{enumerate}
\item[(1)] $u_0 \sim_4 u_1$ if $u_t \in M_{4n}(A)$ is a continuous path of self-adjoint unitaries satisfying $u_t^{\sharp \otimes \tau} = u_t$; 
and 
\item[(2)] $u \sim_4 \iota^{(4)}_n(u)$ for $u \in M_{4n}(A)$ where 
$\iota^{(4)}_n \colon M_{4n}(A) \rightarrow M_{4n+4}(A)$ is given by 
$$\iota^{(4)}_n(a) = {{\diag}}(a, I^{(4)}) \; \quad \text{where $I^{(4)} = {{\diag}}(1,1,-1,-1)$} \; . $$
\end{enumerate}
Then we define $KO^u_4(A, \tau) = U_{\infty}^{(4)}(A, \tau)/\sim_4$, with a binary operation given by $[u] + [v] = \left[\sm{u}{0}{0}{v}\right]$
for $u,v \in U_{\infty}^{(0)}(A, \tau)$.
\end{defn}

In the above definition, the formulas for $\iota^{(4)}_n $, for $I^{(4)}$, and for addition implicitly assume the particular convention for the action of the involution $\sharp \otimes \tau$ on $M_{4n}(A)$ as discussed in Section~\ref{sec:prelim}.
Under this convention, the addition formula and the formula for $\iota_n^{(4)}$ in Definition~\ref{K4def} preserve membership in $U_\infty^{(4)}(A)$.

\begin{prop} \label{K4iso}
If $(A, \tau)$ is a unital \ctalg, then $KO^u_4(A, \tau) \cong KO_4(A, \tau)$. In particular, $KO^u_4(\C, {\id}) = \Z$.
\end{prop}

\begin{proof}
An element of $U_\infty^{(4)}(A, \tau)$ is given by a self-adjoint unitary $u \in M_2(\C) \otimes M_{2n}(A)$ that satisfies $u^{\sharp \otimes \tau} = u$. This is the same as an element of $U_\infty^{(0)}(M_2(\C) \otimes A, \sharp \otimes \tau)$. Therefore
$U_\infty^{(4)}(A, \tau) \cong U_\infty^{(0)}(M_2(\C) \otimes A, \sharp \otimes \tau)$ and hence $KO^u_4(A, \tau) \cong KO^u_{0}(M_2(\C) \otimes A, \sharp \otimes \tau)$.

As a special case of the K\"unneth formula (the Main Theorem of \cite{boersema02}), for a real \calg~ $A$
we know that $KO_n(A) \cong KO_{n+4}(\H \otimes A)$. The same statement in terms of \ctalg s is that $KO_n(A, \tau) \cong KO_{n+4}(M_2(\C) \otimes A, \sharp \otimes \tau)$.  Combining this with Theorem~\ref{K0iso},
\begin{align*}
KO_4^u(A, \tau) &\cong KO_{0}^u(M_2(\C) \otimes A, \sharp \otimes \tau) \\
  & \cong KO_{0}(M_2(\C) \otimes A, \sharp \otimes \tau) \\
  & \cong KO_4(A, \tau)
\end{align*}
\end{proof}

The identity element of $KO^u_4(\C, {\id}) \cong \Z$ is represented by $I^{(4)} \in M_{4}(\C)$ or, more generally, by $I_n^{(4)} = {{\diag}}( I^{(4)}, \dots, I^{(4)}) \in M_{4n}(\C)$. The isomorphism 
$KO_4^u(\C, {\id}) \rightarrow \Z$ can be written as $[u] \mapsto \tfrac{1}{4} {\trace}(u)$.

\begin{thm} \label{K4group}
$KO_{4}^u(A, \tau)$ is a homotopy invariant functor from the category of unital \ctalg s to the category of abelian groups. The inverse of an element $[u]$ in $KO_{4}^u(A, \tau)$ is given by $[-u]$.
\end{thm}

\begin{proof}
The first statement follows immediately from Proposition~\ref{K4iso}. The statement about inverses follows from 
Proposition~\ref{K0group} and
the statement in the proof of Proposition~\ref{K4iso} that 
$U_\infty^{(4)}(A, \tau) \cong U_\infty^{(0)}(M_2(\C) \otimes A, \sharp \otimes \tau)$.
\end{proof}

\begin{defn}
Let $(A, \tau)$ be any \ctalg .  Then we define 
$$KO^u_4(A, \tau) = \ker(\lambda_*)$$
where $\lambda_* \colon KO^u_4(\widetilde{A}, \tau) \rightarrow KO^u_4(\C, {\id})$.
\end{defn}

Combining Theorem~\ref{K4group} with this definition gives the following.

\begin{thm}
If $(A, \tau)$ is any \ctalg, then $KO^u_4(A, \tau) \cong KO_4(A, \tau)$.
\end{thm}

\begin{thm} \label{K4unitpicture}
Let $(A, \tau)$ be a \ctalg. Any element of $K^u_4(A, \tau)$ can be represented as $[u]$ where $u \in M_n(\widetilde{A})$ satisfies $u^2 = 1$, $u^{\sharp \otimes \tau} = u$, and $\lambda_*(u) = I_{n}^{(4)}$.
\end{thm}

\begin{proof}
Consider the following commutative diagram in which each row is a short exact sequence and $g$ is the isomorphism described in the proof of Proposition~\ref{K4iso}.
$$ \xymatrix{
KO^u_0(M_2(\C) \otimes A, \sharp \otimes \tau) \ar[r] \ar[d]^{\id} & 
  KO_0^u((M_2(\C) \otimes A)^{\sim}, \sharp \otimes \tau) \ar[r]^{ \hspace{1cm} \lambda_*} \ar[d]^{k_*} &
  KO_0^u(\C, {\id}) \ar[d]^{\iota_*}  \\
KO^u_0(M_2(\C) \otimes A, \sharp \otimes \tau) \ar[r] \ar[d]^g &
  KO_0^u(M_2(\C) \otimes \widetilde{A}, \sharp \otimes \tau) \ar[r]^{\hspace{1cm} \lambda_*} \ar[d]^g &
  KO_0^u(M_2(\C), \sharp) \ar[d]^g \\
KO_4^u(A, \tau) \ar[r] &
  KO_4^u(\widetilde{A}, \tau) \ar[r]^{\hspace{1cm} \lambda_*} &
  KO_4^u(\C, {\id})
} $$

It follows from the diagram that any element of $KO_4^u(A, \tau)$ can be written as $gk_*([u])$ where 
$[u] \in KO_0^u((M_2(\C) \otimes A)^{\sim}, \sharp \otimes \tau)$ 
and $\lambda_*([u]) = 0$. Using Proposition~\ref{K0unitpicture}, we can take
$u \in M_{2n}((M_2(\C) \otimes A)^\sim)$
such that $u^2 = 1$, $u^{\sharp \otimes \tau} = u$,
and $\lambda(u) = I_{2n}^{(0)}$.  Then $v = g k(u) \in M_{4n}(\widetilde{A})$
satisfies 
$v^2 = 1$, $v^{\sharp \otimes \tau} = v$, and
$\lambda(v) = I_n^{(4)}$ as desired.
\end{proof}

\subsection{$KO_6$ via unitaries}

\begin{defn} \label{K6def}
Let $(A, \tau)$ be a unital \ctalg. Let $U_{\infty}^{(6)}(A, \tau)$ be the set of all unitaries 
$u$ in $\cup_{n \in \N} M_{2n}(A)$ satisfying $u^2 = 1$ and $u^{\sharp \otimes \tau} = -u$. Let $\sim_6$ be the equivalence relation on $U^{(6)}_{\infty}(A, \tau)$, generated by 
\begin{enumerate}
\item[(1)] $u_0 \sim_6 u_1$ if $u_t \in M_{2n}(A)$ is a continuous path of self-adjoint unitaries satisfying $u_t^{\sharp \otimes \tau} = u_t$; 
and 
\item[(2)] $u \sim_6 \iota^{(4)}_n(u)$ for $u \in M_{2n}(A)$ where 
$\iota^{(6)}_n \colon M_{2n}(A) \rightarrow M_{2n+2}(A)$ is given by 
$$\iota^{(6)}_n(a) = {{\diag}}(a, I^{(6)}) \; \quad 
  \text{where $I^{(6)} = 
      \left( \begin{smallmatrix} 0 & i \\ -i & 0 \end{smallmatrix} \right)$} \; . $$
\end{enumerate}
Then we define $KO^u_6(A, \tau) = U_{\infty}^{(6)}(A, \tau)/\sim_6$, with a binary operation given by $[u] + [v] = \left[\sm{u}{0}{0}{v}\right]$
for $u,v \in U_{\infty}^{(0)}(A, \tau)$.
\end{defn}

If, in the definition above, we only used unitaries in $M_{4n}(A)$, then it would be clear that $KO_6^u(A, \tau)$ is isomorphic to $KO_2^u(M_2(\C) \otimes A, \sharp \otimes \tau)$. However, the allowed inclusion of unitaries in $M_{2n}(A)$ does not change the group in the limit. Therefore, the following results follow a development that is similar to that for $KO_4^u(A, \tau)$. 

\begin{prop} \label{K6iso}
If $(A, \tau)$ is a unital \ctalg, then $KO^u_6(A, \tau) \cong KO_6(A, \tau)$. In particular, 
$KO^u_6(\C, {\id}) = 0$.
\end{prop}

\begin{prop} \label{K6group}
$KO_{6}^u(A, \tau)$ is a homotopy invariant functor from the category of unital \ctalg s to the category of abelian groups. The inverse of an element $[u]$ in $KO_{6}^u(A, \tau)$ is given by $[-u]$ if $u \in M_{2n}(A)$.
\end{prop}

\begin{defn}
Let $(A, \tau)$ be any \ctalg .  Then we define 
$$KO^u_6(A, \tau) = \ker(\lambda_*)$$
where $\lambda_* \colon KO^u_6(\widetilde{A}, \tau) \rightarrow KO^u_6(\C, {\id})$.
\end{defn}

\begin{prop}
If $(A, \tau)$ is any \ctalg, then 
$$KO^u_6(A, \tau) \cong KO_6(A, \tau) \; .$$
\end{prop}

\begin{prop} \label{K6unitpicture}
Let $(A, \tau)$ be a \ctalg. Any element of $K^u_6(A, \tau)$ can be represented as $[u]$ where $u \in M_{4n}(\widetilde{A})$ satisfies $u^2 = 1$, $u^{\sharp \otimes \tau} = u$, and $\lambda_*(u) = I_{n}^{(6)}$.
\end{prop}

\section{$K$-theory via unitaries -- the odd cases} \label{sec:odd}

\subsection{$KO_1$ via unitaries}

The following definitions and theorems represent the standard development of $KO_1(A)$ as in Chapter 8 of \cite{rordambookblue} for the complex case. They are included here for reference and terminology; and the proofs will be omitted as appropriate. 

\begin{defn} \label{defn:KO_1}
Let $A$ be a real unital \calg . Let $U_{\infty}^{(1)}(A)$ be the set of all unitaries $u$ in $\cup_{n \in \N} M_{n}(A)$. Let $\sim_1$ the equivalence relation on $U^{(1)}_{\infty}(A)$, generated by 
\begin{enumerate}
\item[(1)] $u_0 \sim_1 u_1$ if $u_t \in M_{n}(A)$ is a continuous path of unitaries for $t \in [0,1]$; 
and 
\item[(2)] $u \sim_1 \iota^{(1)}_n(u)$ for $u \in M_{n}(A)$ where 
$\iota_n^{(1)} \colon M_{n}(A) \rightarrow M_{n+1}(A)$ is given by
$$\iota_n^{(1)}(a) =  \sm{a}{0}{0}{1}
  = {{\diag}} (a,1) \; . $$
\end{enumerate}
Then we define $KO^u_1(A) = U_{\infty}^{(1)}(A)/\sim_1$, with a binary operation given by $[u] + [v] = \left[ \sm{u}{0}{0}{v} \right]$
for $u,v \in U_{\infty}^{(1)}(A)$.
\end{defn}

\begin{prop}
$KO_1^u(A)$ is a homotopy invariant functor from the category of unital \calg s to the category of abelian groups. The inverse of an element $[u]$ in $KO^u_1(A)$ is given by $[u^*]$.
\end{prop}

\begin{prop}
$KO^u_1(\R) \cong \Z_2$. The isomorphism $KO^u_1(\R) \rightarrow \Z_2$ is given by $[u] \mapsto  \det(u)$.
\end{prop}

\begin{defn}
Let $A$ be any unital \calg . Then we define
$KO^u_1(A) = \ker(\lambda_*)$ where $\lambda \colon \widetilde{A} \rightarrow \R$ is the natural projection from the unitization of $A$ with kernel isomorphic to $A$.
\end{defn}

\begin{prop} \label{K1unitpicture}
Let $A$ be a real \calg . Any element of $KO_1^u(A)$ can be represented as $[u]$ where $u \in M_n(\widetilde{A})$ is a unitary satisfying $\lambda(u) = 1_n$.
\end{prop}

\begin{proof}
Let $u \in M_n(\widetilde{A})$ be a unitary element satisfying $\lambda_*([u]) = 0 \in KO^u_1(\R)$.  This implies that $\det(\lambda(u)) = 1$ so there is a path $u_t$ of unitaries in $M_n(\R)$ such that $u_0 = 1_n$ and $u_1 = \lambda(u)$. Then $v_t = u (u_t)^*$ is a path of unitaries in $M_n(\widetilde{A})$ such that $v_0 = u$ and $\lambda(v_1) = \lambda(u) u_1^* = 1_n$.
\end{proof}

\begin{prop}
For any real \calg~ there is an isomorphism 
$$\Gamma \colon KO^u_1(A) \rightarrow KO^u_0(SA)$$ 
given as follows.
Suppose that $u \in M_n(\widetilde{A})$ is a unitary, satisfying $\lambda(u) = 1_n$. Let $v_t \in M_{2n}(\widetilde{A})$ be a continuous path of unitaries such that $v_0 = \1_{2n}$, $v_1 = {{\diag}}(u, u^*)$, and $\lambda(v_t) = 1_{2n}$ for all $t \in [0,1]$.
Then
$\Gamma([u]) = \left[ 2 v_t \sm{\1_n}{0}{0}{0} v_t^* - \1_{2n} \right]$.
\end{prop}

\begin{proof}
There is an isomorphism $\gamma \colon KO^u_1(A) \rightarrow KO_0(SA)$ 
given by 
$$\gamma([u]) = \left[ v \sm{\1_n}{0}{0}{0} v^* \right] - \left[ \1_n \right]$$ where $v$ is as in the statement of the theorem above. This is well-known in the complex case (see for example Theorem~10.1.3 of \cite{rordambookblue}) and works the same in the real case. Then the isomorphism $\Gamma$ is given by $\Gamma = \Theta \circ \gamma$ where $\Theta$ is from Theorem~\ref{K0iso}.
\end{proof}

\begin{cor}
For any real \calg~ $A$, there is a natural isomorphism $KO^u_1(A) \cong KO_1(A)$.
\end{cor}

We end this section by giving a rephrase of the definition of $KO_0^u$ in the context of a \ctalg. This gives a description of $KO_1^u(A, \tau)$ that is parallel to the forthcoming descriptions of $KO_j^u(A, \tau)$ for all values of $j$.

\begin{defn}
Let $(A, \tau)$ be a unital \ctalg. Let $U_{\infty}^{(1)}(A, \tau)$ be the set of all unitaries $u$ in $\cup_{n \in \N} M_{n}(A)$ satisfying $u^\tau = u^*$. Let $\sim_1$ be the equivalence relation on $U^{(1)}_{\infty}(A, \tau)$, generated by 
\begin{enumerate}
\item[(1)] $u_0 \sim_1 u_1$ if $u_t \in M_{n}(A)$ is a continuous path of unitaries satisfying $u_t^\tau = u_t^*$; 
and 
\item[(2)] $u \sim_1 \iota^{(1)}_n(u)$ for $u \in M_{n}(A)$ where 
$\iota^{(1)}_n \colon M_{n}(A) \rightarrow M_{n+1}(A)$ is given by
$$\iota^{(1)}_n(a) = {{\diag}} \left(a, 1  \right) \; . $$
\end{enumerate}
Then we define $KO^u_1(A, \tau) = U_{\infty}^{(1)}(A, \tau)/\sim_1$, with a binary operation given by $[u] + [v] = \left[\sm{u}{0}{0}{v}\right]$
for $u,v \in U_{\infty}^{(1)}(A, \tau)$.
\end{defn}

As in Proposition~\ref{K0iii}, the sets $U_\infty^{(1)}(A, \tau)$ and $U_\infty^{(1)}(A^\tau)$ are easily seen to be identical.

\begin{prop}
Let $(A, \tau)$ be a \ctalg~ and let $A^\tau = \{a \in A \mid a^\tau = a^* \}$ be the associated real \calg .  Then there is an isomorphism
$KO_1^u(A, \tau) \cong KO_1^u(A^\tau)$.
\end{prop}

\subsection{$KO_{-1}$ via unitaries}

\begin{defn}
Let $(A, \tau)$ be a unital \ctalg. Let $U_{\infty}^{(-1)}(A, \tau)$ be the set of all unitaries $u$ in $\cup_{n \in \N} M_{n}(A)$ that satisfy $u^\tau = u$. Let $\sim_{(-1)}$ be the equivalence relation on $U^{(-1)}_{\infty}(A, \tau)$, generated by 
\begin{enumerate}
\item[(1)] $u_0 \sim_{(-1)} u_1$ if $u_t \in M_{n}(A)$ is a continuous path of unitaries satisfying $u_t^\tau = u_t$; 
and 
\item[(2)] $u \sim_{(-1)} \iota^{(-1)}_n(u)$ for $u \in M_{n}(A)$ where 
$\iota_n^{(-1)} \colon M_{n}(A) \rightarrow M_{n+1}(A)$ is given by
$$\iota_n^{(-1)}(a) 
  = {{\diag}} (a,1) \; . $$
\end{enumerate}
Then we define $KO^u_{-1}(A, \tau) = U_{\infty}^{(-1)}(A, \tau)/\sim_{(-1)}$, with a binary operation given by $[u] + [v] = \left[ \sm{u}{0}{0}{v} \right]$
for $u,v \in U_{\infty}^{(-1)}(A, \tau)$.
\end{defn}

\begin{prop} \label{K-1r}
$KO^u_{-1}(\C, {\id}) = 0$.
\end{prop}

\begin{proof}
Let $u \in M_n(\C)$ be a unitary element such that $u^{\tr} = u$. By Corollary~4.4.4 of \cite{hornbook}, there exists a unitary $v$ such that $u = v^{\tr} v$. Since the group of unitaries in $M_n(\C)$ in path connected, we can find a path $v_t$ from $v$ to $1_n$ and let $u_t = v_t^{\tr} v_t$ be the path of unitaries from $u$ to $\1_n$ satisfying $u_t^{\tr} = u_t$. 
\end{proof}

\begin{prop} \label{K-1group}
$KO_{-1}^u(A, \tau)$ is a homotopy invariant functor from the category of unital \ctalg s to the category of abelian groups. The inverse of an element $[u]$ in $KO_{-1}^u(A, \tau)$ is given by $[u^*]$.
\end{prop}

\begin{proof}
We leave the question of functoriality and homotopy invariance to the reader, and show that the binary operation is commutative. 

Let $u \in M_n(A)$ and $v \in M_m(A)$ be unitaries satisfying $u^\tau = u$ and $v^\tau = v$. 
First we claim that there exists a unitary $w$ in $M_{n+m}(\R)$ such that $w \sm{u}{0}{0}{v} w^* = \sm{v}{0}{0}{u}$ and $\det w = 1$. If either $m$ or $n$ is even, then $w$ can be taken to be the obvious change of basis matrix corresponding to an even permutation of the basis elements. On the other hand if $m$ and $n$ are both odd then let $w'$ be the odd permutation matrix that satisfies $w' \sm{u}{0}{0}{v} (w')^* = \sm{v}{0}{0}{u}$. Then let $w = w' \sm{1_n}{0}{0}{-1_m}$. This proves the claim. Then let $w_t$ be a path of special orthogonal matrices from $1_{n+m}$ to $w$ and consider the path $x_t = w_t \sm{u}{0}{0}{v} w_t^*$. Verifying that this satisfies $x_t^\tau = x_t$ completes the proof.

For the statement about inverses, we claim that if $u$ is a unitary in $M_n(A)$ that satisfies $u^\tau = u$, then $\sm{u}{0}{0}{u^*} \sim_{(-1)} \sm{1_n}{0}{0}{1_n}$. Indeed,
$$v_t = \begin{pmatrix}{\cos \left((\pi/2) t  \right) \cdot u}&{\sin \left((\pi/2) t  \right) \cdot 1_n }\\
{\sin \left((\pi/2) t  \right) \cdot 1_n}&{-\cos \left((\pi/2) t  \right) \cdot u^*}
\end{pmatrix} $$
is a homotopy from $\sm{u}{0}{0}{-u^*}$ to $\sm{0}{1_n}{1_n}{0}$ satisfying $v_t^\tau = v_t$.
Since any matrix of the form $w = \sm{u}{0}{0}{\lambda u^*}$ for $|\lambda| = 1$ satisfies $w^\tau = w$, we have $\sm{u}{0}{0}{u^*} \sim_{(-1)} \sm{u}{0}{0}{-u^*}$. The argument given in the proof of Proposition~\ref{K-1r} above implies that $\sm{0}{1_n}{1_n}{0} \sim_{(-1)} \sm{1_n}{0}{0}{1_n}$.
\end{proof}

\begin{defn} 
Let $(A, \tau)$ be any \ctalg. Then we define
$$KO^u_{-1}(A, \tau) = \ker(\lambda_*) \; .$$
\end{defn}

Since $KO^u_{-1}(\C, {\id}) = 0$ it follows that $KO^u_{-1}(A, \tau) = KO^u_{-1}(\widetilde{A}, \tau)$. 

\begin{prop} \label{K-1unitpicture}
Let $(A, \tau)$ be a \ctalg~and let $\widetilde{A}$ be the unitization. Any element of $KO^u_{-1}(A, \tau)$ can be represented as $[u]$ where $u \in M_n(\widetilde{A})$ is a unitary satisfying $u^\tau = u$ and $\lambda(u) = 1_n$.
\end{prop}

\begin{proof}
Let $u$ be a unitary in $M_n(\widetilde{A})$ such that $u^\tau = u$. Let $x = \lambda(u) \in M_n(\C)$. As in the proof of Proposition~\ref{K-1r}, write $x = y^{\tr} y$ where $y$ is a unitary in $M_n(\C)$ and let $y_t$ be a path of unitaries from $1_n$ to $y$. Then $z_t = (y_t^*)^{\tr} u y_t^*$ is a path from $u$ to a unitary $z_1$ in $M_n(\widetilde{A})$ that satisfies $\lambda(z_1) = 1_n$.  Also, note that we have $z_t^\tau = z_t$ for all $t$. 
\end{proof}

\begin{prop}
For any \ctalg~$(A, \tau)$, there is a natural isomorphism 
$KO^u_{-1}(A, \tau) \cong [(C_0(S^1), {\id}), (\K\pr \otimes\sr A, \tau)] $.
\end{prop}

\begin{proof}
Let $u_0$ be the identity function in $C(S^1)$.  That is, $u_0(z) = z$ for all $z \in S^1$.  Then $[u_0] \in KO^u_{-1}(C(S^1), \tau) \cong \Z$. For any $\phi \colon (C_0(S^1), {\id}) \rightarrow ( M_n(A), \tau)$, we can extend to the unitization to obtain
$\phi \colon (C(S^1), {\id}) \rightarrow ( M_n(\widetilde{A}), \tau)$
and write $\theta([\phi]) = [\phi(u_0)] \in KO^u_{-1}(A)$.
For $\phi' = \sm{\phi}{0}{0}{0}$ we have 
$\phi'(u_0) = \sm{\phi(u_0)}{0}{0}{\1}$, so 
$\theta([\phi]) = \theta([\phi'])$.
This gives us a well defined natural transformation
$$\Theta \colon 
[(C_0(S^1), {\id}), (\K\pr \otimes A, \tau)] \rightarrow KO^u_{-1}(A, \tau) \; .$$

Conversely suppose $u \in M_n(\widetilde{A})$ is a unitary satisfying $u^\tau = u$. There is a unique unital homomorphism $\phi \colon C(S^1, \C) \rightarrow M_n(\widetilde{A})$ such that $\phi(u_0) = u$,
and it is easily seen that $\phi$ satisfies $\phi(x^{{\id}}) = \phi(x)^\tau$ for all $x$. Thus $\phi$ is a \ctalg ~homomorphism
$\phi \colon (C(S^1), {\id}) \rightarrow (M_n(\widetilde{A}), \tau)$. The restriction yields a homomorphism 
$\phi \colon (C_0(S^1), {\id}) 
  \rightarrow (M_n(A), \tau) \subset (\K\pr \otimes\sr A, \tau)$. 
This construction gives an inverse to $\Theta$.
\end{proof}

\begin{cor}
For any \ctalg~$(A, \tau)$, there is a natural isomorphism $$KO^u_{-1}(A, \tau) \cong KO_{-1}(A, \tau) \; .$$
\end{cor}

\begin{proof}
This follows from Proposition~\ref{K-1unitpicture} above and Theorem~\ref{thm:classify} (remembering that 
$A_{-1} = (C_0(S^1 \setminus \{1\}), {\id})$).
\end{proof}

\subsection{$KO_{3}$ via unitaries}

\begin{defn} \label{K3def}
Let $(A, \tau)$ be a unital \ctalg. Let $U_{\infty}^{(3)}(A, \tau)$ be the set of all unitaries $u$ in $\cup_{n \in \N} M_{2n}(A)$ satisfying $u^{\sharp \otimes \tau} = u$. Let $\sim_3$ be the equivalence relation on $U^{(3)}_{\infty}(A, \tau)$, generated by 
\begin{enumerate}
\item[(1)] $u_0 \sim_3 u_1$ if $u_t \in M_{2n}(A)$ is a continuous path of unitaries satisfying $u_t^{\sharp \otimes \tau} = u_t$; 
and 
\item[(2)] $u \sim_3 \iota^{(3)}_n(u)$ for $u \in M_{2n}(A)$ where 
$\iota_n^{(3)} \colon M_{2n}(A) \rightarrow M_{2n+2}(A)$ is given by
$$\iota_n^{(3)} = {{\diag}} \left(a, 1_2  \right) \; . $$
\end{enumerate}
Then we define $KO^u_3(A, \tau) = U_{\infty}^{(3)}(A, \tau)/\sim_3$, with a binary operation given by $[u] + [v] = \left[\sm{u}{0}{0}{v}\right]$
for $u,v \in U_{\infty}^{(3)}(A, \tau)$.
\end{defn}

Here we are using the same convention on the involution $\sharp \otimes \tau$ as we did in Definition~\ref{K4def} for $KO^u_4(A, \tau)$.

\begin{prop} \label{K3iso}
If $(A, \tau)$ is a unital \ctalg, then $KO^u_3(A, \tau) \cong KO_3(A, \tau)$. In particular, $KO^u_3(\C, {\id}) = 0$.
\end{prop}

\begin{proof}
An element of $U_\infty^{(3)}(A, \tau)$ is given by a unitary $u \in M_2(\C) \otimes M_n(A)$ that satisfies $u^{\sharp \otimes \tau} = u$. This is the same as an element of $U_\infty^{(-1)}(M_2(\C) \otimes A, \sharp \otimes \tau)$. Therefore
$U_\infty^{(3)}(A, \tau) \cong U_\infty^{(-1)}(M_2(\C) \otimes A, \sharp \otimes \tau)$; and hence $KO^u_3(A, \tau) = KO^u_{-1}(M_2(\C) \otimes A, \sharp \otimes \tau)$.

As a special case of the K\"unneth formula (the Main Theorem of \cite{boersema02}), for a real \calg~ $A$
we know that $KO_n(A) \cong KO_{n+4}(\H \otimes A)$. In terms of \ctalg s, this is the statement that $KO_n(A, \tau) \cong KO_{n+4}(M_2(\C) \otimes A, \sharp \otimes \tau)$.  Therefore
\begin{align*}
KO_3^u(A, \tau) &\cong KO_{-1}^u(M_2(\C) \otimes A, \sharp \otimes \tau) \\
  & \cong KO_{-1}(M_2(\C) \otimes A, \sharp \otimes \tau) \\
  & \cong KO_3(A, \tau) .
\end{align*}
\end{proof}

\begin{prop} \label{K3group}
$KO_{3}^u(A, \tau)$ is a homotopy invariant functor from the category of unital \ctalg s to the category of abelian groups. The inverse of an element $[u]$ in $KO_{3}^u(A, \tau)$ is given by $[u^*]$.
\end{prop}

\begin{proof}
This follows immediately from Proposition~\ref{K3iso}. The statement about inverses follows from 
Proposition~\ref{K-1group} and
the statement in the proof of Proposition~\ref{K3iso} that $U_\infty^{(3)}(A, \tau) 
\cong U_\infty^{(-1)}(M_2(\C) \otimes A, \sharp \otimes \tau)$.
\end{proof}

\begin{defn}
Let $(A, \tau)$ be any \ctalg . Then we define 
$$KO^u_3(A, \tau) = \ker(\lambda_*)$$
where $\lambda_* \colon KO^u_3(\widetilde{A}, \tau) \rightarrow KO^u_3(\C, {\id})$.
\end{defn}

Combining Proposition~\ref{K3group} with this definition gives the following.

\begin{prop}
If $(A, \tau)$ is any \ctalg, then 
$$KO^u_3(A, \tau) \cong KO_3(A, \tau) \; .$$
\end{prop}

\begin{prop} \label{K3unitpicture}
Let $(A, \tau)$ be a \ctalg. Any element of $KO^u_3(A, \tau)$ can be represented as $[u]$ where $u \in M_n(\widetilde{A})$ satisfies $u^{\sharp \otimes \tau} = u$ and $\lambda_*(u) = 1_n$.
\end{prop}

\begin{proof}
Let $u \in U_3(\widetilde{A}, \tau) = U_{-1}(M_2(\C) \otimes \widetilde{A}, \sharp \otimes \tau)$. The unital homomorphism $(M_2(\C) \otimes A)^{\sim} \hookrightarrow M_2(\C) \otimes \widetilde{A}$ induces an isomorphism on $KO^u_{-1}(-)$. Therefore, using Proposition~\ref{K-1unitpicture}, we can replace $u$ by an equivalent unitary $v$ in $M_n(M_2(\C) \otimes \widetilde{A}) \subset U_{-1}(M_2(\C) \otimes \widetilde{A}, \sharp \otimes \tau)$ that satisfies $\lambda_n(v) = \1_{2n}$ where 
$\lambda \colon (M_2(\C) \otimes A)^\sim \rightarrow \C$ and 
$\lambda_n \colon M_n \left( (M_2(\C) \otimes A)^\sim \right) \rightarrow M_n(\C)$.

Now, we consider the same unitary $v$ as an element in $U_3(\widetilde{A}, \tau)$. In that context it is a unitary in 
$M_{2n}(\widetilde{A})$ and it
satisfies $\lambda_{2n}(v) = \1_{2n}$ where 
$\lambda \colon \widetilde{A} \rightarrow \C$ and
$\lambda_{2n} \colon M_{2n}(\widetilde{A}) \rightarrow M_{2n}(\C)$.
\end{proof}

\subsection{$KO_{5}$ via unitaries}

\begin{defn}
Let $(A, \tau)$ be a unital \ctalg. Let $U_{\infty}^{(5)}(A, \tau)$ be the set of all unitaries $u$ in $\cup_{n \in \N} M_{2n}(A)$ satisfying $u^{\sharp \otimes \tau} = u^*$. Let $\sim_5$ be the equivalence relation on $U^{(5)}_{\infty}(A, \tau)$, generated by 
\begin{enumerate}
\item[(1)] $u_0 \sim_5 u_1$ if $u_t \in M_{2n}(A)$ is a continuous path of unitaries satisfying $u_t^{\sharp \otimes \tau} = u_t^*$; 
and 
\item[(2)] $u \sim_5 \iota^{(5)}_n(u)$ for $u \in M_{2n}(A)$ where 
$\iota^{(5)}_n \colon M_{2n}(A) \rightarrow M_{2n+2}(A)$ is given by
$$\iota^{(5)}_n(a) = {\diag} \left(a, 1_2  \right) \; . $$
\end{enumerate}
Then we define $KO^u_5(A, \tau) = U_{\infty}^{(5)}(A, \tau)/\sim_5$, with a binary operation given by $[u] + [v] = \left[\sm{u}{0}{0}{v}\right]$
for $u,v \in U_{\infty}^{(5)}(A, \tau)$.
\end{defn}

In this definition we use the same convention for $\sharp \otimes \tau$ as discussed for Definition~\ref{K3def}.

\begin{prop} \label{K5iso}
If $(A, \tau)$ is a unital \ctalg, then $KO^u_5(A, \tau) \cong KO_5(A, \tau)$.
In particular, $KO_5^u(\C, {\id}) = 0$.
\end{prop}

\begin{proof}
The proof is the same as the proof of Proposition~\ref{K3iso}, using 
$$KO_5^u(A, \tau) = KO_1^u(M_2(\C) \otimes A, \sharp \otimes \tau) \; .$$
\end{proof}

\begin{prop} \label{K5group}
$KO_{5}^u(A, \tau)$ is a homotopy invariant functor from the category of unital \ctalg s to the category of abelian groups. The inverse of an element $[u]$ in $KO_{5}^u(A, \tau)$ is given by $[u^*]$.
\end{prop}

\begin{proof}
This follows from Proposition~\ref{K5iso}.
\end{proof}

\begin{defn}
Let $(A, \tau)$ be any \ctalg . Then we define 
$$KO^u_5(A, \tau) = \ker(\lambda_*)$$
where $\lambda_* \colon KO^u_5(\widetilde{A}, \tau) \rightarrow KO^u_5(\C, {\id})$.
\end{defn}

The following result is immediate from our development so far.
\begin{prop}
If $(A, \tau)$ is any \ctalg, then 
$$KO^u_5(A, \tau) \cong KO_5(A, \tau) \; .$$
\end{prop}

\begin{prop} \label{K5unitpicture}
Let $(A, \tau)$ be a \ctalg. Any element of $KO^u_5(A, \tau)$ can be represented as $[u]$ where $u \in M_n(\widetilde{A})$ satisfies $u^{\sharp \otimes \tau} = u^*$ and $\lambda_*(u) = \1_n$.
\end{prop}

\begin{proof}
There is a proof similar to that of Proposition~\ref{K3unitpicture}, but instead we give the following slightly more constructive proof.

Let $u \in M_{2n}(\widetilde{A}, \tau)$ be a unitary satisfying $u^{\sharp \otimes \tau} = u^*$.
Then $\lambda(u)$ is a unitary in $M_{2n}(\C)$ satisfying $u^{\sharp} = u^*$, which is to say that $u$ is a unitary in $M_n(\H)$. Since the unitary group of $M_n(\H)$ is connected (this follows for example from Theorem~1 of \cite{wiegmann58}), there exists a path $v_t$ from $\1_{2n}$ to $\lambda(u)$ in $M_{2n}(\C)$ satisfying $v_t^{\sharp} = v_t^*$. Then $u v_t^*$ is the desired path from $u$ to a unitary $w = u v_1^*$ satisfying $\lambda(w) = \1_n$.
\end{proof}

\section{Summary and examples} \label{sec:summary}

The following theorem Figure~\ref{table:summary} summarize the unitary description of $KO$-theory from the previous two sections. The statements about $KU$ will be clarified later in this section.

\begin{figure}[t]
\caption{Unitary Picture of $K$-theory --- The Ten-Fold Way} \label{table:summary} 
\begin{tabular}{c|c|c|c|c|} 
\hhline{~----}
& K-group \TT \BB & $n_i$ & $\mathscr{S}_i$ & $I^{(i)}$ \\ \hhline{=====}
\multicolumn{1}{|c|}{\multirow{2}{*}{complex}}  & $KU^u_0(A, \tau)$ \TT \BB & 2 & $u = u^*$ & $\sm{\1}{0}{0}{-\1}$ \\ \hhline{~----}
\multicolumn{1}{|c|}{} & $KU^u_1(A,\tau)$ \TT \BB & 1 & -- & $\1$ \\ \hhline{=====}
\multicolumn{1}{|c|}{\multirow{8}{*}{real}}  & $KO^u_{-1}(A,\tau)$ \TT \BB & 1 & $u^\tau = u $ & $\1$ \\ \hhline{~----}
\multicolumn{1}{|c|}{} & $KO^u_0(A,\tau)$ \TT \BB & 2 &  $u = u^*$, $u^{\tau} = u^* $ & $\sm{\1}{0}{0}{-\1}$  \\ \hhline{~----}
\multicolumn{1}{|c|}{} & $KO^u_1(A,\tau)$ \TT \BB & 1 &  $u^{\tau} = u^* $ & $\1$ \\ \hhline{~----}
\multicolumn{1}{|c|}{} & $KO^u_2(A,\tau)$ \TT \BB & 2 &  $u = u^*$, $u^\tau = -u$ & $\sm{0}{i \cdot \1}{-i \cdot \1}{0}$ \\\hhline{~----}
\multicolumn{1}{|c|}{} & $KO^u_{3}(A,\tau)$ \TT \BB & 2 & $u^{\sharp \otimes \tau} = u$ & $ \1_2 $ \\ \hhline{~----}
\multicolumn{1}{|c|}{} & $KO^u_4(A,\tau)$ \TT \BB & 4 &  $u = u^*$, $u^{\sharp \otimes \tau} = u^*$ & ${\diag}(\1_2,-\1_2)$ \\ \hhline{~----}
\multicolumn{1}{|c|}{} & $KO^u_5(A,\tau)$ \TT \BB & 2 &  $u^{\sharp \otimes \tau} = u^*$ & $\1_2$ \\ \hhline{~----} 
\multicolumn{1}{|c|}{} & $KO^u_6(A,\tau)$ \TT \BB & 2 &  $u = u^*$, $u^{\sharp \otimes \tau} = -u$ & $\sm{0}{i \cdot \1}{-i \cdot \1}{0}$ \\ \hhline{=====}
\end{tabular}
\caption*{
In the unitary picture, the $K$-theory of a \ctalg ~$(A, \tau)$ consists of unitaries in matrix algebras over $A$ satisfying the symmetry $\mathscr{S}_i$. See Theorem~\ref{thm:summary}.}
\end{figure}

\begin{thm} \label{thm:summary}
Let $(A, \tau)$ be a \ctalg, not necessarily unital. 
Let $n_i$ be the positive integer, $\mathscr{S}_i$ be the symmetry relation, and $I^{(i)} \in M_{n_i}(\C)$ be the neutral element as specified in Figure~\ref{table:summary}, for $i \in \{-1, 0, \dots, 6\}$.

Then there exist natural isomorphisms
$KO_i^u(A, \tau) \cong KO_i(A^\tau)$ for all $i$,
where $KO^u_i(A, \tau)$ is defined to be group of equivalence classes of unitaries $u$ in
$\cup_{n \in \N} M_{n_i \cdot n}(\widetilde{A})$ that satisfy $\mathscr{S}_i$ and satisfy
$\lambda(u) = {\diag}(I^{(i)}, \dots, I^{(i)})$. The equivalence relation is generated by path homotopy (within unitaries satisfying 
$\mathscr{S}_i$) and by the relation
$u \sim {\diag}(u, I^{(i)})$.
The binary operation is defined by $[u] + [v] = [{\diag}(u,v)]$.

Similar statements are made for $KU^u_i(A, \tau)$.
\end{thm}

\begin{remark} \label{rem:inverses} The inverse of an element $[u] \in KO_i^u(A, \tau)$ is given by $[u^*]$ if $i$ is odd.  In the even 
case, the inverse of $[u] \in KO_i^u(A, \tau)$ is $[-u]$ when $i = 0, 4$; or when $i=2, 6$ and 
$u \in M_{n_i \cdot n}(\widetilde{A})$ with $n$ even.
\end{remark}  

\begin{remark} \label{rem:lambda} The restriction that $\lambda(u) = I^{(i)}_n = {\diag}(I^{(i)}, \dots, I^{(i)})$ could be replaced by the weaker condition that $[\lambda(u)] = 0 \in KO^u_I(\C, {\id}) = KO^u_i(\R)$. We have shown in each case that a representative of $KO_i^u(A, \tau)$ can always be found that satisfies the stronger $\lambda$ condition. We have not proven, but we believe to be true in each of the ten cases, that the equivalence relation can be taken to be path homotopy not only within unitaries satisfying the appropriate symmetry, but also within unitaries satisfying the stronger condition $\lambda(u) = {\diag}(I^{(i)}, \dots, I^{(i)})$.
\end{remark}

\begin{remark} \label{rem:unit}
If $(A, \tau)$ is already a unital \ctalg , it is not necessary to work in the unitization $\widetilde{A}$. 
We can realize $KO_i^u(A, \tau)$ using unitaries in $M_{n_i}(A)$ satisfying the correct symmetries (and without any $\lambda$ restriction). The isomorphism between the picture of 
$KO_i^u(A)$ given by unitaries in $M_{n_i}(A)$ and that given by unitaries in $M_{n_i}(\widetilde{A})$ is given by $[u] \mapsto [u-I_n^{(i)} \cdot (1_A)_n + I_n^{(i)} \cdot \1_n]$.
\end{remark}

\begin{remark} \label{rem:matrix}
In the cases where the matrices are required to be of dimensions that are multiples of 2 or 4, it is possible to write down a picture of $K$-theory so that any unitary (satisfying the symmetry) in any dimension of square matrices represents a $K$-class. However, this would require a carefully specified and consistent choice of each of the embeddings $M_n(A) \hookrightarrow M_{n+1}(A)$. 
These choices would not be canonical and the designated ``neutral element'' would look different for different values of $n$.
\end{remark}

\begin{remark} \label{rem:conj}
Let $u$ be a unitary in $M_n(\widetilde{A})$ and let $x \in O(n) \subset M_n(\R)$. If $u$ satisfies any symmetry 
$\mathscr{S}_i$ for $-1 \leq i \leq 2$, then so does $xux^*$. Furthermore, $[u] = [xux^*] \in KO^u_i(A)$ if $x \in SO(n)$ (since $SO(n)$ is connected). The equality $[u] = [xux^*]$ also holds if $x \in O(n)$ and $-1 \leq i \leq 1$. Indeed, if $\det x = -1$, then 
${\diag}(x, \1, -\1) \in SO(n+2)$ and ${\diag}(x, -\1) \in SO(n+1)$ so we have 
\begin{align*}
[u] = [{\diag}(u, \1, -\1)] &= [{\diag}(x, \1, -\1) \cdot {\diag}(u, \1, -\1) \cdot {\diag}(x^*, \1, -\1)] \\
	&= [{\diag}(xux^*, \1, -\1)] = [xux^*]  \qquad \text{(for $i = 0$)} \; . \\
[u] = [{\diag}(u, \1)] &= [{\diag}(x, -\1) \cdot {\diag}(u, \1) \cdot {\diag}(x^*, -\1)] \\
	&= [{\diag}(xux^*, \1)] = [xux^*]  \qquad \text{(for $i = \pm 1$).}
\end{align*}
However, for $i = 2$, we may have $[u] \neq [xux^*] \in KO^u_2(A)$ if $x \in O(n)$.

Similar comments hold for $3 \leq i \leq 6$, with respect to conjugation by elements in the image of the injective homomorphism
$O(n) \hookrightarrow O(2n)$ or $SO(n) \hookrightarrow SO(2n)$ given by
$$\begin{pmatrix}  
x_{1\,1} & x_{1\, 2} & \dots & x_{1 \, n} \\
x_{2 \, 1} & x_{2 \, 2} & \dots & x_{2 \, n} \\
\vdots & \vdots & \ddots & \vdots \\
x_{n \, 1} & x_{n \, 2} & \dots & x_{n \, n } \\
\end{pmatrix} 
\mapsto
\begin{pmatrix}  
x_{1\,1} \1_2 & x_{1\, 2} \1_2  & \dots & x_{1 \, n} \1_2  \\
x_{2 \, 1} \1_2  & x_{2 \, 2} \1_2  & \dots & x_{2 \, n} \1_2  \\
\vdots & \vdots & \ddots & \vdots \\
x_{n \, 1} \1_2  & x_{n \, 2} \1_2  & \dots & x_{n \, n } \1_2  \\
\end{pmatrix} 
$$
\end{remark}

\subsection{$KU_i^u(A)$ for real \calg s}
First note that Definitions~\ref{defn:KO_0} and \ref{defn:KO_1} carry over to the complex setting and give pictures of the $K$-theory groups $K^u_0(A)$ and $K^u_1(A)$ for any complex \calg~ $A$. Specifically, $K^u_0(A)$ is given in terms of self-adjoint unitaries in $M_{2n}(A)$ and $K^u_1(A)$ is given in terms of unitaries in $M_n(A)$. The same proofs carry over to show that $K^u_0(A) \cong K_0(A)$ and $K^u_1(A) \cong K_1(A)$ for any complex \calg~ $A$.

Following the convention in \cite{boersema02}, \cite{boersema04}, and later papers; we define $KU_i(A)$ as a functor on the category of real \calg~ via complexification as follows:
\begin{defn}
For a real \calg~$A$, define $KU_i(A) = K_i(A\sc) \cong K^u_i(A\sc)$ for $i = 0,1$. 
\end{defn}

Alternatively, if $(A, \tau)$ is a \ctalg , then we have $KU_i(A, \tau) \cong K_i(A)$ since $A$ is exactly the complexification of the real \calg~$A^\tau$ corresponding to $(A, \tau)$. Thus we end up with unitary pictures $KU_0^u(A, \tau)$ and $KU_1^u(A, \tau)$ in terms of self-adjoint unitaries and unitaries in matrix algebras over $A$, exactly as described in Theorem~\ref{thm:summary} referring to the first two lines of Table~\ref{table:summary}. The fact that the symmetry relations for $KU_i^u(A, \tau)$ do not actually involve 
$\tau$ reflects the fact that these groups depend only on the underlying \calg~ $A$ and not on the real structure imposed on $A$.

In fact, for each $i$ there is a natural transformation
$$c_i^u \colon KO_i^u(A, \tau) \rightarrow 
\begin{cases} KU_0^u(A, \tau) & \text{for $i$ even} \\
	      KU_1^u(A, \tau) & \text{for $i$ odd.} \end{cases}
$$
defined simply by $[u] \mapsto [u]$. These are clearly well-defined and natural, since in each case we are forgetting the extra symmetry requirement involving $\tau$. To simplify notation, we define $KU_i^u(A, \tau)$ for all $i$ by 
$$KU_i^u(A, \tau) = \begin{cases} KU_0^u(A, \tau) & \text{for $i$ even} \\
	      KU_1^u(A, \tau) & \text{for $i$ odd} \end{cases}$$
so that we can simply write 
$c_i^u \colon KO_i^u(A, \tau) \rightarrow KU_i^u(A,\tau)$ in all cases. 

We will verify in Proposition~\ref{prop:c=c} below that, for any real \calg~$B$, the homomorphism $c_i^u$ coincides with the frequently used homomorphism
$c_i \colon KO_i(B) \rightarrow KU_i(B)$ induced by the injective *-algebra homomorphism $c \colon B \rightarrow B\sc$,
This natural transformation appears for example in Section~1.4 of \cite{schroderbook} and Section~1.2 of \cite{boersema02}
and forms one of the maps of the crucial long exact sequence relating real and complex $K$-theory.

First, note that the complexification functor $B \rightsquigarrow B\sc$, rephrased in terms of \ctalg s, is equivalent to the functor 
$(A, \tau) \rightsquigarrow (A \oplus A, \sigma)$ where $(a_1, a_2)^\sigma = (a_2^\tau, a_1^\tau)$. The *-homomorphism $c \colon B \rightarrow B\sc$, rephrased in terms of \ctalg s, is the injective $(*, \tau)$-homomorphism 
$\widetilde{c} \colon (A, \tau) \rightarrow (A \oplus A, \sigma)$ given by
$\widetilde{c}(a) = (a,a)$.  
To verify these claims, one can verify that the restricted map $\widetilde{c} \colon A^\tau \rightarrow (A \oplus A)^{\sigma}$ is the same, up to isomorphism, as the canonical inclusion of $A^\tau$ into its complexification $A$. Indeed,
$$A^\tau = \{a \in A \mid a^* = a^\tau \}$$ 
and 
$$	(A \oplus A)^\sigma = \{(a_1, a_2) \in A \oplus A \mid (a_1, a_2)^\sigma = (a_1,a_2)^* \} 
			= \{(a_1, a_1^{* \tau}) \mid d \in A \}  \cong A \; .
$$

\begin{lemma} \label{lem:KOKU}
Let $(A, \tau)$ be a \ctalg. Consider the \ctalg ~$(A \oplus A, \sigma)$. Then there is an isomorphism
$$\Gamma \colon KU^u_i(A, \tau) \rightarrow KO^u_i(A \oplus A, \sigma) \; .$$
\end{lemma}

\begin{proof}
For $i = 0,1$, define a homomorphism by $[x] \mapsto [(x, x^{* \tau})]$, and check that it is well-defined and is a bijection on the appropriate symmetry classes of unitaries.  For $i= -1$ use $[x] \mapsto [(x, x^{\tau})]$, and for $i = 2$ 
use $[x] \mapsto [(x, -x^\tau)]$.  For $i = 3, 4, 5, 6$, use the same formulas replacing $\tau$ with $\sharp \otimes \tau$.
\end{proof}

\begin{lemma} \label{lem:diagram}
The diagram below commutes. 
\end{lemma}
$$\xymatrix{
KO_i^u(A, \tau) \ar[r]^{c_i^u}  \ar[dr]_{\widetilde{c}_*}
& KU_i^u(A, \tau) \ar[d]^\Gamma \\
& KO_i^u(A \oplus A, \sigma)
}$$

\begin{proof}
Let $[x] \in KO_i^u(A, \tau)$. For each $i$, we use the formulas $\widetilde{c}_i([x]) = [(x,x)]$ and $c_i([x]) = [x]$, and the formula for $\theta([x])$ given in the proof of Lemma~\ref{lem:KOKU}. Combining these formulas with the symmetries that $x$ is assumed to satisfy, it follows that the diagram commutes for each $i$.
\end{proof}

\begin{prop} \label{prop:c=c}
Let $(A, \tau)$ be a \ctalg ~and let $A^\tau$ be the corresponding real \calg . Then 
$c_i^u \colon KO_i^u(A, \tau) \rightarrow KU_i^u(A, \tau)$ corresponds to the natural transformation 
$c_i \colon KO_i(A^\tau) \rightarrow KU_i(A^\tau)$ via the identifications 
$KO_i^u(A, \tau) \cong KO_i(A^\tau)$ and $KU_i^u(A, \tau) \cong KU_i(A^\tau)$.
\end{prop}

\begin{proof}
The claim is that for any \ctalg ~$(A, \tau)$ the diagram
$$\xymatrix{
KO_i^u(A, \tau) \ar[r]^{c_i^u}  \ar[d] 
& KU_i^u(A, \tau) \ar[d] \\
KO_i(A, \tau) \ar[r]^{c_i} \ar@{=}[d]
 & KU_i(A, \tau)  \ar@{=}[d] \\
 KO_i(A^\tau) \ar[r]^{c_i} 
 & KU_i(A^\tau)  \\
}$$
commutes, where the vertical arrows represent the appropriate natural isomorphisms from the previous sections.

By Lemma~\ref{lem:diagram}, it follows that $c_i^u$ is equivalent to the natural homomorphism induced by the homomorphism $A \hookrightarrow A\sc$. Since $c_i$ is induced by the same homomorphism, and since the isomorphisms $KO_i^u(A, \tau) \cong KO_i(A^\tau)$ and $KU_i^u(A, \tau) \cong KU_i(A^\tau)$ are natural with respect to real *-algebra homomorphisms, the result follows.
\end{proof}

\subsection{Examples}

The following theorem will identify for each $i$ a specific unitary that generates the $K$-theory for each of the corresponding classifying algebra $A_i$.
Let
$$Q  = \tfrac{1}{\sqrt{2}} \left( \begin{matrix} 
	1 & 0 & 0 & -i \\ 0 & 1 & i & 0 \\ 0 & i & 1 & 0 \\ -i & 0 & 0 & 1  
					\end{matrix} \right)
\in M_4(\C)$$
be the unitary matrix (from Lemma~1.3 of \cite{hastingsloring}) which facilitates an equivalence between the involutions $\widetilde{\sharp} \otimes \sharp$ and $\tr_4$. That is, 
the equations
$Q x^\tr Q^* = (Q x Q^*)^{\widetilde{\sharp} \otimes \sharp}$
and $Q^* x^{\widetilde{\sharp} \otimes \sharp} Q = (Q^* x Q)^\tr$
hold for all $x \in M_4(A)$ for any \calg~$A$.

\begin{example} \label{ex:Aigenerators}
For each $i$, the class of a generator of $KO^u_i(A_i) \cong \Z$ is given by a unitary element $x_i$ as described below.

\begin{itemize} 
\item $[x_{-1}] \in KO^u_{-1}(A_{-1})$. The associated \ctalg ~is $(C_0(S^1 \setminus \{1\}, \C), {\id})$.
The unitary is
$$x_{-1} = z \in C(S^1, \C)$$
which satsifies $(x_{-1})^{\id} = x_{-1}$.

\item $[x_0] \in KO^u_0(A_0)$. The associated \ctalg ~is, $(q\C, \tr)$.
The unitary is
$$x_0 =
      \left( \begin{smallmatrix}
  1 - 2t & 0 & 0 & 2 \sqrt{t-t^2} \\
  0 & 1 & 0 & 0 \\
  0 & 0 & -1 & 0 \\
  2\sqrt{t-t^2} & 0 & 0 & 2t - 1
    \end{smallmatrix} \right)
\in M_2(\widetilde{q\C}) \; $$
which satisfies $x_0 = x_0^*$ and $x_0^\tr = x_0^*$.

\item $[x_1] \in KO^u_1(A_1)$.   The associated \ctalg ~is $(C_0(S^1 \setminus \{1\}, \C), \zeta)$.
The unitary is
$$x_{1} = z \in C(S^1, \C)$$
which satsifies $(x_{1})^\zeta = x_{1}^*$.

\item $[x_2] \in KO^u_2(A_2)$. The associated \ctalg ~is, $(q\C, \sharp)$.
The unitary is
$$x_2 =  Wx_0 W^*
\in M_2(\widetilde{q\C}) \; $$
which satisfies $x_2 = x_2^*$ and $x_2^\sharp = -x_2$ (where $W$ is as in Section~\ref{sec:prelim}).  

\item $[x_3] \in KO^u_3(A_3)$. The associated \ctalg ~is 
  $(M_2(\C) \otimes C_0(S^1 \setminus \{1\}, \C), \sharp \otimes {\id})$.
  The unitary is 
$$x_3 = Q {\diag}(z, 1, 1, 1) Q^* \in M_2(\C) \otimes M_2(\C) \otimes C_0(S^1, \C) $$
which satisfies $(x_3)^{\widetilde{\sharp} \otimes \sharp \otimes {\id}} = x_3$.   

\item $[x_4] \in KO^u_4(A_4)$. The associated \ctalg ~is, $(M_2(\C) \otimes q\C, \sharp \otimes \tr)$.
The unitary is
$$x_4 = Q {\diag}(x_0, 1, 1, -1, -1) Q^* \in M_2(\C) \otimes M_2(\C) \otimes \widetilde{q\C}$$
which satisfies $(x_4)^{\widetilde{\sharp} \otimes \sharp \otimes \tr} = x_4^*$.

\item $[x_5] \in KO^u_5(A_5)$. The associated \ctalg ~is 
   $(M_2(\C) \otimes C_0(S^1 \setminus \{1\}, \C), \sharp \otimes \zeta)$.
The unitary is 
$$x_5 = Q {\diag}(z, 1, 1, 1) Q^* \in M_2(\C) \otimes M_2(\C) \otimes C_0(S^1, \C) $$
which satisfies $(x_5)^{\widetilde{\sharp} \otimes \sharp \otimes \zeta} = x_5^*$.   
   
\item $[x_6] \in KO^u_6(A_6)$. The associated \ctalg ~is, $(q\C, \trt)$. 
The unitary is
$$x_6 = x_0
\in M_2(\widetilde{q\C}) \; $$
which satisfies $x_6 = x_6^*$ and $(x_6)^{\sharp \otimes \trt} = -x_6$.
\end{itemize}

\begin{proof}
For each $i$, we know from Propositions~\ref{thm:KAieven} and~\ref{thm:KAiodd} that $KO_i^u(A_i) \cong \Z$ and that $c_i \colon KO_i^u(A_i) \rightarrow KU_i^u(A_i)$ is an isomorphism.

The statement that $[x_0]$ generates $KO_0^u(A_0) \cong \Z$ is Proposition~\ref{prop:u0}. It follows that 
$c^u_0[x_0] = [x_0]$ generates $KU_0^u(A_0) \cong \Z$. (This fact can also be derived from the fact that $[\tfrac{1}{2}(x_0+\1)] - [\1]$ is the generator of $KU_0(q\C)$ as in Section~3 of \cite{loring08}.)  Now for $i = 2, 4, 6$, we also have $c^u_i([x_i]) = [x_0] \in KU_i^u(A_i) \cong KU_0^u(q\C)$. Since $c^u_i$ is an isomorphism on $KO^u_i(A_i)$ it follows that $[x_i]$ must be generator of $KO^u_i(A_i)$.

For $i$ odd, let $z$ be the identity function on $S^1$. It is known that $[z]$ is a generator of $KU_1^u(C_0(S^1 \setminus \{1\} ))$. Again in each case, we have $c^u_i([x_i]) = [z]$ so it follows that $[x_i]$ is a generator of $KO_i(A_i) \cong \Z$.
\end{proof}
\end{example}

\begin{example} \label{ex:Rgenerators}
Recall that the groups of $KO_*(\R) \cong KO_*(\C, {\id})$, are given by
$$KO_*(\R) = \begin{cases} \Z & i = 0,4 \\ \Z_2 & i = 1,2 \\ 0 & i = 3,5,6,7 \end{cases}$$
Summarizing from discussions in the previous sections, we identify explicit generators for the non-zero groups, with our unitary picture. 

\begin{enumerate}
\item The generator of $KO^u_0(\C, {\id}) \cong \Z$ is $[1_2]$.
\item The generator of $KO^u_1(\C, {\id}) \cong \Z_2$ is $[-1]$.
\item The generator of $KO^u_2(\C, {\id}) \cong \Z_2$ is 
    $[\left( \begin{smallmatrix} 0 & -i \\ i & 0 \end{smallmatrix} \right)]$.
\item The generator of $KO^u_4(\C, {\id}) \cong \Z$ is $[1_4]$.
\end{enumerate}

Notice also that $KU^u_0(\C, {\id}) \cong \Z$ is generated by $[1_2]$. In terms of these generators, it is easy to verify that the homomorphisms $c_i \colon KO^u_i(\C, {\id}) \rightarrow KU^u_i(\C, {\id})$ agree with their known behavior.  For example $c_0 \colon \Z \rightarrow \Z$ is an isomorphism. The class of $[-1]$ is non-trivial in $KO_2^u(\C, {\id})$ but is trivial in $KU_2^u(\C, {\id})$; this corresponds to the fact that $c_2 \colon \Z_2 \rightarrow \Z$ is trivial (as it must be of course). Also, we have that $c_4 \colon \Z \rightarrow \Z$ is multiplication by $2$, since $[1_4]$ is twice the generator of $KU_0^u(\C, {\id})$.

\end{example}

\begin{example} \label{ex:spheregenerators}
Let $\zeta$ be the reflection on $C(S^1, \C)$ or $C(S^{2}, \C)$ corresponding to negation of the $y$ coordinate.
$(x,y,z) \mapsto (x,-y,z)$. Let $1$ denote the point $(1, 0, \dots, 0)$ in $S^{n-1}$ (for $n = 2,3,4$).
\begin{enumerate}
\item The generator of $KO_{1}^u \left(C_{0}(S^{1} \setminus \{1\}), \zeta \right)\cong \Z$
is the class of the unitary $u$, where 
\[
u(x+iy)=x+iy.
\]

\item The generator of $KO_{-1}^{u} \left(C_{0}(S^{1}\setminus \{1\}),{\id} \right)\cong\mathbb{Z}$
is the class of the unitary $u$, where 
\[
u(x,y)=x+iy.
\]

\item The generator of $KO_{0}^u\left(C_{0}(S^{2}\setminus \{1\} ,\zeta\right)\cong\mathbb{Z}$
is the class of the unitary $u$, where 
\[
u(x,y,z)=\left( \begin{array}{cc}
z & x-iy\\
x+iy & -z
\end{array}\right) \; .
\]

\item The generator of $KO_{-2}^{u}\left(C_{0}(S^{2}\setminus  \{1\}) ,{\id} \right) \cong\mathbb{Z}$
is the class of the unitary $u$, where 
\[
u(x,y,z)=\left( \begin{array}{cc}
z & x-iy\\
x+iy & -z
\end{array}\right) \; .
\]

\item The generator of $KO{}_{-3}^{\mathrm{u}}\left(C_{0}(S^{3}\setminus \{1\}) ,{\id} \right)\cong\mathbb{Z}$
is the class of the unitary $u$, where 
\[
u(x,y,z,w)=\left( \begin{array}{cc}
iz-w & ix+y\\
ix-y & -iz-w
\end{array}\right)  \; .
\]

\end{enumerate}

\begin{proof}
Results (1) and (2) are restatements from Example~\ref{ex:Aigenerators}

For $(3)$, first check that $u = u^\zeta = u^*$. The equation $\lambda(u) = I^{(0)}$ does not hold exactly, but it does hold on the level of $KO^u_0(\C)$. Using the isomorphism $\theta$ from Theorem~\ref{K0iso}, 
we have $\theta([u]) = [\tfrac{1}{2}(u + 1_2)] - [1] = [p_0] - [1]$
where $p_0 = \tfrac{1}{2} \sm{1+z}{x-iy}{x+iy}{1-z}$. But we know from the 
discussion preceding Proposition~\ref{prop:betaA} that $KO_0(C_0(S^2 \setminus \{1\}), \zeta) \cong \Z$ is generated by 
$[p_0] - [1]$. This proves (3).

Since $c_0^u \colon KO_0(C_0(S^2 \setminus \{1\}), \zeta) \rightarrow 
	KU_0(C_0(S^2 \setminus \{1\}), \zeta)$ is an isomorphism, it follows that
$[u]$ also generates $KU_0(C_0(S^2 \setminus \{1\}), \zeta) = KU_0(C_0(S^2 \setminus \{1\}), {\id})$.

For (4), check that  $u^{\sharp \otimes {\id}} = -u$. Now we also know that  
$$c_{-2} \colon KO_{-2}^u \left(C_{0}(S^{2}\setminus \{1\} ;{\id} \right) \rightarrow 
	KU_{-2}^u\left(C_{0}(S^{2}\setminus \{1\} ;{\id} \right) $$
is an isomorphism and $[u]$ is the same generator of 
$KU_{-2}^u\left(C_{0}(S^{2}\setminus \{1\} ;{\id} \right)  =KU_0 (C_0(S^2 \setminus \{1 \}); {\id} )$
identified in the previous paragraph. So $[u]$ is also a generator of 
$KO_{-2}(C_{0}(S^{2}\setminus \{1\} ), {\id} ) $.

For (5), check that $u$ is a unitary and that $u^{\sharp \otimes {\id}} = u^*$.
The transformation
\[
c_{-3} :KO_{-3}^u \left(C_{0}(S^{3}\setminus\{1\}),{\id} \right)\rightarrow 
	KU_{1}^u \left(C_{0}(S^{3}\setminus\{1\}),{\id} \right)
\]
is known to be an isomorphism and sends $u$ to the known generator of $\pi_{3}(SO(2))$ so also
to a generator of $KU_{1}\left(C_{0}(S^{3}\setminus\{1\}),{\id} \right)$.
\end{proof}

\end{example}

We will return to specific computations of $KO(-)$ groups and elements of $KO(-)$ represented by unitaries in Section~\ref{sec:examples2}.

\section{The boundary map} \label{sec:boundary}

There exist known natural boundary maps 
$\partial_i \colon KO_i(B, \tau) \rightarrow KO_{i-1}(I, \tau) \; $
for a short exact sequence 
\begin{equation} 0 \rightarrow (I, \tau) \rightarrow (A, \tau) \rightarrow (B, \tau) \rightarrow 0 \; .
\label{eq:SESct} \end{equation} 

In this section,
we will derive concrete formulas for these boundary maps described in term of the unitary pictures of $K$-theory.
The approach we will take is to first write down specific formulas for maps
\[
\clubsuit_i :KO_{i}^{u}(B,\tau)\rightarrow KO_{i-1}^{u}(I,\tau) \; ,
\]
then prove that those formulas give well defined and natural homomorphisms, and finally prove that the homomorphisms coincide with $\partial_i$ via the natural isomorphisms $KO_{i}^u(-) \cong KO_{i}(-)$. We start with the odd cases $i = -1, 1, 3, 5$. For each of these the basic formula will be the same, but we will have to conjugate by a different unitary $Y^{(i)}$ in each case in order to obtain a unitary $v$ that is in the correct symmetry class and satisfies $\lambda(u) = I^{(i-1)}_n$. Easier formulas would be possible if we relaxed the $\lambda$ condition (see Remark~\ref{rem:lambda}).

We must introduce notation for several classes of unitaries that will be used for conjugation in these definitions.
First let us recall from Section~\ref{sec:even} that we have the matrix
$$W_{2n} = \tfrac{1}{\sqrt{2}} \begin{pmatrix}
i \cdot 1_{n} & 1_{n} \\ 1_{n} & i \cdot 1_{n}
\end{pmatrix} \in M_{2n}(\C) \; ,$$
which satisfies 
$$W_{2n} \begin{pmatrix} 1_{n} & 0 \\ 0 & -1_{n} \end{pmatrix} W_{2n}^*
	= \begin{pmatrix} 0 & i \cdot 1_{n} \\ -i \cdot 1_{n} & 0 \end{pmatrix} \; .$$
Generalizing the matrix $Q$ used in Section~\ref{sec:summary}, we define
$$Q_{4n} = \tfrac{1}{\sqrt{2}} \left( \begin{matrix} 1_{2n} & -I^{(2)}_n \\ I^{(2)}_n & 1_{2n} \end{matrix} \right) 
    \in M_{4n}(\C) \; .$$
The key property, as in Lemma~1.3 of \cite{hastingsloring}, is that  for all $x \in M_{4n}(A)$ we have
$Q_{4n} x^\tau Q_{4n}^* = (Q_{4n} x Q_{4n}^*)^{\widetilde{\sharp} \otimes \sharp \otimes \tau}$ 
where $\widetilde{\sharp}$ and $\sharp$ are as defined in Section~\ref{sec:prelim}.

Let $V_{2n} \in M_{2n}(\R)$ be the unique permutation matrix such that (for diagonal matrices) we have
$$V_{2n} {\diag}(\lambda_1, \dots, \lambda_{2n} ) V_{2n}^*
      = {\diag}(\lambda_1, \lambda_{n+1}, \lambda_2, \lambda_{n+2}, \dots, \lambda_{n}, \lambda_{2n}) \; .$$
Then for all matrices $x \in M_{2n}(A)$ we have
$$V_{2n} x^{\widetilde{\sharp} \otimes \tau} V_{2n}^* = (V_{2n} xV_{2n}^*)^{\sharp \otimes \tau} \; $$
where $\sharp$ and $\widetilde{\sharp}$ are defined as in Section~\ref{sec:prelim}.
Similarly, let
$X_{4n} \in M_{4n}(\R)$ be the unique permutation matrix such that (for diagonal matrices) we have
\begin{align*}
& X_{4n} {\diag}(\lambda_1, \dots, \lambda_{4n} ) X_{4n}^* \\
   & \quad  = {\diag}(\lambda_1, \lambda_2, \lambda_{2n+1}, \lambda_{2n+2}, \lambda_3, \lambda_4, 
  \lambda_{2n+3}, \lambda_{2n+4}, \dots, 
, \lambda_{4n}) \; .
\end{align*}

\begin{defn}
Suppose we have an exact sequence as in Sequence~(\ref{eq:SESct}). 
Let $\pi$ denote both the quotient map
$\pi: A\rightarrow B$ and its extension to 
$M_{n}(\widetilde{A}) \rightarrow M_{n}(\widetilde{B})$
for every $n$. Furthermore, we assume $I=\ker(\pi)$ and we identify
the unit in $\widetilde{I}$ with that of $\widetilde{A}$. 

\begin{enumerate}
\item Suppose $[u] \in KO_{1}^u(B, \tau)$ where 
$u \in M_{n}(\widetilde{B})$ is a unitary with $u^{\tau}=u^{*}$
and $\lambda(u)=I^{(1)}_n$. Then define 
\[  
\clubsuit_{1}([u])=\left[ Y_{2n}^{(1)} \left(\begin{array}{cc}
2aa^{*}- \1_n & 2a\sqrt{ \1_n -a^{*}a} \\
2a^{*}\sqrt{\1_n-aa^{*}} & \1_n-2a^{*}a
\end{array}\right)  Y_{2n}^{(1)*}  \right]
\in KO_{0}^u(I, \tau)
\]
where $a$ in $M_{n}(\widetilde{A})$ is any lift of $u$ with
$\|a\|\leq1$ and $a^{\tau}=a^{*}$; and $Y_{2n}^{(1)} = V_{2n}$. 

\item Suppose $[u] \in KO_{-1}^u(B, \tau)$ where $u \in M_{n}(\widetilde{B})$
is a unitary with $u^{\tau}=u$
and $\lambda_{n}(u)=I^{(-1)}_n$. 
Then define 
\[
\clubsuit_{-1}([u])=
  \left[ Y_{2n}^{(-1)} \left(\begin{array}{cc}
2aa^{*}-\1_{n} & 2a\sqrt{\1_{n}-a^{*}a}\\
2a^{*}\sqrt{\1_{n}-aa^{*}} & \1_{n}-2a^{*}a
\end{array}\right) Y_{2n}^{(-1)*} \right]
\in KO_6^u(I, \tau)
\]
 where $a$ in $M_{n}(\widetilde{A})$ is any lift of $u$ with
$\|a\|\leq1$ and $a^{\tau}=a$; and
$Y_{2n}^{(-1)} = V_{2n} W_{2n}$.

\item Suppose that $[u] \in KO_{5}^{u}(B, \tau)$ where $u \in M_{2n}(\widetilde{B})$,
is a unitary with $u^{\sharp \otimes \tau}=u^{*}$ and $\lambda_{2n}(u)=I^{(5)}_n$.
Then define 
\[
\clubsuit_{5}([u])=\left[ Y_{4n}^{(5)} \left(\begin{array}{cc}
2aa^{*}-\1_n & 2a\sqrt{\1_n-a^{*}a}\\
2a^{*}\sqrt{\1_n-aa^{*}} & \1_n -2a^{*}a
\end{array}\right)  Y_{4n}^{(5)*} \right]
\in KO^u_4(I, \tau)
\]
 where $a$ in $M_{2n}(\widetilde{A})$ is any lift of $u$ with
$\|a\|\leq1$ and $a^{\sharp \otimes \tau}=a^{*}$;
and $Y_{4n}^{(5)} = X_{4n}$.

\item Suppose that $[u] \in KO_{3}^{u}(B, \tau)$ where $u \in M_{2n}(\widetilde{B})$
is a unitary with $u^{\sharp \otimes \tau}=u$
and $\lambda_{2n}(u)=I^{(3)}_n$. Then define 
\[
\clubsuit_{3}([u])=\left[ Y_{4n}^{(3)} \left(\begin{array}{cc}
2aa^{*}-\1_n & 2a\sqrt{\1_n-a^{*}a}\\
2a^{*}\sqrt{\1_n-aa^{*}} & \1_n-2a^{*}a
\end{array}\right)  Y_{4n}^{(3)*}  \right]
\]
where $a$ in $M_{2n}(\widetilde{A})$ is any lift of $u$ with
$\|a\|\leq1$ and $a^{\sharp \otimes \tau}=a$;
and $Y_{4n}^{(3)} = V_{4n} Q_{4n} W_{4n}$.

\end{enumerate}
\end{defn}

\begin{lemma} \label{lem:clubdefined1}
The maps $\clubsuit_{i}$ are well-defined group homomorphisms, for $i$ odd.
\end{lemma}
\begin{proof}

We need to show that the unitaries constructed all satisfy the correct
relations, that the choice of lift is not important, that some lift
is always available, that homotopy is respected, that embedding into
larger matrices via $\iota_n^{(i)}$ does not effect the outcome, and that
the addition is respected at the level of $K$-theory. 

For convenience, we define 
\[
B(a)=\left(\begin{array}{cc}
2aa^{*}-\1_n & 2a\sqrt{\1_n-a^{*}a}\\
2a^{*}\sqrt{\1_n-aa^{*}} & \1_n-2a^{*}a
\end{array}\right).
\]
Making repeated use of the equality
$2a\sqrt{\mathbb{1}-a^{*}a} = 2\sqrt{\mathbb{1}-aa^{*}}a$,
we find that $B(a)^* = B(a)$ and that 
\begin{small}
\begin{align*}
& B(a)^{2}= \\
&\left(\begin{array}{cc}
\left(2aa^{*}-\1_n \right)^{2}+4a\left(\1_n-a^{*}a\right)a^{*} & 4aa^{*}a\sqrt{\1_n-a^{*}a}-4a\sqrt{\1_n-a^{*}a}a^{*}a\!\! \\
\!\! 4a^{*}\sqrt{\1_n-aa^{*}}aa^{*}-4a^{*}aa^{*}\sqrt{\1_n-aa^{*}} \!\! \!\!
& \left(\1_n-2a^{*}a\right)^{2}+4a^{*}\left(\1_n-aa^{*}\right)a
\end{array}\right)\\
 &   = \1_{2n}
\end{align*}
\end{small}
so that $B(a)$ is always self-adjoint unitary. In each case, we will check that $B(a)$ satisfies the appropriate symmetry based on the symmetry that $a$ satisfies.

(1) First we consider the map for $i = 1$,
\[
\clubsuit_{1}:KO_{1}^{u}(B,\tau)\rightarrow KO_{0}^{u}(I,\tau) \; .
\]
We start with $u$ in $M_{n}(\widetilde{B})$, a unitary
with $u^{\tau}=u^{*}$ and $\lambda_{n}(u)= \1_n$. There exists a lift
$x$ in $M_n(\widetilde{A})$ such that $\pi(x) = u$ and then we necessarily have
$\lambda_{n}(x)=\1_{n}$. Set 
$
y=\tfrac{1}{2}\left(x+x^{\tau *}\right)
$
to obtain the relation $y^{\tau}=y^{*}$. We utilize the
usual function $f(\lambda)=\min(\sqrt{1/\lambda},1)$ and set $a=yf(y^{*}y)$. 
Then $a$ satisfies $\|a\| \leq 1$, $a^\tau = a^*$, and $\pi(a) = u$. The condition $\lambda(a) = \1_n$ follows automatically since $a$ is a lift of $u$. This shows that an appropriate lift exists. 
Now, suppose that $a$ is any suitable lift of $u$ and let $B'(a) = Y_{2n}^{(1)} B(a) Y_{2n}^{(1)*}$.
Then
\[ \lambda(B'(a)) =  
Y_{2n}^{(1)}
\lambda( B(a) )
Y_{2n}^{(1)*} 
=
 Y_{2n}^{(1)} \left(\begin{array}{cc}
\1_n & 0\\
0 & -\1_n
\end{array}\right) Y_{2n}^{(1)*} 
= I^{(0)}_n \; .
\]
This shows that $B'(a) \in M_{2n}(\widetilde{I})$.

Using $\left(a^{*}a\right)^{\tau}=a^{*}a$
and $\left(a a^{*}\right)^{\tau}=aa^{*}$, we have
\begin{align*}
B(a)^{\tau} 
 & =\left(\begin{array}{cc}
\left(2aa^{*}-\mathbb{1}\right)^{\tau} & \left(2a^{*}\sqrt{\mathbb{1}-aa^{*}}\right)^{\tau}\\
\left(2a\sqrt{\mathbb{1}-a^{*}a}\right)^{\tau} & \left(\mathbb{1}-2a^{*}a\right)^{\tau}
\end{array}\right)\\
  & =\left(\begin{array}{cc}
2aa^{*}-\mathbb{1} & 2a\sqrt{\mathbb{1}-a^{*}a}\\
2a^{*}\sqrt{\mathbb{1}-aa^{*}} & \mathbb{1}-2a^{*}a
\end{array}\right)  =B(a) \; .
\end{align*}
Thus $B(a)^\tau = B(a)$. Since $Y_{2n}^{(1)}$ is a real orthogonal matrix, $B'(a) = Y_{2n}^{(1)} B(a) Y_{2n}^{(1)*}$ satisfies the same relation, which means that $B'(a)$ is the right sort of unitary to define an element of $KO_{0}^{u}(I,\tau)$. (In fact, in this case we have $[B(a)] = [B'(a)] \in KO_0^u(I, \tau)$. We used $B'(a)$ in our definition because it satisfies $\lambda(B'(a)) = I_n^{(0)}$.)

If we have two lifts $a_{0}$ and $a_{1}$ with the required relations,
then the straight line 
\[
a_{t}=(1-t)a_{0}+ta_{1}
\]
satisfies the relations at every point, and so $B(a_{t})$ provides
the needed homotopy showing that 
$\left[ B'(a_1) \right]
 = \left[ B'(a_2) \right]$.

Suppose now that we have a homotopy $u_{t}$ of unitaries
in $M_{n}(\widetilde{B})$ satisfying $u_t^\tau = u^*$. 
By working with the surjection
\[
M_{2n}((C[0,1],\widetilde{A}))\rightarrow M_{2n}((C[0,1],\widetilde{B}))
\]
induced by $\pi$, the techniques of the first paragraph of this proof show that there is a lift
$a_t$ of the homotopy $u_t$, which then is a homotopy between a lift of $u_0$ and one for $u_1$. 

To complete the proof that $\clubsuit_1$ is well defined, we need to show that the results of this construction for $u$ and for
\[
v= \iota_n^{(1)}(u) = \left(\begin{array}{cc}
u & 0\\
0 & \1
\end{array}\right)
\]
are the same element of $KO_0^u(I)$.
In fact, this result follows from a special case (taking $v = \1$) of the argument below that $\clubsuit_1$ is additive.

Suppose that 
$u \in M_m(\widetilde{B})$ and $v \in M_n(\widetilde{B})$ are unitaries representing elements in $KO_1^u(B, \tau)$; and let $a$ and $b$ be self-adjoint unitary lifts of $u$ and $v$, respectively, such that $a^\tau = a^*$ and $b^\tau = b^*$. 
Then ${\diag}(a,b)$ is a lift of ${\diag}(u,v)$.
Consider the two matrices
\begin{small}
\begin{align*}
B \left( \begin{matrix} a & 0 \\ 0 & b \end{matrix} \right)
  &=  \left(\begin{array}{cccc}
2aa^* - \1_m & 0 & \!\!2a \sqrt{\1_m - a^* a} & 0 \\
0 & 2bb^* - \1_n & 0 & \!\!2b \sqrt{\1_n - b^* b}  \\
\!\!2a^* \sqrt{\1_m - a a^*} & 0 & \1_m - 2a^* a & 0 \\
0 &\!\! 2b^* \sqrt{\1_n - b b^*}\!\! & 0 &  \1_n - 2b^* b \end{array} \right)   \\
 \text{~and~~~~~~~~~} &  \\
 \left( \begin{matrix} \! B(a) & 0 \\ 0 & \!\!\!\!\!\!B(b) \end{matrix} \right)
  &=  \left(\begin{array}{cccc}
2aa^* - \1_m &\!\! 2a \sqrt{\1_m - a^* a} & 0 & 0 \\
\!\!2a^* \sqrt{\1_m - a a^*} & \1_m - 2a^* a & 0 & 0 \\
0 & 0 & 2bb^* - \1_n &\!\!2b \sqrt{\1_n - b^* b} \!\! \\
0 & 0 &\!\! 2b^* \sqrt{\1_n - b b^*} & \1_n - 2b^* b \\
 \end{array} \right)
\end{align*}
\end{small}
in $M_{2m+2n}(\widetilde{I})$ and observe that
$$V_{2m+2n} B\left( \begin{pmatrix} a & 0 \\ 0 & b \end{pmatrix} \right) V_{2m+2n}^* 
= \begin{pmatrix} V_{2m} B(a) V_{2m}^* & 0 \\ 0 & V_{2n} B(b) V_{2n}^* \end{pmatrix} \; $$
showing that $\clubsuit_1([u]+[v]) = \clubsuit_1([u]) + \clubsuit_1([v])$.

(2) Now we consider
$$
\clubsuit_{-1}:KO_{-1}^{u}(B,\tau)\rightarrow KO_{6}^{u}(I,\tau).
$$
This time, we start with a unitary $u \in M_{n}(\widetilde{B})$ that satisfies
$u^{\tau} = u$ and $\lambda(u) = \1_{n}$. Using a similar construction as in the previous case, we find an element $a \in M_{n}(\widetilde{A})$ that is a lift of $u$, has norm at most $1$, and satisfies $a^\tau = a$.
For any such $a$, define $B'(a) = Y_{2n}^{(-1)} B(a) Y_{2n}^{(-1)*}$.
Then we have 
$$\lambda(B'(a)) = V_{2n} W_{2n} \begin{pmatrix} \1_{n} & 0 \\ 0 & -\1_{n} \end{pmatrix} 											W_{2n}^* V_{2n}^*
  = V_{2n} \begin{pmatrix} 0 & i \1_{n} \\ -i \1_{n} & 0 \end{pmatrix} V_{2n}^*
  = I_{n}^{(6)} \; .$$

We have $B(a)^2 = \1_{2n}$ as before, so $B'(a)^2 = \1_{2n}$. We show that 
$$(W_{2n} B(a) W_{2n}^*)^{\widetilde{\sharp} \otimes \tau} = -W_{2n} B(a) W_{2n}^* $$
 from which it follows that
$(B')^{\sharp \otimes \tau} = -B'$. Indeed,
\begin{align*}
(W_{2n} B(a) W_{2n}^*)^{\widetilde{\sharp} \otimes \tau} 
&= W_{2n} B(a)^{\widetilde{\sharp} \otimes \tau} W_{2n}^*
		    &&  \text{(since $W_{2n}^{\widetilde{\sharp}} = -W_{2n}^*$)} \\
&= W_{2n} (-B(a)) W_{2n}^* &&\text{(using $a^\tau = a$)}  \\
&= -W_{2n} B(a) W_{2n}^*  \; .
\end{align*}
Then the formua $V_{2n} x^{\widetilde{\sharp} \otimes \tau} V_{2n}^* = (V_{2n} xV_{2n}^*)^{\sharp \otimes \tau} \; $ implies that
$B'(a)^{\sharp \otimes \tau} = -B'(a)$. So $[B'(a)]$ is an element of $KO_6(I, \tau)$ as desired.

The proof that $\clubsuit_{-1}$ is independent of the choice of lift and of the homotopy class of $[u]$
is similar to that in the previous case. To show that 
$\clubsuit_{-1}([u]) = \clubsuit_{-1}([{\diag}(u, \1)])$,
we again appeal to a special case of the additivity argument in the next paragraph.

Let $a \in M_{m}(\widetilde{A})$ and $b \in M_{n}(\widetilde{B})$ be lifts of unitaries $u$ and $v$, satisfying $a^\tau = a$ and $b^\tau = b$.
Then check that
\begin{align*}
  B'\begin{pmatrix} a & 0 \\ 0 & b \end{pmatrix}  
    &= V_{2m+2n} W_{2m+2n} B \left( \begin{matrix} a & 0 \\ 0 & b \end{matrix} \right)
      W_{2m+2n}^* V_{2m+2n}^* \\
  &= \begin{pmatrix} V_{2m}W_{2m} B(a) W_{2m}^* V_{2m}^* & 0 \\ 
	      0 & V_{2n}W_{2n} B(b) W_{2n}^* V_{2n}^* \end{pmatrix} \\
  &= \begin{pmatrix} B'(a) & 0 \\ 0 &  B'(b) \end{pmatrix}  \, .
\end{align*}

(3) For 
$$\clubsuit_5 \colon KO_5^u(B, \tau) \rightarrow KO_{4}^u(I, \tau)\; ,$$ 
we will focus on the two crucial aspects: that the proposed element satisfies the symmetries required to be an element of $KO_{4}^u(I, \tau)$ and that it is respects addition. The other aspects are similar to the previous cases.

Start with a unitary $u \in M_{2n}(\widetilde{B})$ satisfying $u^{\sharp \otimes \tau} = u^*$
and $\lambda(u) = I^{(5)}_n = \1_{2n}$ and suppose that $a \in M_{2n}(\widetilde{A})$ satisfies $a^{\sharp \otimes \tau} = a^*$, $\|a\| \leq 1$, 
and $\lambda(a) = \1_{2n}$.
Let $B'(a) = Y_{4n}^{(5)} B(a)  Y_{4n}^{(5)*}$. Then we have
$$\lambda (B'(a)) = Y_{4n}^{(5)} 
    \begin{pmatrix} \1_{2n} & 0 \\ 0 & -\1_{2n} \end{pmatrix} 
	    Y_{4n}^{(5)*} 
	    = {\diag}(\1_2, -\1_2, \dots \1_2, -\1_2)
	    = I_n^{(4)} \; .$$
Using the fact that
$a^{\sharp \otimes \tau} = a^*$, we can show that $B(a)^{\sharp \otimes \tau} = B(a)^*$ 
just as in the proof of $(1)$ (with the involution $\sharp \otimes \tau$ in place of $\tau$).
Conjugation by $Y_{4n}^{(5)}$ rearranges the $2 \times 2$ blocks of the matrix and the action of $\sharp$ is contained within each such block, so we have
$$Y_{4n}^{(5)} x^{\sharp \otimes \tau} Y_{4n}^{(5)*} 
  = \left(Y_{4n}^{(5)} x Y_{4n}^{(5)*} \right) ^{\sharp \otimes \tau} \; .$$
Thus $B'(a)^{\sharp \otimes \tau} = B'(a)^*$. 

Let $a \in M_{2m}(\widetilde{A})$ and $b \in M_{2n}(\widetilde{B})$ be lifts of unitaries $u$ and $v$, satisfying $a^{\sharp \otimes \tau} = a^*$ and $b^{\sharp \otimes \tau} = b^*$. Then
\begin{align*}
  B'\begin{pmatrix} a & 0 \\ 0 & b \end{pmatrix}   
    &= X_{4m+4n} B \begin{pmatrix} a & 0 \\ 0 & b \end{pmatrix}  X_{4m+4n}^* \\
 &= \begin{pmatrix} X_{4m} B(a) X_{4m}^* & 0 \\ 
	      0 & X_{4n} B(b) X_{4n}^* \end{pmatrix} \\
  &= \begin{pmatrix} B'(a) & 0 \\ 0 &  B'(b) \end{pmatrix}  \, .
\end{align*}

(4) Show that $\clubsuit_3 \colon KO_3^u(B, \tau) \rightarrow KO_{2}^u(I, \tau)$ is well-defined.
We will prove that the proposed element satisfies the symmetries required to be an 
element of $KO_{2}^u(I, \tau)$ and that it is additive.

Suppose that $u$ is a unitary in $M_{2n}(\widetilde{B})$ satisfying $u^{\sharp \otimes \tau} = u$ and 
$\lambda(u) = I_n^{(3)} = \1_{2n}$, and that $a \in M_{2n}(\widetilde{A})$ is lift of norm not more than $1$ satisfying the same symmetry.  Let $B(a)$ be as before and let
$B'(a) = Y_{4n}^{(3)} B(a) Y_{4n}^{(3)*}$ where $Y^{(3)}_{4n} = V_{4n} Q_{4n} W_{4n}$. Then 
we have
\begin{align*}
  \lambda (B'(a)) &= Y_{4n}^{(3)} \lambda(B(a)) Y_{4n}^{(3)*}  \\
  &= V_{4n} Q_{4n} W_{4n}{\diag}(\1_{2n}, -\1_{2n}) W_{4n}^* Q_{4n}^* V_{4n}^* \\
  &= \tfrac{1}{2} V_{4n} 
  \begin{pmatrix} \1_{2n} & -I^{(2)}_n \\ I^{(2)}_n & \1_{2n} \end{pmatrix}
   \begin{pmatrix} 0 & i \cdot \1_{2n}   \\ -i \cdot \1_{2n} & 0 \end{pmatrix} 
	 \begin{pmatrix} \1_{2n} & I^{(2)}_n \\ -I^{(2)}_n & \1_{2n} \end{pmatrix}
	 	  V_{4n}^* \\
  &= V_{4n}  
   \begin{pmatrix} 0 & i \cdot \1_{2n} \\ -i \cdot \1_{2n} & 0 \end{pmatrix}
  V_{4n}^*  \\
  &= {\diag}( I^{(2)}, \dots, I^{(2)} )
   = I_{2n}^{(2)} \; .
\end{align*}

Now $a^{\sharp \otimes \tau} = a^*$ implies that
$B(a)^{\widetilde{\sharp} \otimes \sharp \otimes \tau} = -B(a)$ similar to case $(2)$.
Also $W_{4n}^{\widetilde{\sharp} \otimes \sharp} = -W_{4n}^*$ so
$$(W_{4n} B(a) W_{4n}^*)^{\widetilde{\sharp} \otimes \sharp \otimes \tau}
  = W_{4n} B(a)^{\widetilde{\sharp} \otimes \sharp \otimes \tau} W_{4n}^*
  = -W_{4n} B(a) W_{4n}^* \; .$$
Now we have $Q_{4n} x^\tau Q_{4n}^* = (Q_{4n} x Q_{4n}^*)^{\widetilde{\sharp} \otimes \sharp \otimes \tau}$ for all $x$, from it which it follows that
$Q_{4n}^* x^{\widetilde{\sharp} \otimes \sharp \otimes \tau} Q_{4n} = (Q_{4n} x Q_{4n}^*)^{\tau}$.
Thus 
$$(Q_{4n} W_{4n} B(a) W_{4n}^* Q_{4n}^*)^\tau = -Q_{4n} W_{4n} B(a) W_{4n}^* Q_{4n}^* \; . $$
Since $V_{4n}$ is special orthogonal, the same formula holds for 
$$B'(a) = V_{4n} Q_{4n} W_{4n} B(a) W_{4n}^* Q_{4n}^* V_{4n}^* \; .$$
showing that $B'(a)$ represents an element in $KO^u_2(I, \tau)$.

For additivity,
let $a \in M_{2m}(\widetilde{A})$ and $b \in M_{2n}(\widetilde{B})$ be lifts of unitaries $u$ and $v$, satisfying $a^{\sharp \otimes \tau} = a$ and $b^{\sharp \otimes \tau} = b$. Then
\begin{small}
$$
  B\begin{pmatrix} a & 0 \\ 0 & b \end{pmatrix} = \left(\begin{array}{cccc}
2aa^* - \1_{2m} & 0 & \!\!2a \sqrt{\1_{2m} - a^* a} & 0 \\
0 & 2bb^* - \1_{2n} & 0 & \!\!2b \sqrt{\1_{2n} - b^* b}\!\! \\
\!\!2a^* \sqrt{\1_{2m} - a a^*} & 0 & \1_{2m} - 2a^* a & 0 \\
0 & \!\!2b^* \sqrt{\1_{2n} - b b^*} & 0 &  \1_{2n} - 2b^* b \end{array} \right)  
$$  
\end{small}
as before.  Taking advantage of the block matrix structure of $Q_{4n}$ and $W_{4n}$ and $B\begin{pmatrix} a & 0 \\ 0 & b \end{pmatrix}$, we have
$$Q_{4m+4n} W_{4m+4n} B\begin{pmatrix} a & 0 \\ 0 & b \end{pmatrix} W_{4m+4n}^* Q_{4m+4n}^*
			= \left(\begin{array}{cccc}
A_{11} & 0 & A_{12} & 0 \\
0 & B_{11} & 0 & B_{12}  \\
A_{21} & 0 & A_{22} & 0 \\
0 & B_{21} & 0 &  B_{22} \end{array} \right)  $$
where 
\begin{align*}
\begin{pmatrix} A_{11} & A_{12} \\ A_{21} & A_{22} \end{pmatrix}  &= Q_{4m} W_{4m} B(a) W_{4m}^* Q_{4m}^*  \\
\text{and~}  
 \begin{pmatrix} B_{11} & B_{12} \\ B_{21} & B_{22} \end{pmatrix}  &= Q_{4n} W_{4n} B(b) W_{4n}^* Q_{4n}^* \; . 
 \end{align*}
Hence,
\begin{align*}
V_{4m+4n} Q_{4m+4n} &W_{4m+4n} B\begin{pmatrix} a & 0 \\ 0 & b \end{pmatrix} W_{4m+4n}^* Q_{4m+4n}^*  V_{4m+4n}^* \\
			&= \begin{pmatrix} Q_{4m}W_{4m} B(a) W_{4m}^* Q_{4m}^* & 0 \\ 
								0 & Q_{4n} W_{4n} B(b) W_{4n}^* Q_{4n}^* \end{pmatrix} \; ,
\end{align*}
which shows that $B'\begin{pmatrix} a & 0 \\ 0 & b \end{pmatrix} = \begin{pmatrix} B'(a) & 0 \\ 0 & B'(b) \end{pmatrix}$.
\end{proof}

\begin{defn}
Suppose we have an exact sequence in Sequence~(\ref{eq:SESct}). 
Let $\pi$ denote both the quotient map
$\pi: A\rightarrow B$ and its extension to 
$M_{n}(\widetilde{A}) \rightarrow M_{n}(\widetilde{B})$
for every $n$. Furthermore, we assume $I=\ker(\pi)$ and we identify
the unit in $\widetilde{I}$ with that of $\widetilde{A}$. 

\begin{enumerate}
\item Suppose $[u] \in KO_{2}^u(B, \tau)$ where 
$u \in M_{2n}(\widetilde{B})$ is a unitary with $u^{\tau}=-u$, $u^* = u$,
and $\lambda(u)=I^{(2)}_n$. Then define 
\[  
\clubsuit_{2}([u])=\left[  -\exp(\pi i a)   \right]
\in KO_{1}^u(I, \tau)
\]
 where $a$ in $M_{2n}(\widetilde{A})$ is any lift of $u$ with
$-1 \leq a \leq 1$ and $a^{\tau}=-a$. 

\item Suppose $[u] \in KO_{0}^u(B, \tau)$ where 
$u \in M_{2n}(\widetilde{B})$ is a unitary with $u^{\tau}= u^* = u$,
and $\lambda(u)=I^{(0)}_n$. Then define 
\[  
\clubsuit_{0}([u])=\left[ -\exp(\pi i a)  \right]
\in KO_{-1}^u(I, \tau)
\]
 where $a$ in $M_{2n}(\widetilde{A})$ is any lift of $u$ with
$-1 \leq a \leq 1$ and $a^{\tau}=a$. 

\item Suppose $[u] \in KO_{6}^u(B, \tau)$ where 
$u \in M_{2n}(\widetilde{B})$ is a unitary with $u^{\sharp \otimes \tau}=-u$, $u^* = u$,
and $\lambda(u)=I^{(6)}_n$. Then define 
\[  
\clubsuit_{6}([u])=\left[  -\exp(\pi i a)   \right]
\in KO_{5}^u(I, \tau)
\]
 where $a$ in $M_{2n}(\widetilde{A})$ is any lift of $u$ with
$-1 \leq a \leq 1$ and $a^{\sharp \otimes \tau}=-a$. 

\item Suppose $[u] \in KO_{4}^u(B, \tau)$ where 
$u \in M_{4n}(\widetilde{B})$ is a unitary with $u^{\sharp \otimes \tau}= u^* = u$
and $\lambda(u)=I^{(4)}_n$. Then define 
\[  
\clubsuit_{4}([u])=\left[ -\exp(\pi i a)  \right]
\in KO_{3}^u(I, \tau)
\]
 where $a$ in $M_{4n}(\widetilde{A})$ is any lift of $u$ with
$-1 \leq a \leq 1$ and $a^{\sharp \otimes \tau}=a$. 

\end{enumerate}

\end{defn}

\begin{lemma} \label{lem:clubdefined2}
The maps $\clubsuit_{i}$ are well-defined group homomorphisms, for $i$ even.
\end{lemma}
\begin{proof}
This involves proving all the same assertions as in the proof of Lemma~\ref{lem:clubdefined1} above.

(1) First we consider $\clubsuit_2 \colon KO^u_2(B, \tau) \rightarrow KO^u_1(I, \tau)$.
Suppose $u$ is a unitary in $M_{2n}(\widetilde{B})$
with $u^{\tau}=-u$ and $u^{*}=u$ and 
$\lambda(u)= I_n^{(2)}$.
We can lift by standard methods to an element $a \in M_{n}(\widetilde{A})$
with $-1\leq a\leq1$ and are guaranteed $\lambda(a)=I_n^{(2)}$ automatically. 
Replacing with the element $\tfrac{1}{2}(a - a^\tau)$,
we can assume the relation $a^{\tau}=-a$. This shows that an appropriate lift exists.

Now consider any lift $a$ of $u$ satisfying $-1\leq a\leq1$, $a^{\tau}=-a$,
and $\lambda(a)=I_n^{(2)}$. Let $E(a) = - \exp(\pi i a)$, which is of course a unitary when $a$ is self-adjoint;
and satisfies $E(a^\tau) = E(a)^\tau$ and $E(-a) = E(a)^*$.
Since $I^{(2)}$ has eigenvalues $\pm1$ we find $\lambda(E(a))=\1_{2n}$
and we have also the familiar fact that $E(a)$ is a unitary when $a$ is self-adjoint.
As to the real structure, we check
$$E(a)^\tau = E(a^\tau) = E(-a) = E(a)^* \; ,$$
so we have indeed obtained a representative of an element in $KO_{1}^u(I, \tau)$. 

If we have two lifts $a_{0}$ and $a_{1}$ with the required relations,
then the straight line 
\[
a_{t}=(1-t)a_{0}+ta_{1}
\]
satisfies the relations at every point, and so $E(a_t)$
provides the needed homotopy showing $[E(a_1)] = [E(a_2)]$.
To deal with a homotopy from $u_{t}$ in $M_{n}(\widetilde{B})$,
we again use the surjection 
\[
M_{2n}((C[0,1],\widetilde{A})) \rightarrow M_{2n}((C[0,1],\widetilde{B}))
\]
induced by $\pi$ to get $a_{t}$, a homotopy of lifts covering the
$u_{t}$. Then $E(a_t)$ is the needed homotopy.

Next we compare the results of this construction for $u$ and that for
\[
v= {\diag}(u, I^{(2)}) = \left(\begin{array}{ccc}
u & 0 & 0\\
 0 & 0 & i\\
 0 & -i & 0
\end{array}\right).
\]
Take $a$ to be a lift of $u$ with $-1\leq a\leq1$ and $a^{\tau}=-a$,
and take
$b= {\diag}(a, I^{(2)})$
as an appropriate lift of $v$. Then
$E(b) = {\diag}(E(a), E(I^{(2)}) ) = {\diag}(E(a), \1_2)$
showing $[E(a)]=[E(b)]$.

Finally, showing that $\clubsuit_i$ is a group homomorphism is straightforward in each even case, since we have
$E( {\diag}(a,b) ) = {\diag}(E(a), E(b))$ exactly (rather than just up to homotopy as in the odd cases).

(2) Next we show that $\clubsuit_0 \colon KO_0^u(B, \tau) \rightarrow KO_{-1}^u(I, \tau)$ is well defined.
Suppose $u$ in $M_{2n}(\widetilde{B})$ is unitary with
$u^{\tau}=u=u^{*}$ and 
$\lambda_{2n}(u)= I^{(0)}_n$.
Lift to an element $a \in M_{n}(\widetilde{A})$
with $-1\leq a\leq1$ and the condition $\lambda_{2n}(a)=I^{(0)}_n$ is immediate. Make the appropriate replacement to obtain $a^{\tau}=a$. 

Now for any lift $a$ of $u$ with $-1\leq a\leq1$, $a^{\tau}=a$,
and $\lambda(a) = I^{(0)}_n$, 
we see even more easily this time that $\lambda(E(a))=I_{2n}$
and that $E(a)$ is unitary. This time the real structure calculation
is
$$
E(a)^\tau = E(a^\tau) = E(a) \; ,$$
so we have $[E(a)] \in KO_{-1}^u(I, \tau)$ as desired.

Dealing with different lifts and dealing with a homotopy of $u_{t}$, is
accomplished just as in (1). 

Finally, we need to compare the results of this construction for $u$ and for
$v= {\diag}(u, I^{(0)})$. Let $a$ be a lift of $u$ and $b = {\diag}(a, I^{(0)})$ be a lift of $v$.
Then $E(b) = {\diag}( E(a), E(I^{(0)})) = {\diag}(E(a), \1_2)$, showing that  
$\clubsuit_0([u]) = \clubsuit_0([v])$.

(3), (4) The proofs that $\clubsuit_4$ and $\clubsuit_6$ are well defined are the same as the proofs that $\clubsuit_0$ and $\clubsuit_2$ are well defined, using $\sharp \otimes \tau$ instead of $\tau$ everywhere.
\end{proof}

\begin{lemma}
Each $\clubsuit_i$ is natural with respect to morphisms of short exact sequences of real \calg s.
\end{lemma}

\begin{proof}
Suppose we have a commutative diagram 
$$\xymatrix{
0 \ar[r]
& (I_1, \tau) \ar[r] \ar[d]^\iota
& (A_1, \tau) \ar[r]^{\pi_1} \ar[d]^\alpha
& (B_1, \tau) \ar[r] \ar[d]^\beta
& 0  \\
0 \ar[r]
& (I_2, \tau) \ar[r] 
& (A_2, \tau) \ar[r]^{\pi_2} 
& (B_2, \tau) \ar[r]
& 0
}$$
of real
\ctalg s, with exact rows.
We show that $\iota_{*} \circ \clubsuit_{i} = \clubsuit_{i} \circ \beta_{*}$ for all $i$.

Suppose that $[u_1] \in KO_i^u(B_1, \tau)$ is given by a unitary $u_1 \in M_{n}(\widetilde{B_1})$ satisfying the specific symmetry relations and $\lambda$ requirement. 
Let $u_2 = \beta(u_1) \in M_{n}(\widetilde{B_2})$.
Then select an element $a_1 \in M_n(\widetilde{A_1})$ such that $\pi_1(a_1) = u$ and $a_1$ satisfies the requirements for the lift described in the definition of $\clubsuit_i$. 
Then $a_2 = \alpha(a_1)$ satisfies $\pi_2(a_2) = u_2$ and 
$a_2$ satisfies the requirements to be an appropriate lift of $u_2$.
In the case that $i$ is odd, since $B(\alpha(a_1)) = \alpha(B(a_1))$, we have
\begin{align*}
\alpha_* \clubsuit_i([u_1]) 
  &= \alpha_*( [Y^{(i)} B(a_1) Y^{(i)*}  ])  \\
  &=  [Y^{(i)} \alpha(B(a_1) )  Y^{(i)*}  ]  \\
  &=   [Y^{(i)} B(a_2 )  Y^{(i)*}  ]  \\
  &= \clubsuit_i([u_2]) \\
  &= \clubsuit_i \alpha_*([u_1]) \; .
\end{align*}
A similar calculation using $E(a_1)$ instead of $B(a_1)$ addresses the even cases.
\end{proof}

For reference, we include a parallel definition of the index maps in the complex case.

\begin{defn}
Suppose we have a exact sequence of real \calg s.
Let $\pi$ denote both the quotient map
$\pi: A\rightarrow B$ and its extension to 
$M_{n}(\widetilde{A}) \rightarrow M_{n}(\widetilde{B})$
for every $n$. Furthermore, we assume $I=\ker(\pi)$ and we identify
the unit in $\widetilde{I}$ with that of $\widetilde{A}$. 

\begin{enumerate}
\item Suppose $[u] \in KU_{1}^u(B)$ where $u \in M_{n}(\widetilde{B})$ 
is a unitary with $\lambda(u)=I^{(1)}_n$. Then define 
\[
\clubsuit_{1}([u])
=\left[\left(\begin{array}{cc}
2aa^{*}- \1_n & 2a\sqrt{ \1_n -a^{*}a} \\
2a^{*}\sqrt{\1_n-aa^{*}} & \1_n-2a^{*}a
\end{array}\right)\right] \in KU_0^u(I) \]
where $a$ in $M_{n}(\widetilde{A})$ is any lift of $u$ with
$\|a\|\leq1$.

\item Suppose $[u] \in KU_{0}^u(B)$ where $u \in M_{n}(\widetilde{B})$ 
is a unitary with $u = u^*$ and $\lambda(u)=I^{(0)}_n$. Then define 
\[
\clubsuit_0([u]) 
= \left[ -e^{\pi ia} \right]
 \in KU_1^u(I) \]
where $a$ in $M_{n}(\widetilde{A})$ is any lift of $u$ with
$-1 \leq a \leq 1$.
\end{enumerate}
\end{defn}

These homomorphisms are well-defined and natural, as can be shown by proofs similar to those we have just performed in the real case. In a sense made precise by the following lemma, these definitions of the index map are equivalent to the standard definitions found in the literature.

\begin{lemma} \label{lemma:partialclubsuitcommute}
Let $0 \rightarrow I \rightarrow A \rightarrow B \rightarrow 0$ be a short exact sequence of real \calg s.
\begin{enumerate}
\item For all $0 \leq i < 8$, the following diagram commutes
$$ \xymatrix{
KO_i^u(B) \ar[r]^{\clubsuit_i} \ar[d]^{c_i}
& KO^u_{i-1}(I) \ar[d]^{c_{i-1}}  \\
KU_i^u(B) \ar[r]^{\clubsuit_i} 
& KU^u_{i-1}(I) 
}$$
\item For all $0 \leq i < 2$, the following diagrams commute up to sign
$$ \xymatrix{
KU_1^u(B) \ar[r]^{\clubsuit_1} \ar[rd]_{\partial_1}
& KU^u_{0}(I) \ar[d]^{\Theta}  
& KU_0^u(B) \ar[r]^{\clubsuit_0} \ar[d]_{\Theta}
& KU_1^u(I) \\
& KU_{0}(I) 
& KU_0(B) \ar[ur]_{\partial_0}
& 
}$$
where $\Theta$ is the isomorphism of Theorem~\ref{K0iso}
and $\partial_i$ is the boundary map of the literature in the complex setting.
\end{enumerate}
\end{lemma}

\begin{proof}
Statement (1) is immediate from the definitions of $\clubsuit_i$ for each $i$, noting that $KU^u_i(-)$ classes are unchanged by conjugation by any unitary in $M_n(\C)$.

For statement (2), first let $[u] \in KU_1^u(B)$ where $u \in M_n(\widetilde{B})$ is a unitary with $\lambda(u) = \1_n$. Find a lift $a \in M_n(\widetilde{A})$ of norm at most 1.
Then 
\begin{align*}
(\Theta \circ \clubsuit_1)([u]) 
  &= \Theta([B(a)])   \\
  &= \left[ \tfrac{1}{2}(B(a) + \1_{2n}) \right] - \left[ \1_n \right]  \\
  &= \left[ \begin{pmatrix} aa^* + \1_n & a \sqrt{\1_n - a^* a} \\
			    a^* \sqrt{\1_n - a a^*} & \1_n - a^* a
		\end{pmatrix} \right] 
			- \left[ \1_n \right]   \; . 
\end{align*}
On the other hand, using the description of $\partial_1$ from Proposition~9.2.2 of \cite{rordambookblue}, it is defined in terms of the same lift $a$ and works out to
$$\partial_1([u]) = [\1_n] - \left[ \begin{pmatrix} aa^* + \1_n & a \sqrt{\1_n - a^* a} \\
			    a^* \sqrt{\1_n - a a^*} & \1_n - a^* a
		\end{pmatrix} \right] \; .$$
Hence the diagram commutes after adjusting by a factor of $-1$.

Now let $[u] \in KU_0^u(B)$ where $u \in M_{2n}(\widetilde{B})$ is a unitary with 
$u = u^*$ and $\lambda(u) = I^{(0)}_n$. 
Then $\Theta([u]) = \left[ \tfrac{1}{2}(u + \1_{2n}) \right] - \left[ \1_n \right]$. 
Let $a \in M_{2n}(\widetilde{A})$ be a lift of $u$ satisfying $-1 \leq a \leq 1$.
Then $a' = \tfrac{1}{2}(a + \1_{2n})$ is a self-adjoint lift of the projection 
$p = \tfrac{1}{2}(u + \1_{2n})$ so
using the formulas for $\partial_0$ from Proposition~12.2.2 of \cite{rordambookblue} or~9.3.2 of \cite{blackadarbook},
\begin{align*}
(\partial_0 \circ \Theta)([u])
  &= \partial_0([p] - [\1_n])  \\
  &= [\exp(2\pi ia')]    \\
  &= [\exp(\pi i (a + \1_{2n}))]     \\
  &= [-\exp(\pi ia)]       \\
  &= \clubsuit_0([u]) \; .
\end{align*}
\end{proof}

Our goal for the rest of this section is to prove that in the real case the homomorphisms
$\clubsuit_i$ and $\partial_i$ are the same, up to the same sign adjustment necessary in Lemma~\ref{lemma:partialclubsuitcommute} above. Since any convention of the index map can be adjusted by a sign, we will henceforth assume that the diagrams in Lemma~\ref{lemma:partialclubsuitcommute} commute exactly and prove that $\clubsuit_i = \partial_i$ exactly for all $i$.

For each $i \in \{0, 1, \dots, 7\}$, the algebra $A_i$ is generated by a finite number of elements. By Theorem~5.1.5 of \cite{sorensenthesis}, which is the real \calg~ counterpart of Theorem~2.10 of \cite{loring10}, 
there exist universal real \calg s on a given set of generators subject to the relation that these generators are bounded in norm by $1$. 
Thus, for each $i$, there is such a real \calg~ $P_i$ and a surjective homomorphism 
$$\rho_i \colon P_i \rightarrow A_i \; .$$
Furthermore, since the relation is liftable, the algebras $P_i$ are projective.
From this we obtain the short exact sequences
$$0 \rightarrow J_i \rightarrow P_i \xrightarrow{\rho_i} A_i \rightarrow 0 \; $$
which are universal for the boundary map in a sense that we will take advantage of in the proof of Theorem~\ref{thm:partial=clubsuit} below.

\begin{lemma}
For all $i$ we have 
$\partial_i = \clubsuit_i \colon KO^u_i(A_i) \rightarrow KO^u_{i-1}(J_i)$.
\end{lemma}

\begin{proof}
Since $P_i$ is projective, we have $K\crt(P_i) \cong 0$, so
$\partial_i \colon KO^u_i(A_i) \rightarrow K^u_{i-1}(J_i)$ is an isomorphism of degree $-1$. By Theorems~\ref{thm:KAieven} and~\ref{thm:KAiodd}, both of these groups are isomorphic to $\Z$.
In fact, since $K\crt(A_i) \cong \Sigma^{-i} K\crt(\R)$, the structure of this \ct-module also implies that all four groups in the diagram below are isomorphic to $\Z$, and that the vertical maps $c_i$ and $c_{i-1}$ are isomorphisms.
$$\xymatrix{
KO^u_i(A_i) \ar[d]^{c_i} \ar[r]^{\partial_i}
&  KO^u_{i-1}(J_i) \ar[d]^{c_{i-1}}  \\
KU^u_i(A_i) \ar[r]^{\partial_i}
& KU^u_{i-1}(J_{i})
}$$
This diagram commutes by the naturality of the index map and the naturality of the complexification map. By Lemma~\ref{lemma:partialclubsuitcommute}, the diagram also commutes if we replace $\partial_i$ in the upper horizontal arrow with $\clubsuit_i$. 
It follows that these two homomorphisms must coincide.
\end{proof}

\begin{thm} \label{thm:partial=clubsuit}
For all $i$, $\partial_i = \clubsuit_i$.
\end{thm}

\begin{proof}
Let 
$$0 \rightarrow I \rightarrow A \rightarrow B \rightarrow 0$$
be a short exact sequence of \ctalg s. Let $\xi \in KO^u_i(B)$. Then by Theorem~\ref{thm:classify} for some integer $n$ there exists a homomorphism $\phi \colon A_i \rightarrow M_n(B)$,
such that $\phi_*([x_i]) = \xi$. Since $P_i$ is projective, there exists a homomorphism $\psi$ and we obtain a homomorphism of short exact sequences,
$$\xymatrix{
0 \ar[r]
& J_i \ar[r] \ar[d]
& P_i \ar[r]^{\rho_i} \ar[d]^\psi
& A_i \ar[r] \ar[d]^\phi
& 0 \\
0 \ar[r]
& M_n(I) \ar[r] 
& M_n(A) \ar[r]
& M_n(B) \ar[r]
& 0  \\
}$$
which then induces a commutative diagram on $K$-theory
$$\xymatrix{
KO^u_i(A_i) \ar[d]^{\phi_*} \ar[rr]^{\text{$\partial_i$ or $\clubsuit_i$}}
&&  KO^u_{i-1}(J_i) \ar[d]  \\
KO^u_i(M_n(B)) \ar[rr]^{\text{$\partial_i$ or $\clubsuit_i$}}
&& KO^u_{i-1}(M_n(I_{i}))
}$$
This diagram commutes if we take the horizontal homomorphisms to be both $\partial_i$ or both $\clubsuit_i$. Since these two choices coincide for the upper arrow, they must coincide for the lower arrow on $\xi$.

Finally, consider the commutative diagrams below arising from the morphism of the original short exact sequence into the same one tensored by $M_n$.
$$\xymatrix{
KO^u_i(B) \ar[d] \ar[rr]^{\text{$\partial_i$ or $\clubsuit_i$}}
&&  KO^u_{i-1}(I_i) \ar[d]  \\
KO^u_i(M_n(B)) \ar[rr]^{\text{$\partial_i$ or $\clubsuit_i$}}
&& KO^u_{i-1}(M_n(I_{i}))
}$$
Since the vertical arrows are isomorphisms, the result of the previous paragraph shows that 
$\partial_i = \clubsuit_i \colon KO_i^u(B) \rightarrow KO_{i-1}^u(I)$.
\end{proof}

\section{Boundary map examples: spheres and Calkin algebras} \label{sec:examples2}

\begin{example}
Let $\sigma$ be the involution on $C(S^1)$ given by $f^\sigma(z) = f(-z)$. The corresponding real \calg~ is
$\{ f \in C(S^1, \C) \mid f(-z) = \overline{f(z)} \} $
which is isomorphic to the real \calg~ $T$ associated with self-conjugate $K$-theory discussed in \cite{boersema02}. The groups $KO_i(T)$ are calculated in Corollary~1.6 of \cite{boersema02}, but we will present a self-contained calculation of $KO_*(C(S^1), \sigma)$ and also find unitary elements representing generators of the non-trivial $KO$-classes.

Let $\sigma$ also denote the involution on $\C \oplus \C$ given by $(z,w)^\sigma = (w, z)$. Then there is a short exact sequence
$$0 \rightarrow (C_0(S^1 \setminus \{\pm 1\}), \sigma)
  \rightarrow (C(S^1), \sigma)
  \xrightarrow{\pi} (\C \oplus \C, \sigma) \rightarrow 0 \; $$
where $\pi = ({\ev}_1, {\ev}_{-1})$ and we will describe the boundary maps
$$\partial_i \colon KO_i^u(\C \oplus \C, \sigma) \rightarrow
      KO_{i-1}^u(C_0(S^1 \setminus \{\pm 1 \}), \sigma) \; .$$

By Lemma~\ref{lem:KOKU}, we have 
$$KO_i^u(\C \oplus \C, \sigma) \cong KU_i(\C) 
    \cong \begin{cases} \Z & \text{$i$ even} \\ 0 & \text{$i$ odd.} \end{cases} $$
Similarly, since there is an isomorphism 
$(C_0(S^1 \setminus \{\pm 1\}), \sigma) \cong ( C_0(0,1) \oplus C_0(0,1), \sigma)$, we have
$$KO_i^u (C_0(S^1 \setminus \{\pm 1\}), \sigma) \cong 
  \begin{cases} 0 & \text{$i$ even} \\ \Z & \text{$i$ odd.} \end{cases} $$
Thus for $i$ even we have $\partial_i \colon \Z \rightarrow \Z$ and we claim that
$$\partial_i = \begin{cases} 0 & i = 0,4 \\  2 & i = 2,6 \end{cases}$$
(up to sign determined by the choices of isomorphism).

From Example~\ref{ex:Rgenerators} and Lemma~\ref{lem:KOKU}, the generator of 
$KO^u_0(\C \oplus \C, \sigma)$ is $[w]$ where $w = (1_2, 1_2)$. This lifts to
$a = 1_2 \in M_2(C(S^1, \C))$, which still satisfies $a^* = a$ and $a^\sigma = a^*$. Then $\partial_0([w]) = [-\exp(\pi i 1_2)] = [1_2]$, so
$\partial_0 =0$.

The generator of $KO^u_2(\C \oplus \C, \sigma)$ is $[w]$ where $w = (1_2, -1_2)$. One lift of $w$ is $a \in M_2(C(S^1, \C))$ defined by
$$a(e^{2\pi i t}) = \begin{pmatrix} f(t) & 0 \\ 0 & f(t) \end{pmatrix} \text{~where~} 
	f(t) = \begin{cases} 1-4t & 0 \leq t \leq 1/2 \\-3 + 4t & 1/2 \leq t \leq 1. \end{cases}$$
Check that $a^* = a$ and $a^\sigma = -a$.
Then $\partial_2([w]) = [E(a)] = [-\exp(\pi i a)]$. We have
$$ E(a)(e^{2 \pi i t}) = \begin{pmatrix} -\exp(\pi i f(t)) & 0 \\ 0 & -\exp(\pi i f(t)) \end{pmatrix}$$
so
$$E(a)
=  \begin{pmatrix} v & 0 \\ 0 & v \end{pmatrix} \text{~where~} 
		v(z) = \begin{cases} \overline{z}^2 & \Im(z) \geq 0 \\ z^2 & \Im(z) < 0. \end{cases}$$
Using the natural isomorphism 
$$(C_0(S^1 \setminus \{\pm 1\}), \sigma) \cong 
  (C_0(0,1) \oplus C_0(0,1), \sigma)$$
and combining Example~\ref{ex:Rgenerators} and Lemma~\ref{lem:KOKU}, we see that
$[v]$ is a generator of $KO_{1}^u (C_0(S^1 \setminus \{\pm 1\}), \sigma)$. Thus $[E(a)]$ is two times a generator.

For $i = 4$, the generator of $KO^u_4(\C \oplus \C, \sigma)$ is $[w]$ where 
$$w = \left( {\diag}(1, 1, 1, -1),{\diag}(1, 1, -1, 1) \right) \in M_4(\C) \oplus M_4(\C) \; .$$
To verify this, note that $w^{\sharp \otimes \sigma} = w$ and that $[{\diag}(1,1,1,-1)]$ is a generator of $KU_0(\C, {\id})$. Then a lift of $w$ is $a$ where
$$a(x,y) = {\diag}\left(1_2 , \sm{x}{y}{y}{-x} \right) \in M_4(C(S^1)) \; .$$ 
which satisfies $a^{\sharp \otimes \sigma} = a$. Since $a$ in fact is a self-adjoint unitary,
then $[E(a)]$ is the trivial class.

Finally, for $i = 6$, the generator of $KO^u_6(\C \oplus \C, \sigma)$ is $[w]$ where
$$w = \left( 1_2, -1_2 \right) \; .$$
This satisfies $w^{\sharp \otimes \sigma} = -w$ and a lift $a$ that satisfies 
$a^{\sharp \otimes \sigma} = -a$ is
$$a(e^{2 \pi i t}) = {\diag}(f(t), f(t) ) \; .$$
Then $E(a) = {\diag}(v, v)$. Again this implies that $[E(a)]$ is two times a generator.

Therefore, we have
$$\begin{array}{|c||c|c|c|c|c|c|c|c|}
\hline
i & 0 & 1 & 2 & 3 & 4 & 5 & 6 & 7  \\ \hline
KO_i(C(S^1), \sigma) & \Z & \Z_2 & 0 & \Z
	    & \Z & \Z_2 & 0 & \Z\\
\hline
\end{array}$$

We will show furthermore that the non-trivial $KO$-groups have generators represented by the following unitaries.
\begin{itemize}
\item $KO_{-1}^u(C(S^1), \sigma) \cong \Z$ generated by $[w_{-1}]$ where $w_{-1}(z) = z^2$.
\item $KO_0^u(C(S^1), \sigma) \cong \Z$ generated by $[w_0]$ where $w_0 = 1_2$.
\item $KO_1^u(C(S^1), \sigma) \cong \Z_2$ generated by $[w_1]$ where 
      $w_1(z) = -1$
\item $KO_3^u(C(S^1), \sigma) \cong \Z$ generated by $[w_3]$ where
       $w_3(z) = {\diag}(z, -z).$
\item $KO_4^u(C(S^1), \sigma) \cong \Z$ generated by $[w_4]$ where 
$w_4(x,y) = {\diag}\left(1_2 , \sm{x}{y}{y}{-x} \right) $.
\item $KO_5^u(C(S^1), \sigma) \cong \Z_2$ generated by $[w_5]$ where 
      $w_5(z) = {\diag}(z, \overline{z}).$
\end{itemize} 

For $i=0,4$, we know that $\pi_*$ is surjective, so it is just a matter of checking that the induced class $[\pi(w_i)]$ is a known generator of $KO_i^u(\C \oplus \C, \sigma)$. For $i$ odd, in each case we start with a known generator $[x_i]$ of 
$KO_i^u (C_0(S^1 \setminus \{\pm 1\}), \sigma)$ and find the induced element in $KO_i^u (C_0(S^1), \sigma)$.

For example, for $i = 1$, consider the unitary 
	$$x_1 = \begin{cases} z^2 & \Im(z) \geq 0 \\ \overline{z}^2 & \Im(z) < 0 \end{cases}$$
which represents the generating class of $KO_1^u (C_0(S^1 \setminus \{\pm 1\}), \sigma)$.  
Note that ${\ev}_{-1}(x_1) = {\ev}_1(x_1) = 1$.  However, as a class of
$KO_1^u(C(S^1), \sigma)) \cong \Z$ we have $[x_1] = [-1]$ since there is a homotopy from $x_1$ to $w_1$ unitaries $w_t$ satisfying $(w_t)^{\sharp \otimes \sigma} = w_t$.  Indeed, note that $x_1$ restricted to the right half of the circle is a unitary-valued path from $-1$ to $-1$ which is homotopic to a constant through such paths. Also note that any such path on the right half of the circle can be extended to a function on the whole circle satisfying the proper symmetry.  (For unitaries in $(C(S^1), \sigma)$ there is no requirement that ${\ev}_1 = 1$.).

For $i = 3$, we start with the unitary
$$x_3 = \begin{cases} {\diag}(z^2, 1) & \Im(z) \geq 0 \\ 
		  {\diag}(1, z^2) & \Im(z) < 0. \end{cases}$$
representing a class in $KO_1^u (C_0(S^1 \setminus \{\pm 1\}), \sigma)$.
In $KO_3^u(C(S^1, \sigma)) \cong \Z$ we have $[x_3] = [w_3]$ where $w_3$ is as above.
Indeed, if $f(z)$ is any unitary-valued function on the circle such that $f(1) = 1$, then
$\sm{f(z)}{0}{0}{f(-z)}$ 
is in the appropriate symmetry class for $KO_3^u(C(S^1), \sigma)$.  Now, the two choices
$f(z) = z$ and $f(z) = \begin{cases} z^2 & \Im(z) \geq 0 \\ 1 & \Im(z) < 0 \end{cases}$ yield the two unitaries $x_3$ and $w_3$ under consideration.  Since these two choices of $f$ are themselves homotopic, they yield a homotopy from $x_3$ to $w_3$.

In a similar way, we obtain the given generator $[w_5] \in KO_5^u(C(S^1, \sigma)) \cong \Z_2$.
 
\end{example}

\begin{example}
We now study the boundary maps for the short exact sequence
\begin{equation} \label{seq:ex2}
0 \xrightarrow{\iota} (C_0(S^1 \setminus \{\pm 1\}), \zeta)
  \rightarrow (C(S^1), \zeta)
  \xrightarrow{\pi} (\C \oplus \C, {\id}) \rightarrow 0 \; \end{equation}
where $\pi = ({\ev}_1, {\ev}_{-1})$ and $f^\zeta(z) = f(\overline{z})$.
The associated real \calg~ to $(C(S^1), \zeta)$
is 
$$A = \{f \colon S^1 \rightarrow \C \mid f(\overline{z}) = \overline{f(z)} \} \; .$$
There is a different split exact sequence involving $A$, namely
$$0 \rightarrow 0 \rightarrow S^{-1} \R \rightarrow A \rightarrow \R \rightarrow 0$$
which easily implies $KO_*(A) \cong KO_*(\R) \oplus \Sigma^{-1} KO_*(\R)$, 
with individual groups shown in the table below.
$$\begin{array}{|c||c|c|c|c|c|c|c|c|}
\hline
i & 0 & 1 & 2 & 3 & 4 & 5 & 6 & 7  \\ \hline
KO_i(C(S^1), \zeta) & \Z & \Z \oplus \Z_2 & \Z_2 \oplus \Z_2 &
	    \Z_2 & \Z & \Z & 0 & 0 \\
\hline
\end{array}$$
However, we will independently calculate the boundary maps associated with the 
Sequence~(\ref{seq:ex2})
using our methods, arrive at the same abstract groups, and identify explicit unitary generators. 

To compute
$$\partial_i \colon KO_i^u(\C \oplus \C, {\id}) \rightarrow 
  KO_{i-1}^u(C_0(S^1 \setminus \{ \pm 1 \}), \zeta)$$
we first identify the relevant groups as
$$\begin{array}{|c||c|c|c|c|c|c|c|c|}
\hline
i & 0 & 1 & 2 & 3 & 4 & 5 & 6 & 7  \\ \hline
KO^u_i(\C \oplus \C, {\id}) & \Z^2 & \Z_2^2 & \Z_2^2 & 0 & \Z^2 & 0 & 0 & 0 \\ \hline
KO^u_{i-1}(C_0(S^1 \setminus \{ \pm 1 \}), \zeta)
  & 0 & \Z & 0 & \Z & 0 & \Z & 0 & \Z \\ \hline
\end{array}$$
so we know right away that $\partial_i = 0$ unless $i = 0,4$. For $i = 0,4$ 
we have $\partial_i \colon \Z \oplus \Z \rightarrow \Z$. We will show that
$\partial_0(r,s) = r-s$ and $\partial_2(r,s) = 2r  - 2s$ (with appropriate identifications).

Suppose $j=0$. Recall that $I^{(0)} = {\diag}(1, -1)$. Then generators of 
 $KO^u_{0}\left(\C \oplus\C,{\id} \right)$
are $[w_1]$ and $[w_2]$ where
\[
w_{1}= \left( 1_2, I^{(0)} \right),
w_{2} = \left( I^{(0)}, 1_2  \right)
\in M_2(\C) \oplus M_2(\C) \; .
\] 
We find self-adjoint lifts $a_i$ of $w_i$ to be
\[
a_{1}(e^{2\pi it})= \begin{pmatrix} 1 & 0 \\ 0 & f(t) \end{pmatrix}
\qquad \text{and} \qquad
a_{2}(e^{2\pi it})= \begin{pmatrix} 1 & 0 \\ 0 & -f(t) \end{pmatrix}
\]
where $f$ is as in Example~1. Check that $a_i^\zeta = a_i$.
Then $\partial_{0}([w_{i}]) = [u_i]$
where $u_{i}=-\exp(\pi ia_{i})$ so
\[
u_{1}= \begin{pmatrix} 1 & 0 \\ 0 & v \end{pmatrix}
\qquad \text{and} \qquad
u_2 = \begin{pmatrix} 1 & 0 \\ 0 & v^*  \end{pmatrix}
\]
and $v$ is as in Example~1 (check that $v^\zeta = v$).
We have an isomorphism 
$(C_0(S^1 \setminus \{ \pm 1\}), \zeta) \cong (C_0(0,1) \oplus C_0(0,1), \sigma)$.
Since $[v]$ represents a generator of 
$KO_{-1}^u(C_0(S^1 \setminus \{ \pm 1 \}), \zeta) \cong \Z$ and $[u_1] = -[u_2]$, this proves our claim for $\partial_0$.

Now suppose $j=4$. The generators of $KO^u_{4}\left( \C \oplus \C,{\id} \right)$
are represented by the unitaries 
$$ w_1 = (1_4, I^{(4)} )
\qquad \text{and} \qquad
w_2 = (I^{(4)}, 1_4 ) \; $$
that satisfy $w_i^{\sharp \otimes {\id}} = w_i$.
Lifts of $w_{i}$ that satisfy $a_i^{\sharp \otimes \zeta} = a_i$ are
$$
a_1(e^{2 \pi i t}) = {\diag}(1, 1, f(t), f(t)) 
\qquad \text{and} \qquad 
a_2(e^{2 \pi i t}) = {\diag}(1, 1, -f(t), -f(t)) \; .
$$
Therefore $\partial_{1}([w_{i}]) = [E(a_i)] = [u_i]$
where
$$
u_{1} = {\diag}(1, 1, v, v) 
\qquad \text{and} \qquad u_2 = {\diag}(1, 1, v^*, v^*) \; .
$$
Through the isomorphisms  
$$\Z \cong KU_1^u(C_0(0,1), {\id} ) \cong KO_3^u(C_0(0,1) \oplus C_0(0,1), \sigma)
  \cong KO_3^u(S^1 \setminus \{\pm 1\}, \zeta)$$
we conclude that the $KO_3^u(S^1 \setminus \{\pm 1\}, \zeta) \cong \Z$ class of a unitary
is determined by the winding number of that unitary on the top half of the circle.  Thus
$[{\diag}(v,v)]$ is twice a generator. 
(A generator would be given for example by a unitary such as $w_3$ below).
Since $[u_1] = -[u_2]$, 
This proves the claim for $i =4$.

Now that the boundary maps are understood, the only group that is not fully determined up to isomorphism by the exact sequences is
$KO_1^u(C(S^1), \zeta)$ which is an extension of $\Z_2 \oplus \Z_2$ by $\Z$. We will show that
$KO_1^u(C(S^1), \zeta) \cong \Z \oplus \Z_2$. Note that 
the generator of $KO_1^u(C_0(S^1 \setminus \{ \pm 1 \}), \zeta)$ is given by
$[v']$ where $v'(z) = z^2$ (check that $(v')^\zeta = (v')^*$). The image $\iota_*[v']$ is divisible by $2$ in $KO_1^u(C(S^1), \zeta)$, since $[\iota(v')] = 2[v'']$ where $v''(z) = z$ (again, check that $(v'')^\zeta = (v'')^*$). This shows that the extension problem for
$$0 \rightarrow \Z \xrightarrow{\iota_*} KO_1^u(C(S^1), \zeta) 
  \xrightarrow{\pi_*} \Z_2^2 \rightarrow 0$$
is solved by $ KO_1^u(C(S^1), \zeta) \cong \Z \oplus \Z^2$.

Now that we have determined up to isomorphism that groups $KO_i^u(C(S^1), \zeta)$, we identify the
generators and write down specific isomorphisms.
\begin{itemize}
\item $KO_0^u(C(S^1), \zeta) \cong \Z$ generated by $[w_0]$ where $w_0 = 1_2$.
The isomorphism can be realized as $[w] \mapsto \tfrac{1}{2}\text{trace}(w(1))$.
\item $KO_1^u(C(S^1), \zeta) \cong \Z \oplus \Z_2$ generated by $[w_1]$ and $[w'_1]$ where
$w_1(z) = z$ and $w_1'(z) = -1$. Clearly, $[w_1']$ generates the element of order 2.
The map
\[
KO^u_{1}(C(S^1), \zeta)\rightarrow\mathbb{Z}\oplus\mathbb{Z}_{2}
\]
is described by
\[
[w]
=
\left(
\mathrm{winding}(t\mapsto\det\left(w(e^{2\pi it})\right)),
\tfrac{1}{2}-\tfrac{1}{2}\det(w(1))
\right)
\]
where $t$ is in $[0,1]$.
\item $KO_2^u(C(S^1), \zeta) \cong \Z_2 \oplus \Z_2$ generated by $[w_2]$ and $[w'_2]$ where
  $$w_2(x+iy) = \begin{pmatrix} y & ix \\ -ix & y \end{pmatrix}	\qquad \text{and} \qquad
  w'_2(x+iy) = \begin{pmatrix} y & -ix \\ ix & y \end{pmatrix} \; .	$$
Moreover, a nice formula is that the isomorphism 
$KO_2^u(C(S^1), \zeta) \rightarrow \Z_2 \oplus \Z_2$
is given by
\[
[w]\mapsto\left(\mathrm{sign}(\mathrm{Pf}(w(1)),\mathrm{sign}(\mathrm{Pf}(w(-1))\right).
\]

\item $KO_3^u(C(S^1), \zeta) \cong \Z_2$ generated by $[w_3]$ where
  $$w_3(z) = \begin{cases} 
      \begin{pmatrix} z^2 & 0 \\ 0 & 1 \end{pmatrix} & \Im z \geq 0 \\
        \begin{pmatrix} 1 & 0 \\ 0 & \overline{z}^2 \end{pmatrix} & \Im z < 0. \\
	      \end{cases}$$
The class of any unitary $w$ can be determined by looking at the winding number of $w$ restricted to the top half of the circle (modulo 2).
\item $KO_4^u(C(S^1), \zeta) \cong \Z$ generated by $[w_4]$ where $w_4 = 1_4$.
The isomorphism can be realized as $[w] \mapsto \tfrac{1}{2}\text{trace}(w(1))$.

\item $KO_5^u(C(S^1), \zeta) \cong \Z$ generated by $[w_5]$ where 
  $$w_5(z) = \begin{cases} 
      \begin{pmatrix} z^2 & 0 \\ 0 & 1 \end{pmatrix} & \Im z \geq 0 \\
        \begin{pmatrix} 1 & 0 \\ 0 & z^2 \end{pmatrix} & \Im z < 0. \\
	      \end{cases}$$
The class of any unitary $w$ can be determined by looking at the winding number of $w$ restricted to the top half of the circle.
\end{itemize} 

\begin{proof}[Sketch of Proof.]
For $i = 0, 1, 2$ it suffices to show that the shown generators map via $\pi_*$ to corresponding generators of $KO_i^u(\C \oplus \C, {\id})$.  
For $i = 4$, it suffices to show that the shown generator $[1_4]$ maps via $\pi_*$ to the generator of the kernel of $\partial_4$. This generator corresponds to $(2,2)$ in our conventional isomorphism
$KO_3^u(\C \oplus \C, {\id}) \cong \Z \oplus \Z$.

For $i = 3, 5$ it suffices to show that the shown generators are the image of generators of 
$KO_i^u(C_0(S^1 \setminus \{ \pm 1\}), \zeta)$.

For example, to identify the generators of $KO_2(C(S^1), \zeta) \cong \Z_2 \oplus \Z_2$, first note that $(w_2)^{\zeta} = -w_2$ and similarly for $w'_2$. We know that $\pi_*$ is an isomorphism on $KO_2^u(-)$ so it suffices to note that $[\pi(w_2)] = [I^{(2)}, -I^{(2)}]$ and $[\pi(w'_2)] = [-I^{(2)}, I^{(2)}]$ which are the generators of $KO_2^u(\C \oplus \C, {\id}) \cong \Z_2 \oplus \Z_2$.
\end{proof}

\end{example}

\begin{example}
We will consider the exact sequence
$$0 \rightarrow C_0(U, \R) \rightarrow C(D, \R) \rightarrow C(S^1, \R) \rightarrow 0$$
where $D$ is the unit disk and $U = D \setminus S^1$ is the interior of $D$. In terms of \ctalg s, we have
$$0 \rightarrow (C_0(U), {\id}) \rightarrow (C(D), {\id}) \rightarrow (C(S^1), {\id}) \rightarrow 0 \; .$$

We wish to disregard the summands of $KO^u_*(C(D), {\id})$ and $KO^u_*(C(S^1) ,{\id})$ associated with the unit, so we consider the reduced $K$-theory 
$$0 \rightarrow C_0(U, \R) \rightarrow C_0(D \setminus \{1\}, \R) 
  \rightarrow C_0(S^1 \setminus \{1\}, \R) \rightarrow 0$$
Then the boundary map is an isomorphism 
$$\partial_i \colon KO_i^u(C_0(S^1 \setminus \{1\}), {\id})  \rightarrow KO_{i-1}^u(C_0(U), {\id}) $$ 
with the groups as shown.
$$\begin{array}{|c||c|c|c|c|c|c|c|c|}
\hline
i & -2 & -1 & 0 & 1 & 2 & 3 & 4 & 5   \\ \hline
KO_*^u(C(S^1 \setminus \{1\}), {\id})  & 0 & \Z & \Z_2 & \Z_2 & 0 & \Z & 0 & 0 \\ \hline
KO^u_i(C_0(U), {\id}) & \Z & \Z_2 & \Z_2 & 0 & \Z & 0 & 0 & 0\\
\hline
\end{array}$$

We will focus on the case when $i = -1$. The free abelian generator of 
$KO_{-1}^u(C_0(S^1 \setminus \{1\}), {\id})$ 
is $[w_{-1}]$ where
$w_{-1}(x,y) = x + iy$, clearly satisfying 
$w_{-1}^{{\id}} = w_{-1}$. 
Then an appropriate lift of $z$ in $C(D)$ is
$a(x,y) = x+iy$, so
$\partial_{-1}[w_{-1}] = [W_2 B(a) W_2^*] \in KO_{-2}(C_0(U), {\id})$ where
\begin{align*} B(a) 
  &= \begin{pmatrix} 2a a^* - 1 & 2a \sqrt{1 - a^* a} \\ 2a^* \sqrt{1 - a a^*} & 1 - 2a^* a \end{pmatrix}  \\
  &= \begin{pmatrix} 2(x^2 + y^2) - 1 & 2(x+iy) \sqrt{1 - (x^2 + y^2)} \\
		2(x-iy) \sqrt{1 - (x^2 + y^2) } & 1 - 2(x^2 + y^2)
    \end{pmatrix} \; .
\end{align*}
Notice that on the boundary of $D$, $B(a) = {\diag}(1, -1)$ so $\lambda(W_2 B(a) W_2^*) = I^{(2)}$ as expected.

There is a continuous map from $S^2  \setminus \{(0,0,1)\}$ to $U$ given by 
$$(x,y,z) \mapsto \sqrt{\frac{z+1}{2(x^2 + y^2)}} (x,y)$$
which gives an isomorphism $C_0(S^2 \setminus \{(0,0,1) \})$ to $C_0(U)$. Under this transformation, we have
$$B(a) = \begin{pmatrix}
                z & x + iy \\ x-iy & -z 
                \end{pmatrix} \in M_2(C(S^2)) \; .$$
Then 
$$W_2 B(a) W_2^*  = \begin{pmatrix} -y & x + iz \\ x - iz & y \end{pmatrix} \; $$
which is equivalent (via a rigid automorphism of the sphere) to the generator of $KO_{-2}^u(C_0(S^2 \setminus \{ 1 \} ) ) \cong \Z$ that was identified in Example~\ref{ex:spheregenerators}.
\end{example}

Our final example uses our machinery to deal with index maps applied to
Fredholm operators with various symmetries.  For a more direct approach, and one
that allows for a $\mathbb{Z}_2$-graded in addition to a real structure,
see \cite{grossmann2015index}.

\begin{example}
Let $\B\pr$ be the real \calg~ of bounded operators on a separable infinite dimensional real Hilbert space and let $\K\pr$ be the ideal of compact operators. Then we have a short exact sequence
$$0 \rightarrow \K\pr \rightarrow \B\pr \xrightarrow{\pi} \Q\pr \rightarrow 0 $$
where $Q\pr$ is the real Calkin algebra. For any real \calg~ $A$, it follows from Theorem 1.12 and Proposition 1.15 of \cite{boersema02} that $KO_*(A) = 0$ if and only if $KU_*(A) = 0$.  Therefore $KO_*(\B\pr) = 0$ and $\partial_i \colon KO_i(\Q\pr) \rightarrow KO_{i-1}(\K\pr)$ is an isomorphism for each $i$. Therefore,
$$\begin{array}{|c||c|c|c|c|c|c|c|c|}
\hline
i & 0 & 1 & 2 & 3 & 4 & 5 & 6 & 7  \\ \hline
KO_i(\Q\pr) & 0 & \Z & \Z_2 & \Z_2 & 0 & \Z & 0 & 0 \\ \hline
KO_i(\K\pr) &  \Z & \Z_2 & \Z_2 & 0 & \Z & 0 & 0 & 0\\ 
\hline
\end{array}$$
We will identify the generators of $KO_i^u(\K\pr)$ and, working backwards, the generators of $KO_i^u(\Q\pr)$ for all $i$.

 Let $e$ be  a rank 1 projection in $\K\pr$ and let $\iota \colon \R \rightarrow \K\pr$ be the homomorphism given by $t \mapsto t e$, which induces an isomorphism on $KO_*(-)$. The generators of $KO^u_i(\R)$ are given in Example~\ref{ex:Rgenerators}. Recall that the generator of $KO^u_0(\R)$ was identified as $[1_2]$ where $1_2 \in M_2(\R)$. Working in the unitization $\widetilde{\R}$ this corresponds to the unitary $w = {\diag}(0,2) + {\diag}(\1, -\1) \in M_2(\widetilde{\R})$, which satisfies $[\lambda(w)] = [{\diag}(\1, -\1)] = 0$ (see Remark~\ref{rem:unit}).
Then the generator of $KO_0^u(\K\pr)$ is given by 
$[\widetilde{\iota}(w)] = [{\diag}(\1, 2e - \1)]$, as shown below. The rest of the generators of $KO_i(\K\pr)$ for $i = 1,2,4$ are worked out similarly.
Note that these unitary representatives are given in terms of the \ctalg ~$(\K, \tau)$ where $\tau$ is the associated involution on $\K$.
\begin{itemize}
\item The generator of $KO_0^u(\K\pr)$ is given by $[w_0]$ where
$w_0 = {\diag}(\1, 2e - \1,  ) \in M_2(\widetilde{\K})$.
\item The generator of $KO_1^u(\K\pr)$ is given by $[w_1]$ where
$w_1 = -2e + \1 \in \widetilde{\K}$.
\item The generator of $KO_2^u(\K\pr)$ is given by $[w_2]$ where
$$w_2 = \begin{pmatrix}
		0 & i(-2e + \1) \\ i(2e - \1) & 0
\end{pmatrix} \in M_2(\widetilde{\K}) \; .$$
\item The generator of $KO_4^u(\K\pr)$ is given by $[w_4]$ where
$$w_4 = {\diag}(\1, \1, 2e - \1, 2e - \1) \in M_4(\widetilde{\K}) \; .$$
\end{itemize}
In each case, an element generating $KO_i^u(\R)$ corresponds to an element in $KO_i^u(\K\pr)$ via 
$[u] \mapsto \left[(u - I_n^{(i)}) \otimes e + I_n^{(i)} \otimes \1_{\K} \right]$.

Let $s \in \B\pr$ be the one-sided shift operator defined  by $s(e_i) = e_{i+1}$ where $\{e_i\}$ is a given basis of the underlying (real) Hilbert space. Then $s$ satisfies $s^\tau = s^*$. Let $u = \pi(s) \in \Q\pr$, which is a unitary and also satisfies $u^\tau = u^*$. We claim that that generators of $KO^u_i(\Q\pr)$ are the following.
\begin{itemize}
\item The generator of $KO_1^u(\Q\pr)$ is given by $[v_1]$ where
$v_1 = u \in M_1(\widetilde{\Q})$.
\item  The generator of $KO_2^u(\Q\pr)$ is given by $[v_2]$ where
$$v_2 = \begin{pmatrix} 0 & iu \\ -iu^* & 0  \end{pmatrix} \in M_2(\widetilde{\Q}) \; .$$
\item The generator of $KO_3^u(\Q\pr)$ is given by $[v_3]$ where
$$v_3 = \begin{pmatrix} u & 0 \\ 0 & u^*  \end{pmatrix} \in M_2(\widetilde{\Q}) \; ,$$
\item The generator of $KO_5^u(\Q\pr)$ is given by $[v_5]$ where
$v_5 = {\diag}(u,u) \in M_2(\widetilde{\Q})$.
\end{itemize}

\begin{proof}
In each case, we verify that $\partial_i([v_i]) = [w_{i-1}]$ and then the result follows since $\partial_i$ is an isomorphism.

For the first statement, we calculate $\partial_1([v_1]) = \partial_1([u])$. First lift $u$ back to the partial isometry $s \in \B\pr$. Using the formulas $s^* s = 1$ and $s s^* = 1 - e$ we have
$$B(s) = \begin{pmatrix} 2 s s^* - \1 & 2 s \sqrt{\1 - s^* s} \\
			2 s^* \sqrt{\1 - s s^*} & \1 - 2 s^* s 
         \end{pmatrix}
         = \begin{pmatrix} \1 - 2e & 0 \\ 0 & -\1 \end{pmatrix} $$
Then, noting that $Y_2^{(1)} = 1_2$, we have
$$\partial_1[u] = [ Y_2^{(1)} B(s) Y_2^{(1)} ] = [{\diag}(\1-2e, -\1)] 
  =  [{\diag}(\1, 2e-\1)] \; ,$$ 
which is the generator of $KO_0^u(\K\pr)$. 

For the second statement, first notice that $v_2$ is a self-adjoint unitary and that $v_2^\tau = -v_2$. The appropriate lift to $\widetilde{\K}$ is 
$$a = \begin{pmatrix} 0 & is \\ -is^* & 0 \end{pmatrix} \; .$$
Then 
\begin{small}
\begin{align*}
& \partial_2([v_2]) \\
		&\quad= \left[-\exp \left( \pi \begin{pmatrix} 0 & -s \\ s^* & 0 \end{pmatrix} \right) \right] \\
		&\quad= \left[- \left( \begin{pmatrix} \1 & 0 \\ 0 & \1 \end{pmatrix} 
				    + \pi \begin{pmatrix} 0 & \!\!-s \\ s^* & 0 \end{pmatrix}
				    + \frac{\pi^2}{2!} \begin{pmatrix} -\1 + e  \!\!\!\!& 0 \\ 0 &\!\! -\1 \end{pmatrix}
				    + \frac{\pi^3}{3!} \begin{pmatrix} 0 & s \\ -s^* & 0 \end{pmatrix}
				    + \cdots \right) \right] \\
		&\quad= \left[-\left( \begin{pmatrix} e & 0 \\ 0 & 0 \end{pmatrix} + 
			  \cos \left( \pi \cdot \begin{pmatrix} \1 - e & 0 \\ 0 & \1 \end{pmatrix} \right)
			+ \sin \left( \pi \cdot \begin{pmatrix} 0 & -s \\ s^* & 0 \end{pmatrix} \right) \right) \right] \\
		&\quad= \left[ \begin{pmatrix} \1 - 2e & 0 \\ 0 & \1 \end{pmatrix} \right] \\	
		&\quad= [w_1]
\end{align*}
\end{small}
which is the desired generator of $KO_1^u(\K\pr)$.

Now consider $v_3$, which is a unitary that satisfies $v_3^{\sharp \otimes \tau} = v_3$. The obvious lift is
${\diag}(s,s^*) \in {\B}$. We calculate
$$B \left( \begin{pmatrix} s & 0 \\ 0 & s^* \end{pmatrix} \right) 
        =  \begin{pmatrix} \1 - 2e & 0 & 0 & 0 \\ 0 & \1 & 0 & 0 \\ 0 & 0 & -\1 & 0 \\ 0 & 0 & 0 & -\1 + 2 e 
           \end{pmatrix} $$
and then conjugate by $Y_4^{(3)} = V_4 Q_4 W_4$ to obtain
$$B' =  \begin{pmatrix} 0 & i(\1-e) & 0 & -i e \\ -i(\1-e) & 0 & ie & 0 \\ 0 & -ie & 0 & i(\1-e)
				\\ ie & 0 & -i(\1-e) & 0
           \end{pmatrix} \; .$$
Notice that $\lambda(B') = I_2^{(2)}$ as expected. 
The class $[B']$ corresponds to the class in $KO^u_2(\R)$ given by the unitary
$$ C' = \begin{pmatrix}
	  0 & 0 & 0 & -i \\ 0 & 0 & i & 0 \\ 0 & -i & 0 & 0 \\ i & 0 & 0 & 0
   \end{pmatrix} \in M_4(\R)
 \; .$$
The Pfaffian of $C'$ distinguishes it from the trivial element represented by
$$I^{(2)}_2 = \begin{pmatrix}
	  0 & i & 0 & 0 \\  -i & 0 & 0 & 0 \\ 0 & 0 & 0 & i \\ 0 & 0 &-i & 0 
   \end{pmatrix} 
 \; .$$
Therefore, $[B'] = [w_2]$ is the non-trivial class in $KO_2^u(\K\pr)$.

Finally, the proof for $v_5$ is similar to that for $v_1$.  The lift for $v_5$ is $a = {\diag}(s, s)$ and then 
$$\partial_5([v_5]) = [B(a)] = [ {\diag}(\1 - 2e, \1 - 2e, -\1, -\1) ] = [w_4]$$
as desired. (The conjugation matrix in this case is $Y^{(4)}_4 = X_4 = 1_4$.)
\end{proof}

\end{example}

\section*{Acknowledgments}

We would like to thank Efren Ruiz, Adam S{\o}rensen, and Hermann Schulz-Baldes for various assistances that contributed to this work.
The first named author would like to express his gratitude to the University of New Mexico for its hospitality during the time this work was accomplished.


\begin{thebibliography}{99}
 
 \bibitem{altland1997nonstandard}
	{\scshape Altland, A.; Zirnbauer, M.R.}
	Nonstandard symmetry classes in mesoscopic normal-superconducting hybrid structures.
	{\em Phys. Rev. B} {\bf 55} (1997) 1142.

\bibitem{atiyah66} {\scshape Atiyah, M.F.} $K$-theory and reality. 
{\em Quart. J. Math. Oxford Ser. (2) }
{\bf 17} (1966) 367--386.

\bibitem{berenstein2012matrix} {\scshape Berenstein, D.;  Dzienkowski, E.}
  Matrix embeddings on flat $\mathbb{R}^{3}$ and the geometry of membranes.
 {\em Phys. Rev. D} {\bf 86} (2012) 086001.

\bibitem{blackadarbook} {\scshape Blackadar, B.}
    $K$-theory for operator algebras, Second Edition. Cambridge University Press, Cambridge, 1998.

\bibitem{boersema02} {\scshape Boersema, J.L.} Real \calg s, united $K$-theory, and the K\"unneth formula.
 {\em $K$-Theory} {\bf 26} (2002) 345--402. 

\bibitem{boersema04} {\scshape Boersema, J.L.} Real \calg s, united $KK$-theory, and the
              universal coefficient theorem. {\em $K$-Theory} {\bf 33} (2004) 107--149.

\bibitem{BLR} {\scshape Boersema, J.L.; Loring, T.A.; Ruiz, E.}
 Pictures of {$KK$}-theory for real \calg s  
and almost commuting matrices. {\em Banach J. Math Anal.} {\bf 10} (2016) 27--47.
		
\bibitem{BRS} {\scshape Boersema, J.L.; Ruiz, E.; Stacey, P.J.}
The classification of real purely infinite simple \calg s.
{\em Doc. Math.} {\bf 16} (2011) 619--655.

\bibitem{bousfield90} {\scshape Bousfield, A.K.}
 A classification of {$K$}-local spectra. {\em J. Pure Appl. Algebra} {\bf 66} (1990) 121--163.

\bibitem{dadarlat94} {\scshape Dadarlat, M.}
 A note on asymptotic homomorphisms.
 {\em $K$-Theory} {\bf 8} (1994) 465--582.

\bibitem{dadarlatloring94} {\scshape Dadarlat, M.; Loring, T.A.}
 {$K$}-homology, asymptotic representations, and unsuspended {$E$}-theory.
 {\em J. Funct. Anal.} {\bf 126} (1994) 367--383. 
  
\bibitem{deNittisClassificationReal} {\scshape De Nittis, G.;  Gomi, K.}
 Classification of ``Real'' {B}loch-bundles: Topological Quantum Systems of type {AI}.
{\em J. Geom. Phys.} {\bf 86} (2014) 303--338.

\bibitem{deNittisClassificationQuaternion} {\scshape De Nittis, G.;  Gomi, K.}
 Classification of ``Quaternionic'' {B}loch-bundles: Topological Quantum Systems of type {AII}.
 {\em Commun. Math. Phys.} {\bf 339} (2015) 1--55.

\bibitem{DoranTwistedKRtheory} {\scshape Doran, C.;  M{\'e}ndez-Diez, S. ; Rosenberg, J.}
 T-duality for orientifolds and twisted {KR}-theory.
 {\em Lett. Math. Phys.} {\bf 104} (2014) 1333--1364.

\bibitem{EsinGurari-Bulk-boundary} {\scshape Essin, A,M. ; Gurarie, V.}
 Bulk-boundary correspondence of topological insulators from their respective Green's functions.
{\em Phys. Rev. B} {\bf 84} (2011) 125--132.
 
\bibitem{fu2007topological} {\scshape Fu, L.; Kane, C.L.; Mele, E.J.}
 Topological insulators in three dimensions.
 {\em Phys. Rev. Lett.} {\bf 98} (2007) 106803.
 
 \bibitem{goodearlbook} {\scshape Goodearl, K. R.}
 Notes on real and complex \calg s, Shiva Mathematics Series, Vol. 5, Shiva Publishing Ltd., Nantwich,
 1982.    
 
 \bibitem{grossmann2015index} {\scshape Grossmann, J. ; Schulz-Baldes, H.}
  Index pairings in presence of symmetries with applications to topological insulators.
  {\em Comm. Math. Phys.} {\bf 343} (2015) 477--513.

 \bibitem{hastingsloring} {\scshape Hastings, M.B.; Loring, T.A.}
  Topological insulators and \calg s: theory and numerical practice.
  {\em Ann. Phys.} {\bf 326} (2011) 1699--1759.
 
\bibitem{kitaev-2009} {\scshape Kataev, A.}
 Periodic table for topological insulators and superconductors.
 AIP Conference Proceedings, {\bf {1134}}
(2009)

\bibitem{hornbook} {\scshape Horn, R.A.; Johnson, C.R.}
Matrix Analysis, Second Edition, Cambridge University Press, Cambridge, 2013. 
  
 \bibitem{kasparov80} {\scshape Kasparov, G.G.}
 The operator {$K$}-functor and extensions of \calg s.
 {\em Izv. Akad. Nauk SSSR Ser. Mat.} {\bf 44} (1980) 571--636.
 
 \bibitem{loring93} {\scshape Loring, T.A.}
 Projective \calg s.
  {\em Math. Scand.} {\bf 73} (1993) 274--280.

\bibitem{loring08} {\scshape Loring. T.A.}
A projective \calg~ related to {$K$}-theory.
 {\em J. Funct. Anal.} {\bf 254} (2008) 3079--3092.

\bibitem{loring10} {\scshape Loring, T.A.}
 \calg~ relations.
{\em Math. Scand.} {\bf 107} (2010) 43--72.

\bibitem{loring14} {\scshape Loring, T.A.}
Quantitative {$K$}-theory related to spin {C}hern numbers.
 {\em SIGMA Symmetry Integrability Geom. Methods Appl.} {\bf 10} (2014)

\bibitem{Loring15} {\scshape Loring, T.A.}
$K$-theory and pseudospectra for topological insulators.
{\em Ann. Physics} {\bf 356} (2015) 383--416.

\bibitem{loringpedersen98}
{\scshape Loring, T.A.; Pedersen, G.K.}
 Projectivity, transitivity and AF-telescopes.
 {\em Trans. Amer. Math. Soc.} {\bf 350} (1998) 4313--4339.


\bibitem{LorSorensenOrtho}
 {\scshape Loring, T.A.; S{\o}rensen, A.P.W.}
Almost commuting orthogonal matrices.
 {\em J. Math. Anal. Appl.} {\bf 420} (2014) 1051--1068.


\bibitem{loringsorensen}
{\scshape Loring, T.A.; S{\o}rensen, A.P.W.}
 Almost commuting self-adjoint matrices:  the real and self-dual cases.
{\em Rev. Math. Phys.}  {\bf 28} (2016) 1650017,
 DOI: 10.1142/S0129055X16500173.

\bibitem{loringbook}
{\scshape Loring, T.A.}
  Lifting solutions to perturbing problems in \calg , Fields Institute Monograph {\bf 8},
  American Mathematical Society, Providence, RI. 1997.

\bibitem{Mondr_Prodan_AIII_1D}
{\scshape Mondragon-Shem, I.; Hughes, T.L.; Song, J.; Prodan, E.}
Topological Criticality in the Chiral-Symmetric AIII Class at Strong Disorder.
{\em Phys. Rev. Lett.} {\bf 113} (2014) 046802.

\bibitem{pedersen98}
 {\scshape Pedersen, G.K.}
Factorization in \calg .
{\em Exposition. Math.} {\bf 16} (1998) 145--156.

\bibitem{rordambookblue}
 {\scshape R{\o}rdam, M.; Larsen, F.; Laustsen, N.}
  An introduction to {$K$}-theory for \calg , London Mathematical Society Student Texts {\bf 49},
   Cambridge University Press, Cambridge, 2000.

\bibitem{rosenbergbook}
 {\scshape Rosenberg, J.}
 Algebraic {$K$}-theory and its applications, Graduate Texts in Mathematics {\bf 147}, 
  Springer-Verlag, New York, 1994.

\bibitem{ryu2010topological}
{\scshape Ryu, S.; Schnyder, A.P.; Furusaki, A.; Ludwig, A.W.W.}
 Topological insulators and superconductors: tenfold way and dimensional hierarchy.
 {\em New J. Phys.} {\bf 12} (2010) 065010.
 
\bibitem{schroderbook} {\scshape Schr{\"o}der, H.}
{$K$}-theory for real \calg s~ and applications,
 Pitman Research Notes in Mathematics Series {\bf 290},
 Longman Scientific \& Technical, Harlow, 1993.

\bibitem{sorensenthesis} {\scshape S{\o}rensen, A.P.W.}
Semiprojectivity and the geometry of graphs.
 University of Copenhagen, 2012.

\bibitem{wiegmann58} {\scshape Wiegmann, N.A.}
 The structure of unitary and orthogonal quaternion matrices.
  {\em Illinois J. Math.} {\bf 2} (1958) 402--407.
 
 \bibitem{WittenKthDbranes}
{\scshape  Witten, E.}
 D-branes and {$K$}-theory.
 {\em J. High Energy Phys.} {bf 12} (1998) 
 
 \bibitem{wood1966banach} 
 {\scshape Wood, R.}
  Banach algebras and Bott periodicity.
{\em Topology} {\bf 4} (1966) 371--389.  
 
 
\end{thebibliography}
\end{document}